\documentclass[12pt, reqno]{amsart}

\usepackage[a-2b,mathxmp]{pdfx}[2018/12/22]

\usepackage{kpfonts}

\numberwithin{equation}{section}
\newcommand*{\LatexDef}{.}

\usepackage{amssymb}

\usepackage{amsthm}

\usepackage{amsmath}

\usepackage{dsfont} 

\usepackage{thmtools}

\usepackage{cancel}

\usepackage{verbatim}

\usepackage{lipsum} 

\newlength{\enummargin}
\usepackage[shortlabels]{enumitem}
\setlist[enumerate]{align=left,font=\upshape,leftmargin=*, widest=i]}

\usepackage{graphicx}
\usepackage{float}

\usepackage{color}

\usepackage{xspace}

\usepackage{calrsfs}

\DeclareSymbolFont{defaultmathcal}{OMS}{zplm}{m}{n}
\DeclareSymbolFontAlphabet{\mathcal}{defaultmathcal}

\DeclareSymbolFont{handwritten}{OMS}{rsfs}{m}{n}
\DeclareSymbolFontAlphabet{\handcal}{handwritten}

\usepackage[alphabetic, nobysame]{amsrefs}

\definecolor{darkgreen}{RGB}{54,124,50}
\usepackage{hyperref}
\hypersetup{
	pdfstartview={XYZ null null 1.00}, 
	pdfpagemode=UseNone, 
	colorlinks,
	breaklinks,
	linkcolor=blue,
	urlcolor=blue, 
	anchorcolor=blue,
	citecolor=darkgreen,
}

\usepackage[capitalise]{cleveref}

\usepackage{xparse} 

\usepackage{geometry}
\usepackage{microtype}

\newcommand{\straighttext}[1]{\text{\textnormal{#1}}}

\ExplSyntaxOn
\NewDocumentCommand{\RN}{m}
{
	\textup{ \int_to_Roman:n { #1 } }
}
\NewDocumentCommand{\rn}{m}
{
	\textup{ \int_to_roman:n { #1 } }
}
\ExplSyntaxOff

\newcommand{\E}{\mathbb{E}}
\newcommand{\F}{\mathbb{F}}
\newcommand{\N}{\mathbb{N}}

\newcommand{\R}{\mathbb{R}}
\newcommand{\Z}{\mathbb{Z}}

\let\0\emptyset

\newcommand{\w}{\infty}

\newcommand{\1}{\mathds{1}} 

\newcommand{\ga}{\gamma}
\newcommand\de{\delta}
\newcommand{\e}{\varepsilon}

\newcommand\la{\lambda}
\newcommand\om{\omega}
\newcommand\si{\sigma}
\newcommand{\Ga}{\Gamma}
\newcommand{\De}{\Delta}

\renewcommand{\phi}{\varphi}

\DeclareRobustCommand{\rchi}{{\mathpalette\irchi\relax}}
\newcommand{\irchi}[2]{\raisebox{\depth}{$#1\cchi$}}
\let\cchi\chi
\let\chi\rchi

\newcommand{\AC}{\mathcal{A}} 
\newcommand{\BC}{\mathcal{B}}

\newcommand{\DC}{\mathcal{D}}

\newcommand{\IC}{\mathcal{I}}

\newcommand{\MC}{\mathcal{M}}

\newcommand{\PC}{\mathcal{P}}
\newcommand{\QC}{\mathcal{Q}}

\newcommand{\SC}{\mathcal{S}}

\newcommand{\GF}{\mathfrak{G}}
\newcommand{\MF}{\mathfrak{M}}

\newcommand{\dist}{\operatorname{dist}}
\newcommand{\dom}{\operatorname{dom}}

\newcommand{\Graph}{\operatorname{Graph}}

\newcommand{\Id}{\operatorname{Id}}

\newcommand{\Pow}{\handcal{P}}

\newcommand{\proj}{\operatorname{proj}}

\newcommand{\del}{\partial}

\makeatletter
\def\moverlay{\mathpalette\mov@rlay}

\def\mov@rlay#1#2{\leavevmode\vtop{%
		\baselineskip\z@skip \lineskiplimit-\maxdimen
		\ialign{\hfil$\m@th#1##$\hfil\cr#2\crcr}}}

\newcommand{\charfusion}[3][\mathord]{
	#1{\ifx#1\mathop\vphantom{#2}\fi
		\mathpalette\mov@rlay{#2\cr#3}
	}
	\ifx#1\mathop\expandafter\displaylimits\fi}
\makeatother

\renewcommand{\le}{\leqslant}

\renewcommand{\ge}{\geqslant}

\newcommand{\symdiff}{\mathrel{\triangle}}

\newcommand{\actson}{\curvearrowright}

\newcommand{\onto}{\twoheadrightarrow}

\newcommand{\imp}{\Rightarrow}
\newcommand{\pmi}{\Leftarrow}

\newcommand{\shortiff}{\Leftrightarrow}

\newcommand{\lefttoright}{\noindent $\mathbf{\imp:}$\xspace}
\newcommand{\righttoleft}{\noindent $\mathbf{\pmi:}$\xspace}

\newcommand*{\defeq}{\mathrel{\vcenter{\baselineskip0.5ex \lineskiplimit0pt \hbox{\scriptsize.}\hbox{\scriptsize.}}}=}

\newcommand*{\defequiv}{\mathrel{\vcenter{\baselineskip0.5ex \lineskiplimit0pt
			\hbox{\scriptsize.}\hbox{\scriptsize.}}}\shortiff}

\newcommand{\bigincU}{\charfusion[\mathop]{\bigcup}{\uparrow}}

\DeclareMathSymbol{\lqm}{\mathord}{operators}{``}
\DeclareMathSymbol{\rqm}{\mathord}{operators}{`'}

\newcommand{\Abert}{Ab\'{e}rt\xspace}

\newcommand{\Folner}{F{\o}lner\xspace}

\newcommand{\set}[1]{\left\{ #1 \right\}}

\newcommand{\rest}[1]{\mathord{|_{#1}}}

\newcommand{\eqcomment}[1]{\Big[\text{#1}\Big] \hspace{6pt}}

\newcommand{\onum}[2][th]{$#2^\text{#1}$}

\theoremstyle{plain}

\newtheorem{theorem}[equation]{Theorem}
\newtheorem*{theorem*}{Theorem}

\newcommand{\namedthmlabelref}[1]{\ref{#1}\xspace}

\makeatletter
\def\@empty{}
\def\ifemptycredit#1{%
	\def\tmp{#1}%
	\ifx\tmp\@empty%
	\else%
	{~(#1)}%
	\fi%
}
\makeatother

\newenvironment{namedthm}[2][]{
	\refstepcounter{equation}
	\medskip
	\par\noindent \textbf{#2~\theequation}\ifemptycredit{#1}\textbf{.}\itshape\xspace
}{}

\newenvironment{namedthm*}[2][]{
    \medskip
	\par\noindent \textbf{#2}\ifemptycredit{#1}\textbf{.}\itshape\xspace
}{}

\crefname{prop}{Proposition}{Propositions}
\newtheorem{prop}[equation]{Proposition}
\newtheorem*{propo*}{Proposition}

\crefname{property}{Property}{Properties}

\newtheorem*{property*}{Property}

\newtheorem{lemma}[equation]{Lemma}
\newtheorem*{lemma*}{Lemma}

\crefname{claimlemma}{Claim}{Claims}
\newtheorem{claimlemma}[equation]{Claim}

\crefname{cor}{Corollary}{Corollaries}
\newtheorem{cor}[equation]{Corollary}
\newtheorem*{cor*}{Corollary}

\crefname{obs}{Observation}{Observations}
\newtheorem{obs}[equation]{Observation}
\newtheorem*{obs*}{Observation}

\crefname{obss}{Observations}{Observations}

\newtheorem{obss*}{Observations}

\crefname{fact}{Fact}{Facts}

\newtheorem*{fact*}{Fact}

\theoremstyle{definition}

\crefname{defn}{Definition}{Definitions}
\newtheorem{defn}[equation]{Definition}
\newtheorem*{defn*}{Definition}

\newenvironment{defn**}[1][]{\par\medskip\noindent \textbf{Definition\xspace#1.}\xspace}{}

\crefname{question}{Question}{Questions}

\newtheorem*{question*}{Question}

\crefname{conj}{Conjecture}{Conjectures}

\newtheorem*{conj*}{Conjecture}

\crefname{example}{Example}{Examples}
\newtheorem{example}[equation]{Example}
\newtheorem*{example*}{Example}

\crefname{examples.plain}{Examples}{Examples}
\newtheorem{examples.plain}[equation]{Examples}
\newtheorem*{examples.plain*}{Examples}

\theoremstyle{remark}

\crefname{remark}{Remark}{Remarks}
\newtheorem{remark}[equation]{Remark}
\newtheorem*{remark*}{Remark}

\newenvironment{remarklike*}[2][]{\par\medskip\noindent \textit{#2}#1\textbf{.}\rmfamily\xspace}{\smallskip}

\crefname{claim+}{Claim}{Claims}
\newtheorem{claim+}[equation]{Claim}

\crefname{claim}{Claim}{Claims}

\newtheorem*{claim*}{Claim}

\crefname{subclaim}{Subclaim}{Subclaims}

\newtheorem*{subclaim*}{Subclaim}

\newenvironment{case*}[1]{\smallskip\par\noindent \textit{Case}:~#1.\rmfamily}{}

\crefname{notation}{Notation}{Notations}

\newtheorem*{notation*}{Notation}

\newtheorem*{terminology*}{Terminology}

\crefname{convention}{Convention}{Conventions}

\newtheorem*{convention*}{Convention}
\newtheorem*{conventions*}{Conventions}

\crefname{spec}{Speculation}{Speculations}

\newtheorem*{spec*}{Speculation}

\crefname{caution}{Caution}{Cautions}

\newtheorem*{caution*}{Caution}

\crefname{hypothesis}{Hypothesis}{Hypotheses}
\newtheorem{hypothesis}[equation]{Hypothesis}
\newtheorem*{hypothesis*}{Hypothesis}

\crefname{assumption}{Assumption}{Assumptions}
\newtheorem{assumption}[equation]{Assumption}
\newtheorem*{assumption*}{Assumption}

\newcommand{\fntsz}[1][11]{\fontsize{#1}{#1}\selectfont}

\crefname{examples}{Examples}{Examples}

\newenvironment{examples*}[1][\alph*]
{
	\refstepcounter{equation}
	\medskip
	\noindent\textbf{Examples.}
	\medskip
	\begin{enumerate}[\bfseries(\theequation.#1),ref=(\theequation.#1),itemsep=5pt]
}
{
	\end{enumerate}
	\smallskip
}

\theoremstyle{remark}

\declaretheoremstyle[
spaceabove=\topsep, 
spacebelow=6pt,
headfont=\normalfont\itshape,
notefont=\normalfont, notebraces={(}{)},
bodyfont=\normalfont,
postheadspace=4pt,
qed=\mbox{\smaller[4]$\boxtimes$}
]{claimproofstyle}
\declaretheorem[name={Proof of Claim}, style=claimproofstyle, unnumbered]{pf}

\crefname{subsection}{Subsection}{Subsections}

\crefformat{footnote}{#2\footnotemark[#1]#3}

\crefrangeformat{enumi}{#3#1#4--#5#2#6}

\theoremstyle{plain}

\usepackage{mdframed}
\newmdenv[
leftmargin = 1cm,
rightmargin = 0pt,
skipabove = 8pt,
skipbelow = 3pt,
innerleftmargin = 8pt,
innertopmargin = 0pt,
innerbottommargin = 0pt,
innerrightmargin = 0pt,
linewidth = 3pt,
topline = false,
rightline = false,
bottomline = false
]{leftbar}

\definecolor{gris}{RGB}{90,90,90}

\definecolor{vert}{RGB}{7,126,26}

\definecolor{rougefonce}{RGB}{136,0,21}

\definecolor{purple}{RGB}{116,0,159}

\makeatletter
\def\@settitle{\begin{center}%
		\baselineskip14\p@\relax
		\bfseries
		\uppercasenonmath\@title
		\@title
		\ifx\@subtitle\@empty\else
		\\[1ex]\uppercasenonmath\@subtitle
		\footnotesize\mdseries\@subtitle
		\fi
	\end{center}%
}
\def\subtitle#1{\gdef\@subtitle{#1}}
\def\@subtitle{}
\makeatother

\renewcommand{\theequation}{\thesection.\arabic{equation}}

\let\subsectionstar\subsection

\makeatletter
\def\l@section{\@tocline{1}{5pt}{0pc}{}{}}
\renewcommand{\tocsection}[3]{%
	\indentlabel{\@ifnotempty{#2}{\makebox[20pt][l]{%
				\ignorespaces#1 #2.\hfill}}}\sc #3\dotfill}

\newdimen{\tocsubsecmarg}
\def\l@subsection{\@tocline{2}{3pt}{0pc}{\tocsubsecmarg}{}}
\renewcommand{\tocsubsection}[3]{%
	\indentlabel{\@ifnotempty{#2}{\makebox[30pt][l]{%
				\ignorespaces#1 #2.\hfill}}}#3\dotfill}
\makeatother
\let\oldtocsubsection=\tocsubsection
\renewcommand{\tocsubsection}[2]{\hspace{3em} \oldtocsubsection{#1}{#2}}

\setlist[itemize]{itemsep=5pt,leftmargin=1\parindent}

\let\~\widetilde
\let\^\overline

\renewcommand{\phi}{\varphi}

\newcommand{\coc}{{\mathfrak{w}}}

\newcommand{\mean}[2]{A^{#1}_{#2}}
\newcommand{\meanr}[1]{\mean{\coc}{#1}}
\newcommand{\meanrf}[1]{\meanr{#1}f}
\newcommand{\meanrFf}{\meanrf{F}}
\newcommand{\meanf}[1]{\mean{}{#1} f}

\newcommand{\Linf}[1]{\|#1\|_{_\w}}
\newcommand{\Linfbig}[1]{\big\|#1\big\|_{_\w}}
\newcommand{\Lone}[1]{\|#1\|_{_1}}

\newcommand{\Lp}[2][p]{\|#2\|_{_{#1}}}

\newcommand{\infl}[1][\phi]{\mathord{\mathrm{in}^\coc #1}}
\newcommand{\outfl}[1][\phi]{\mathord{\mathrm{out}^\coc #1}}

\newcommand{\Finw}[1]{[#1]^{<\w}}
\newcommand{\FinX}[1][X]{\Finw{#1}}
\newcommand{\FinE}[1][X]{\Finw{#1}_E}
\newcommand{\FinrE}[1][X]{[#1]_E^{\coc < \w}}
\newcommand{\FinG}[1][X]{\Finw{#1}_G}
\newcommand{\FinrG}[1][X]{[#1]_G^{\coc < \w}}
\newcommand{\FinGmod}[2][X]{[\Xmod[#1]{#2}]_{\Gmod{#2}}^{< \w}}
\newcommand{\FinrGmod}[2][X]{[\Xmod[#1]{#2}]_{\Gmod{#2}}^{\rmod{#2} < \w}}

\newcommand{\submod}[2]{#1_{\mathord{/} #2}}
\newcommand{\Xmod}[2][X]{\submod{#1}{#2}}
\newcommand{\Gmod}[2][G]{\submod{#1}{#2}}
\newcommand{\Emod}[2][E]{\submod{#1}{#2}}
\newcommand{\mumod}[2][\mu]{\submod{#1}{#2}}
\newcommand{\rmod}[2][\coc]{\submod{#1}{#2}}
\newcommand{\SCmod}[2][\SC]{\submod{#1}{#2}}

\newcommand{\p}{\alpha}

\newcommand{\Minr}[1][\coc]{\mathrm{Min}_{#1}}
\newcommand{\minr}[1][\coc]{\mathrm{min}_{#1}}
\newcommand{\Maxr}[1][\coc]{\mathrm{Max}_{#1}}
\newcommand{\maxr}[1][\coc]{\mathrm{max}_{#1}}
\newcommand{\MaxrF}[1][F]{\mathrm{Max}_{\rmod{#1}}}

\newcommand{\Cone}[1][G]{C_\coc^G}
\newcommand{\rCone}[1][G]{c_\coc^G}

\newcommand{\lift}[1]{{\hat{#1}}}

\newcommand{\bdrout}[1]{\overline{\del}_{#1}}
\newcommand{\bdrin}[1]{\underline{\del}_{#1}}

\newcommand{\cocmax}{\coc_\ast}
\newcommand{\cocmaxF}[1][F]{(\rmod{#1})_\ast}
\newcommand{\cocratio}[1][$\coc$]{#1-ratio\xspace}

\newcommand{\Mean}[2]{\AC^{#1}_{#2}}

\newcommand{\Meanrf}[1]{\Mean{\coc}{#1} f}
\newcommand{\MeanrfGrest}[1]{\Meanrf{G \rest{#1}}}
\newcommand{\MeanGf}{\Mean{}{G} f}
\newcommand{\Meanf}[1]{\Mean{}{#1} f}
\newcommand{\MeanrGf}{\Meanrf{G}}

\newcommand{\EdgsBtw}[2]{[#2]_{#1}}
\newcommand{\DirEdgsBtw}[2]{(#2)_{#1}}

\newcommand{\fvp}[1][\mu]{\mathrm{fvp}_{#1}}
\newcommand{\fep}[1]{\mathrm{fep}_{#1}}
\newcommand{\hvp}[1][\mu]{\mathrm{hvp}_{#1}}
\newcommand{\hep}[1]{\mathrm{hep}_{#1}}

\newcommand{\vdom}{\mathrm{vdom}}

\newcommand{\Sources}{\mathrm{Sources}}
\newcommand{\Sinks}{\mathrm{Sinks}}

\newcommand{\phib}{\overline{\phi}}

\newcommand{\vis}{\preccurlyeq}
\newcommand{\viscocG}{\vis^\coc_G}
\newcommand{\sees}{\succcurlyeq}

\newcommand{\dcone}[1]{\mathord{(\vis)^{#1}}}
\newcommand{\ucone}[1]{\mathord{(\vis)_{#1}}}
\newcommand{\dconecocG}[1]{\mathord{(\viscocG)^{#1}}}

\newcommand{\cocmore}{>^\ast_\coc}
\newcommand{\cocless}{<^\ast_\coc}

\usepackage[alphabetic, nobysame]{amsrefs}

\usepackage{geometry}
\usepackage{subfiles}

\title[]{Pointwise ergodic theorem for locally countable quasi-pmp graphs}


\author[]{Anush Tserunyan}

\address[Anush Tserunyan]{Mathematics and Statistics Department, McGill University, Montr\'eal, QC, Canada}

\email{anush.tserunyan@mcgill.ca}

\thanks{The author's research was partially supported by NSF Grant DMS-1501036, NSF grant DMS-1855648, and NSERC Discovery Grant RGPIN-2020-07120.}

\date{\today}

\begin{document}

\maketitle

\begin{abstract}
	We prove a pointwise ergodic theorem for quasi-probability-measure-preserving (quasi-pmp) locally countable measurable graphs, equivalently, Schreier graphs of quasi-pmp actions of countable groups. For ergodic graphs, the theorem gives an increasing sequence of Borel subgraphs with finite connected components over which the averages of any $L^1$ function converges to its expectation. This implies that every (not necessarily pmp) locally countable ergodic Borel graph on a standard probability space contains an ergodic hyperfinite subgraph. A consequence of this is that every ergodic treeable equivalence relation has an ergodic hyperfinite free factor.
	
	The pmp case of the main theorem was first proven by R. Tucker-Drob using a deep result from probability theory. Our proof is different: it is self-contained and applies more generally to quasi-pmp graphs. Among other things, it involves introducing a graph invariant concerning asymptotic averages of functions and a method of tiling a large part of the space with finite sets with prescribed properties. The non-pmp setting additionally exploits a new quasi-order called visibility to analyze the interplay between the Radon--Nikodym cocycle and the graph structure, providing a sufficient condition for hyperfiniteness.
\end{abstract}

\tableofcontents

\section{Introduction}

\subsectionstar*{Main results and applications}
We consider measurable actions $\Gamma \actson^a (X,\mu)$ of a countable group $\Gamma$ on a standard probability space\footnote{A standard Borel space $X$ (i.e. the $\sigma$-algebra of $X$ is the Borel $\sigma$-algebra of some Polish topology on $X$) equipped with a Borel probability measure.} $(X,\mu)$. To avoid pathologies coming from null sets interfering with the dynamics, we assume that the actions are \textit{quasi-pmp}\footnote{An action $\Ga \actson (X,\mu)$ of a countable group on a standard probability space is \emph{probability-measure-preserving} (\textit{pmp}) if for each $\ga \in \Ga$, $\ga_\ast \mu = \mu$.} (aka nonsingular or null-preserving), i.e. each group element maps null sets to null sets. Intuitively, this means that points in the same orbit have possibly different relative weights, and this is captured precisely by the \textit{Radon--Nikodym cocycle} of the orbit equivalence relation $E_a$ with respect to $\mu$, namely, a Borel function $\coc : E_a \to \R^+$, $(x,y) \mapsto \coc_x(y)$, such that
\begin{enumerate}[(i),leftmargin=1.1cm, labelwidth=16pt,itemsep=4pt]
    \item $\coc$ is a cocycle, i.e. $\coc_x(y) \coc_y(z) = \coc_x(z)$ for all $E_a$-related $x,y,z \in X$ (so indeed, one can think of $\coc_x(y)$ as the weight of $y$ divided by that of $x$);
    
    \item $\mu(\gamma A) = \int_A \coc_x(\gamma x) d\mu(x)$ for all $\gamma \in \Gamma$ and measurable $A \subseteq X$.
\end{enumerate}
In particular, $\mu$ is invariant under the action if and only if $\coc \equiv 1$. The existence and a.e. uniqueness of this cocycle is proven in \cite{Kechris-Miller}*{Section 8}.

Seeking to prove a pointwise ergodic theorem in this generality, for all groups and quasi-pmp actions at once, we have to modify the classical form of pointwise ergodic theorems. Indeed, although the natural analogue of the classical pointwise ergodic theorem holds for quasi-pmp actions of $\Z^d$ \cites{Dowker:nonsingular-ergodic-N,Feldman:ratio_ergodic}, it fails for some quasi-pmp action of $\bigoplus_{n \in \N} \Z$ along any sequence of finite subsets of $\bigoplus_{n \in \N} \Z$ \cite{Hochman:ratio_ergodic}*{Theorems 1.2 and 1.3}. Even for pmp actions, although it holds for all amenable groups along tempered \Folner sequences \cite{Lindenstrauss:ptwise_amenable}, it fails for the free group $\F_2$ on $2$ generators along spheres or balls \cite{Tao:failure-ergodic}. Of course, there are other versions of the pointwise ergodic theorem for pmp actions (e.g. \cites{Bufetov:balls,Nevo-Fujiwara,Ts-Zomback:backward-ergodic}), but they too are specific to the group, while we would like a theorem for all groups and actions at once. So what we do is abandon the group and the action, and look at the induced Schreier graph instead (defined below).

By a \textit{graph} on $X$, we mean a set of edges between the points in $X$, i.e. a symmetric subset of $X^2$; in particular, a \textit{locally countable Borel graph} $G$ on $X$ is a symmetric Borel subset of $X^2$ whose each fiber is countable (i.e. each vertex has countably-many neighbors). We denote by $E_G$ the $G$-connectedness equivalence relation, so for $x \in X$, $[x]_{E_G}$ is the $G$-connected component of $x$. When $X$ is equipped with a Borel probability measure $\mu$, we say that $G$ is \textit{quasi-pmp} if each Borel bijection $\gamma : X \to X$ with $\Graph(\gamma) \subseteq E_G$ maps $\mu$-null sets to $\mu$-null sets.

Going back to the action $\Gamma \actson^a (X,\mu)$, for a symmetric generating set $S$ of $\Gamma$, the \textit{Schreier graph} $G_S \subseteq X^2$ of this action with respect to $S$ is defined by 
\[
x G_S y \defequiv \si x = y \text{ for some } \si \in S.
\]
Assuming without loss of generality that the action $a$ is Borel (measurable transformations are Borel mod null), $G_S$ is a locally countable Borel graph. The Feldman--Moore theorem \cite{Feldman-Moore} implies that in fact \textit{every} locally countable Borel graph $G$ on a standard Borel space $X$ arises in this fashion. Furthermore, the Schreier graph $G_S$ is quasi-pmp if and only if the action $a$ is quasi-pmp. Thus, although our main result (\cref{ptwise_ergodic_for_graphs}) is stated for an arbitrary locally countable quasi-pmp Borel graph $G$, one can safely think of $G$ as a Schreier graph of some quasi-pmp action of a countable group.

\begin{theorem}[Ergodic theorem for quasi-pmp graphs]\label{ptwise_ergodic_for_graphs}
	Let $G$ be a locally countable quasi-pmp Borel graph on a standard probability space $(X,\mu)$ and let $\coc : E_G \to \R^+$ be the Radon--Nikodym cocycle of $E_G$ with respect to $\mu$. There is an increasing sequence $(G_n)$ of component-finite Borel subgraphs of $G$ (typically not adding up to $G$) such that for any $p \ge 1$ and $f \in L^p(X,\mu)$,
	\[
	\lim_{n \to \w} \frac{1}{\coc_x([x]_{E_{G_n}})} \sum_{y \in [x]_{E_{G_n}}} f(y) \coc_x(y) = \E(f | \BC_{E_G})(x) \hspace{8pt}\text{a.e. and in $L^p$},
	\]
	where $\coc_x([x]_{E_{G_n}}) \defeq \sum_{y \in [x]_{E_{G_n}}} \coc_x(y)$ and $\E(f | \BC_{E_G})$ is the conditional expectation of $f$ with respect to the $\sigma$-algebra $\BC_{E_G}$ of $E_G$-invariant Borel sets.
\end{theorem}

This generalizes to quasi-pmp graphs the unpublished theorem of R. Tucker-Drob for pmp graphs, proven by completely different techniques. Indeed, while our proof is descriptive\allowbreak-set\allowbreak-theoretic and self\allowbreak-contained, Tucker-Drob's proof is based on a deep result in probability theory: the indistinguishability of trees in the Wired Uniform Spanning Subforest \cite{Hutchcroft-Nachmias:indistinguishability}*{Theorem 1.1}; it also uses Wilson's algorithm rooted at infinity as in \cite{Gaboriau-Lyons}*{Proposition 9}, as well as an analogue for graphs of the \Abert--Weiss theorem \cite{Abert-Weiss}*{Theorem 1} derived by Tucker-Drob from \cite{Hatami-Lovasz-Szegedy}*{Lemmas 7.9 and 7.10}). The mentioned results are all for the pmp case, so generalizing Tucker-Drob's proof to the quasi-pmp setting would require generalizing these results as well.

\begin{remark}\label{remark:ergodic_for_graphs}
Points to note about \cref{ptwise_ergodic_for_graphs}:
    \begin{enumerate}[(a)]
        \item\label{item:easy_for_eq-rel} When $G$ is an equivalence relation, i.e. the whole orbit equivalence relation of a group action, the conclusion of \cref{ptwise_ergodic_for_graphs} has been known and is not very hard to prove. Indeed, the pmp case is explicitly stated and proven in \cite{Kechris:global}*{Theorem 3.5}, and the general quasi-pmp case can be easily extracted from earlier works, namely, by putting together \cite{Schmidt:Cocycles on ergodic transformation groups}*{Theorem 8.22} and the Hurewicz ergodic theorem. 
        
        From the perspective of measured group theory, the particular interest in proving this theorem for graphs, equivalently, Schreier graphs of group actions, is to at least have some involvement of the group itself: indeed, our sampling windows (the $G_n$-connected components) correspond to finite connected subsets of the Cayley graph of the group. From the descriptive-set-theoretic view point, when $G$ is ergodic, the increasing union $\bigincU_n G_n$ is an ergodic hyperfinite subgraph (see \cref{ergodic_hyperfin_subgraph} below). The existence of such a subgraph is a powerful tool, which has been sought after because of its immediate applications, e.g. \cref{free_factor} below.
        
        \item\label{item:easy_for_hyperfin} When $G$ is $\mu$-hyperfinite\footnote{$G$ is \textit{hyperfinite} if it is an increasing union of component-finite Borel graphs. \textit{$\mu$-hyperfinite} just means hyperfinite off of a $\mu$-null set.} (equivalently, $E_G$ is $\mu$-amenable, by the Connes--Feldman--Weiss theorem \cite{Connes-Feldman-Weiss}), the result is again not very hard. Indeed, discarding a null set, $G$ is an increasing union of component-finite Borel graphs $G_n$ and the conclusion of \cref{ptwise_ergodic_for_graphs} follows from a (much simpler) pointwise ergodic theorem for hyperfinite equivalence relations. Versions of the pmp case of this theorem have appeared in the literature, for example, in \cite{Bowen-Nevo:amenable_eq_rel_ergodic_gp_actions} and in \cite{Miller-Ts:erg_hyp_dec}*{Theorem 7.3}, and we state it below as \cref{hyperfin_averages} for the general quasi-pmp case.
    \end{enumerate}
    
    \noindent Thus, \cref{ptwise_ergodic_for_graphs} is most valuable for graphs (as opposed to equivalence relations), which are not $\mu$-hyperfinite/$\mu$-amenable.
\end{remark}

\cref{ptwise_ergodic_for_graphs} immediately implies what was the author's main goal:

\begin{theorem}[Ergodic hyperfinite subgraph]\label{ergodic_hyperfin_subgraph}
	Every ergodic locally countable Borel graph $G$ on a standard probability space $(X,\mu)$ admits an ergodic hyperfinite Borel subgraph $H \subseteq G$.
\end{theorem}

This immediately implies a positive answer to a question of L. Bowen, asked for pmp equivalence relations:

\begin{cor}\label{free_factor}
	Every ergodic treeable\footnote{\label{ftnt:def_treeable+free-factor}See \cite{Kechris-Miller}*{Sections 19 and 27} for the definitions of \textit{treeable} and \textit{free factor}.} countable Borel equivalence relation admits an ergodic hyperfinite free factor\cref{ftnt:def_treeable+free-factor}.
\end{cor}

Again, \cref{ergodic_hyperfin_subgraph,free_factor} generalize the corresponding unpublished results of R. Tucker-Drob in the pmp setting.

For an ergodic pmp graph $G$, the existence of an ergodic hyperfinite subgraph (i.e. Tucker-Drob's theorem) was initially also used in \cite{Miller-Ts:erg_hyp_dec}, although later the authors included a self-contained proof of a weaker statement that was sufficient for their purpose.

Lastly, the fact that \cref{ptwise_ergodic_for_graphs} holds for quasi-pmp  (and not just pmp) graphs, implies a ratio ergodic theorem, whose more general version without the ergodicity assumption is given in \cref{ratio_ergodic_for_graphs}.

\begin{theorem}[Ratio ergodic theorem for quasi-mp graphs]\label{ratio_ergodic_for_ergodic_graphs}
Let $G$ be a locally countable quasi-mp\footnote{This is the same as quasi-pmp, but the measure is not assumed to be finite.} ergodic Borel graph on a $\sigma$-finite standard measure space $(X,\mu)$ and let $\coc : E_G \to \R^+$ be the Radon--Nikodym cocycle of $E_G$ with respect to $\mu$. There is an increasing sequence $(G_n)$ of component-finite Borel subgraphs of $G$ such that for any $f,g \in L^1(X,\mu)$ with $g > 0$,
	\[
	\lim_{n \to \w} \frac{\sum_{y \in [x]_{E_{G_n}}} f(y) \coc_x(y)}{\sum_{y \in [x]_{E_{G_n}}} g(y) \coc_x(y)} 
	= 
	\frac{\int_X f d\mu}{\int_X g d\mu} \hspace{8pt}\text{a.e.}
	\]
\end{theorem}

As mentioned above, \cref{ergodic_hyperfin_subgraph,ratio_ergodic_for_ergodic_graphs} are derived from \cref{ptwise_ergodic_for_graphs}, which itself is derived from the following lemma by a diagonalization argument. (See \cref{sec:equiv_of_thms} for proofs of all these implications.)

\begin{namedthm}{Main Lemma}\label{core_lemma}
Let $G$ be a locally countable $\mu$-nowhere hyperfinite\footnote{\label{ftnt:nowhere-hyperfinite}This means that there is no $E_G$-invariant Borel set of positive $\mu$-measure on which $G$ is hyperfinite. Equivalently \cite{JKL}*{1.3(vi)}, there is no set of positive measure on which $E_G$ is hyperfinite.} quasi-pmp Borel graph on a standard probability space $(X,\mu)$ and let $\coc : E_G \to \R^+$ be the Radon--Nikodym cocycle of $E_G$ with respect to $\mu$. For any $f \in L^\infty(X,\mu)$ and $\e > 0$, there is a component-finite Borel subgraph $H \subseteq G$ such that for all $x$ in a set of measure $\ge 1-\e$, the average $\frac{1}{\coc_x([x]_{E_H})} \sum_{y \in [x]_{E_H}} f(y) \coc_x(y)$ differs from $\E(f | \BC_{E_G})(x)$ by at most $\e$.
\end{namedthm}

The proof of this is the main content of the paper and we give a sketch of it below.

\subsectionstar*{Auxiliary results}

Here, we highlight some tools we obtained to prove Main Lemma \namedthmlabelref{core_lemma} as they are interesting and may be useful elsewhere. Below, let $G$ be a locally countable Borel graph on a standard Borel space $X$.

\subsubsection*{Cuts and hyperfiniteness}

Call a set $V \subseteq X$ a \emph{hyperfinitizing vertex-cut for $G$} if $G \cap (X \setminus V)^2$ is hyperfinite. For a Borel probability measure $\mu$, put
	\[
	\hvp(G) 
	\defeq
	\inf \set{\mu(C) : C \subseteq X \text{ is a Borel hyperfinitizing vertex-cut for $G$}}
	\]
and call it the \emph{hyperfinitizing vertex-price} of $G$ (with respect to $\mu$). We also define the analogous notions for edge-cuts. This was already done in \cite{Miller-Ts:erg_hyp_dec}*{Section 9} as well as earlier in \cite{Elek:combinatorial_cost} in a slightly different context. The following is a useful and easily applicable way of exploiting the nonhyperfiniteness of a measurable graph and variations of it have appeared in the aforementioned two papers:

\begin{prop}\label{intro:char_of_hyperfin_via_price}
	A locally countable Borel graph $G$ is $\mu$-hyperfinite if and only if $\hvp(G) = 0$.
\end{prop}

When $G$ is locally finite, this proposition is merely an observation based on the Borel--Cantelli lemma. However, for locally countable graphs, the proof (still very easy) goes through the analogous statement for edge-cuts and this analogous statement immediately implies the Dye--Krieger theorem (\cref{Dye-Krieger}). See \cref{char_of_hyperfin_via_price} for the full version of \cref{intro:char_of_hyperfin_via_price}.

\cref{intro:char_of_hyperfin_via_price} is used to give a lower bound on the measure of a set based on the graph structure: for a $\mu$-nonhyperfinite graph $G$, if a set is a hyperfinitizing vertex-cut for $G$ then its measure is at least $\hvp(G) > 0$.

\subsubsection*{Approximately saturated and packed tilings}

Given a countable Borel equivalence relation $E$ on a standard Borel space $X$ and a Borel collection $\SC$ of finite $E$-related subsets\footnote{A set is called \textit{$E$-related} if it is contained in one $E$-class.}, one often needs a Borel tiling $\PC$ of a large part of $X$ with tiles from $\SC$. By \cite{Kechris-Miller}*{Lemma 7.3}, maximal such tilings exist, but for our purposes here and for those in \cite{Miller-Ts:erg_hyp_dec}, a stronger notion of maximality is needed: each tile in $\PC$ should be maximally big (i.e. it cannot be combined with some untiled points to form a tile from $\SC$) and tiles in $\PC$ cannot be combined together with proportionally-many untiled points to form a tile from $\SC$. A tiling $\PC$ with the first property is called \textit{saturated}, and with the second property, \textit{$p$-packed}, where $p \in \R^+$ is the proportion parameter.

It was proven in \cite{Miller-Ts:erg_hyp_dec}*{Subsection 4.D} that saturated and packed Borel tilings exist off of an $E$-compressible set and hence, off a null set for any $E$-invariant probability measure. Here, in \cref{sec:packing_and_saturation}, we generalize this to the quasi-pmp setting, i.e. in the presence of a Borel cocycle $\coc : E \to \R^+$. It is not true that saturated tilings exist off of a null set if the measure is not invariant: the fact that there are points of arbitrarily small relative $\coc$-weight is an issue. However, we define an approximate notion of saturation in \cref{defn:saturated} and prove existence modulo null in \cref{existence-of-saturated}. The proof of this theorem uses the existence of Borel label-maximizing maximal matchings in Borel bipartite graphs with edges labeled by positive reals, proven in \cref{maximal_labeled_matchings}.

Approximately saturated and packed (sequences of) tilings are related to and may be useful in the so-called \textit{toast} constructions (see, e.g., \cite{Grebik-Rozhon:local_prob_grids}).

\subsubsection*{Cocycle-visibility in graphs and hyperfiniteness}

Let $G$ be a locally countable Borel graph on $X$ and let $\coc : E_G \to \R^+$ be a Borel cocycle, i.e. $\coc_x(y) \coc_y(z) = \coc_x(z)$ for all $E_G$-related $x,y,z \in X$. A \emph{$(G,\coc)$-visible neighborhood} of $x \in X$ is any $G$-connected set $V \ni x$ such that $\coc_x(v) \le 1$ (the weight of $x$ is at least as much as that of $v$) for each $v \in V$. This induces a quasi-order on $X$: $x \sees y$ if $y$ is in a visible neighborhood of $x$. We say that $G$ has \textit{finite $\coc$-visibility} if for each $x \in X$, the downward cone $(\vis)^x \defeq \set{y \in X : y \vis x}$ is $\coc$-finite, i.e. $\sum_{y \vis x} \coc_x(y) < \infty$. This provides a sufficient condition for Borel hyperfiniteness:

\begin{theorem}\label{finite_visibility=>hyperfinite}
	If $G$ has finite $\coc$-visibility, then it is Borel hyperfinite.
\end{theorem}

We use this to argue that a set $D \subseteq X$ is large: if removing $D$ from a $\mu$-nowhere hyperfinite graph $G$ results in a graph with finite visibility, then $D$ is a hyperfinitizing vertex-cut, whence $\mu(D) \ge \hvp(G) > 0$.

\subsubsection*{\cocratio[Cocycle] and tiling with arbitrarily large sets}\label{subsubsec:intro:coc-ratio}

Let $E$ be a quasi-pmp Borel equivalence relation on a standard probability space $(X,\mu)$ and let $\coc : E \to \R^+$ be the Radon--Nikodym cocycle of $E$ with respect to $\mu$. We often need to find a $\mu$-nowhere smooth hyperfinite subequivalence relation $F \subseteq E$ (see \cref{subsec:eq-rel} for definitions) with some additional properties. For $F$ to be $\mu$-nowhere smooth, a.e. $F$-class $[x]_F$ has to be \textit{$\coc$-infinite}, i.e. $\sum_{y \in [x]_F} \coc_x(y) = \infty$. In the pmp case, i.e. when $\coc \equiv 1$, we obtain $F$ as an increasing union of finite Borel equivalence relations $F_n$. However, when $\coc$ is arbitrarily, it could be that although the cardinality of the $F_n$-classes grows with $n$, they union up to a $\coc$-finite set. Therefore, cardinality of the $F_n$-classes is not the right measurement to control.

The solution is to look at what we call the \textit{\cocratio} of an $E$-related finite set $U \subseteq X$, defined by $\cocmax(U) \defeq \min_{x \in U} \coc_x(U)$, where $\coc_x(U) \defeq \sum_{u \in U} \coc_x(u)$. Indeed, if the $U_n$ are increasing $E$-related finite sets with $\cocmax(U_n) \to \infty$, then $\bigincU_n U_n$ is $\coc$-infinite.

To build desired $F_n$, we need to tile most of the space $X$ with finite tiles of large \cocratio. This is surprisingly challenging because unlike cardinality, \cocratio is typically not monotone (under subsets). We build such tilings in \cref{tiling_with_large_visibility}.

\subsectionstar*{Sketch of proof of Main Lemma \namedthmlabelref{core_lemma}}\label{subsec:proof-sketch}

Let $G$, $(X,\mu)$, $\coc$, and $f$ be as in \namedthmlabelref{core_lemma}. To simplify notation, we assume that $G$ is ergodic.

\subsubsection*{The invariant case} 

Here, we sketch the proof for a pmp $G$, i.e. $\coc \equiv 1$. We begin by establishing a connection (\cref{local-global_bridge}) between the global average (integral) of $f$ and the local (finite) averages of $f$ around a point: for any finite Borel equivalence relation $F$ on $X$,
\begin{equation}\label{eq:intro:bridge}
\int_X f d\mu = \int_X \meanf{F} d \mu,
\end{equation}
where, for each $x \in X$, $\meanf{F}(x)$ is the average of $f$ over $[x]_F$.

Next, for each $x \in X$, we define the set $\MeanGf(x)$ of \emph{$G$-asymptotic averages at $x$}, namely, the set of all reals $r$ that can be approximated arbitrarily well by the averages of $f$ over finite $G$-connected sets $V \ni x$ of arbitrarily large cardinality. It is easy to see that the map $x \mapsto \MeanGf(x)$ is $E_G$-invariant, hence constant a.e. by ergodicity. Moreover, this set, denoted by $\MeanGf$, is a closed interval.

We then show that for each $\delta > 0$, one can construct a finite $G$-connected Borel equivalence relation $F_\delta$ such that for a.e. $x \in X$, the average of $f$ over $[x]_{F_\delta}$ is at most $\delta$-away from $\MeanGf$. This implies, via \labelcref{eq:intro:bridge}, that $\MeanGf$ contains the global average $\int_X f d\mu$. Thus, if the set $\MeanGf$ was just a small interval around $\int_X f d\mu$ of size less than $\frac{\e}{2}$, taking $F \defeq F_{\frac{\e}{2}}$ would satisfy the conclusion of Main Lemma \namedthmlabelref{core_lemma}.

Even if $\MeanGf$ is initially a large interval, maybe quotienting out by some finite $G$-connected Borel equivalence relation shrinks it, in which case we would also be done. Thus, we assume that there is a $\delta > 0$ such that the set $\Meanf{\Gmod{R}}$ of asymptotic averages in the quotient graph $\Gmod{R}$ is not contained in 
\[
\^{I_\delta} \defeq [\int_X f d\mu - \delta,\int_X f d\mu + \delta],
\]
for every finite $G$-connected Borel equivalence relation $R$. In fact, by an additional argument, we may assume that $\Meanf{\Gmod{R}}$ spills over both sides of $\^{I_\delta}$. This assumption allows us to tile a significant part of the space by $G$-connected finite tiles whose $f$-averages are in $\^{I_\delta}$ for arbitrarily small $\delta > 0$. It is here that packed tilings come into play: the packedness condition ensures that only finite $G$-connected components are left after removing the union $D$ of all the tiles from $X$. In other words, the domain $D$ of of each packed tiling is a finitizing vertex-cut. The $\mu$-nowhere hyperfiniteness of $G$ gives a lower bound $\lambda > 0$ for the measure of all finitizing vertex-cuts, hence domains of packed tilings. This allows us to eventually cover most of $X$ by an iterative coherent construction of saturated and packed tilings, whose tiles become larger and larger and more and more packed. After sufficiently many iterations, the resulting tiling is such that the induced subgraphs on each tile together form a subgraph $H \subseteq G$ satisfying the conclusion of Main Lemma \namedthmlabelref{core_lemma}.

\subsubsection*{The quasi-invariant case}

We only mention what changes one has to make in the above (pmp) argument to make it work in the quasi-pmp case.

To show that a set is null in the pmp setting, one often proves that it is compressible. The notion of compressibility was generalized to the quasi-invariant setting by Benjamin Miller in \cite{Miller:meas_with_cocycle_II}, and we present a rephrasing of this in \cref{sec:flows} in terms of mass transport.

As mentioned above, saturated tilings do not exist modulo null in the quasi-pmp setting, so we use approximately saturated tilings instead, proving their existence in \cref{existence-of-saturated}.

For points $x,y$ in the same $E_G$-class, the fact that $\coc_x(y)$ can be arbitrarily large destroys that convexity of the set $\MeanGf(x)$ of $G$-asymptotic averages, while our construction above crucially relies on this property. To fix this, we introduce the notion of $(G,\coc)$-visibility (developed in \cref{sec:cocycle_visibility}) and take asymptotic averages at a point $x$ only within the part of the graph that is visible to some point $y$ that sees $x$.

Another important difference is in controlling the $\coc$-weight of increasing unions $\bigincU_n U_n$ of finite sets. Instead of ensuring that $|U_n|\to \infty$, we have to ensure that $\cocmax(U_n) \to \infty$, which is much harder since $\cocmax$ is not monotone (under subsets).

\subsectionstar*{Organization}

\cref{sec:prelims} establishes notation and terminology that are globally used in the paper. In \cref{sec:hyperfin_averages}, we discuss finite and hyperfinite averages, in particular, stating the pointwise ergodic theorem for hyperfinite equivalence relations mentioned in \cref{remark:ergodic_for_graphs}\labelcref{item:easy_for_hyperfin}. \cref{sec:equiv_of_thms} provides proofs of \cref{ptwise_ergodic_for_graphs,ergodic_hyperfin_subgraph,ratio_ergodic_for_ergodic_graphs} assuming Main Lemma \namedthmlabelref{core_lemma}. In \cref{sec:flows}, we discuss mass transport along a cocycled equivalence relation, introduce the notion of deficiency for sets as a generalization of compressibility, and provide a lemma for building Borel transport functions. \cref{sec:finitizing_cuts} discusses cuts in a graph and their connection with hyperfiniteness. \cref{sec:packing_and_saturation} introduces saturated and packed tilings with respect to a cocycle and proves their existence. \cref{sec:cocycle_visibility} discusses the notion of cocycle-visibility in a graph, provides a sufficient condition for hyperfiniteness (\cref{finite_visibility=>hyperfinite}), and proves the lemma on tiling the space with sets of large \cocratio[cocycle] (\cref{tiling_with_large_visibility}). In \cref{sec:asymptotic_averages}, we introduce the set of asymptotic averages for a graph, whose role is instrumental for the proof of Main Lemma \namedthmlabelref{core_lemma}; we then establish an important tiling lemma for the set of asymptotic averages. Finally, \cref{sec:proof} is the proof of Main Lemma \namedthmlabelref{core_lemma}.

\subsectionstar*{Acknowledgments}

The author thanks

\begin{itemize}[-,itemsep=3pt]
    \item Benjamin Miller for getting her into this topic and way of thinking.
    
    \item Robin Tucker-Drob for sharing his proof of the pmp case of \cref{ergodic_hyperfin_subgraph} and for insightful conversations.
    
    \item The two anonymous referees for providing incredibly helpful and detailed reports, which, among other things, provided a more streamlined and insightful argument for \cref{finite_visibility=>hyperfinite} and pinpointed an error in the use of saturated tilings.
    
    \item Peter Burton for providing a number of references, suggesting that the reduction from $L^1$ to $L^\w$ should be explained, and asking about and verifying \cref{ratio_ergodic_for_ergodic_graphs}.
    
    \item Anton Bernshteyn for pointing out an error in the derivation of \cref{ptwise_ergodic_for_graphs} from Main Lemma \namedthmlabelref{core_lemma}.
    
    \item Ran Tao for noticing an oversight in the definition of packed tilings, and for suggesting a more straightforward proof of \cref{finite-flow-definer}.
    
    \item Lewis Bowen for useful remarks and references.
    
    \item Benjamin Weiss for pinpointing \cite{Schmidt:Cocycles on ergodic transformation groups}*{Theorem 8.22} and prompting \cref{remark:ergodic_for_graphs}\labelcref{item:easy_for_eq-rel}.
    
    \item Clinton Conley, Alexander Kechris, Andrew Marks, and Jenna Zomback for helpful conversations and suggestions.
\end{itemize}

\section{Preliminaries}\label{sec:prelims}

Our set $\N$ of natural numbers includes $0$. For reals $a,b \in \R$, we write $a \approx_\e b$ to mean that $|a - b| \le \e$. For $i=1,2$, $\proj_i : X_1 \times X_2 \to X_i$ is defined by $(x_1,x_2) \mapsto x_i$. For a set $X$, let $\Pow(X)$ denote the powerset of $X$, and let $\Id_X$ denote the identity (equality) equivalence relation on $X$.

Throughout, let $X$ be a standard Borel space. For a Borel measure $\mu$ on $X$ and $\e>0$, we call a set $X' \subseteq X$ \textit{$\mu$-co-$\e$} if it is measurable and $\mu(X \setminus X') \le \e$.

\subsection{Equivalence relations}\label{subsec:eq-rel}

Let $E$ denote a \textit{countable Borel} equivalence relation on $X$, that is: $E$ is a Borel subset of $X^2$ and each $E$-class is countable. We refer to \cite{JKL} and \cite{Kechris-Miller} for the general theory of countable equivalence relations. 

For $x \in X$ and $A \subseteq X$, we write $[x]_E$ to mean the $E$-class of $X$ and $[A]_E$ to mean the \textit{$E$-saturation} of $A$, i.e. $\bigcup_{x \in A} [x]_G$. We say that a set $A \subseteq X$ is \emph{$E$-related} if it is contained in a single $E$-class; similarly, we say that points $x_0,x_1,...,x_n \in X$ are \emph{$E$-related} if $\set{x_0,x_1,...,x_n}$ is $E$-related. We denote by $\FinE$ the standard Borel space of finite nonempty $E$-related sets.

Let $\mu$ be a Borel measure on $X$. We say that $E$ is
\begin{itemize}
    \item \textit{smooth} if for some/any Polish space $Y$, there is a Borel map $f : X \to Y$ such that for all $x,y \in X$, $x E y \iff f(x)=f(y)$. In fact, one can take $Y \defeq X$ and have in addition that $x E f(x)$, so $f$ is a \textit{Borel selector} for $E$.
    
    \item \textit{finite} if each $E$-class is finite. Note that any finite Borel equivalence relation is smooth because, by Luzin--Novikov uniformization \cite{bible}*{Theorem 18.10}, $x \mapsto \min_< [x]_E$ is a Borel selector for $E$, where $<$ is some a priori fixed Borel linear order on $X$.
    
    \item (\textit{Borel}) \textit{hyperfinite} if $E$ is an increasing union of finite Borel equivalence relations.
    
    \item \textit{$\mu$-hyperfinite} if it is hyperfinite on a conull set; this set can be chosen to be Borel and $E$-invariant, by \cite{JKL}*{1.3(vi)}.
    
    \item \textit{$\mu$-nowhere hyperfinite} if there is no set of positive measure on which $E$ is hyperfinite; again by \cite{JKL}*{1.3(vi)}, this is equivalent to the inexistence of an $E$-invariant Borel set of positive $\mu$-measure on which $E$ is hyperfinite.
    
    \item \textit{measure-preserving} (\emph{mp}) or that \emph{$\mu$ is $E$-invariant} if $\gamma_\ast\mu = \mu$ for every Borel \textit{automorphism} $\gamma$ of $E$ (i.e. a Borel bijection $X \to X$ mapping every point to an $E$-equivalent point).
    
    \item \emph{quasi-measure-preserving} (\emph{quasi-mp}) or that $\mu$ is \emph{$E$-quasi-invariant} if $\gamma_\ast\mu \sim \mu$ for every Borel automorphism $\gamma$ of $E$.
    
    \item \emph{pmp} (resp. \emph{quasi-pmp}) if it is mp (resp. quasi-mp) and $\mu$ is a probability measure.
\end{itemize}

\subsection{Cocycles}\label{subsec:cocycles}

For a countable Borel equivalence relation $E$ on $X$, a \emph{cocycle on $E$} is a map $\coc: E \to \R^+$, $(x,y) \mapsto \coc_x(y)$ satisfying the \emph{cocycle identity}: $\coc_x(y) \cdot \coc_y(z) = \coc_x(z)$ for all $E$-related $x,y,z \in X$. A Borel measure $\mu$ on $X$ is called \emph{$\coc$-invariant} if for every Borel set $B \subseteq X$ and a Borel automorphism $\gamma$ of $E$,
\[
\mu(\gamma B) = \int_B \coc_x(\gamma x) \, d\mu(x).
\]
By \cite{Kechris-Miller}*{Section 8}, every $E$-quasi-invariant probability measure $\mu$ is $\coc$-invariant for some Borel cocycle $\coc$ on $E$. Such a cocycle is clearly unique $\mu$-a.e., and it is called the \textit{Radon--Nikodym cocycle of $E$ with respect to $\mu$}. 

For a Borel cocycle $\coc : E \to \R^+$ and $x \in X$, the function $\coc_x : [x]_E \to \R^+$ defines a measure on $Y \defeq [x]_E$ by $\coc_x(A) \defeq \sum_{y \in A} \coc_x(y)$ for each $A \subseteq Y$. Thus, it makes sense to write $\int_Y f d \coc_x$ for any absolutely summable $f : Y \to \R$.

\subsubsection*{Independence of the base point}

As mentioned in the introduction, we think of $\coc_x(y)$ as the weight of $y$ divided by that of $x$; indeed, in expressions like $\frac{\coc_x(y)}{\coc_x(z)}$ and statements like ``$\lim_{n \to \infty} \coc_x(x_n) \to \infty$'' the particular choice of the base point $x$ does not matter due to the cocycle identity. We call such expressions and statements \textit{$\coc$-homogeneous} and omit writing the subscript $x$ from them. For example, for an $E$-class $C$, $A \subseteq C$, and $f : C \to \R$, we may write $\frac{2 \cdot \coc(x) + \int_C f d \coc}{\coc(A)}$ and ``$\coc(A)$ is finite''. We also say that \textit{all $\coc$-large enough} $U \subseteq C$ satisfy some property $\PC$ if for any/some base point $x \in C$, there is $L_x > 0$ such that all $U \subseteq C$ with $\coc_x(U) \ge L_x$ satisfy $\PC$. \textit{Arbitrarily $\coc$-large} is defined analogously.

\subsubsection*{Maximum and minimum}

For a subset $A$ of an $E$-class, put
\[
\Minr A \defeq \set{x \in A : \forall y \in A \; \coc(x) \le \coc(y)}.
\]
If $\Minr A \ne \0$, for any $x \in [A]_E$, we put $\minr[\coc_x] A \defeq \min_{y \in A} \coc_x(y)$, otherwise, $\minr[\coc_x] A \defeq 0$. We also analogously define $\Maxr A$ and $\maxr[\coc_x] A$, and we omit the base point $x$ in $\coc$-homogeneous expressions. 

\subsubsection*{The space $\FinrE$ of $\coc$-finite sets}

An $E$-related set $A$ is said to be \emph{$\coc$-finite} if $\coc(A)$ is finite, otherwise, it is \emph{$\coc$-infinite}. We denote by $\FinrE$ the collection of all nonempty $E$-related $\coc$-finite sets. We say that $E$ is \textit{$\coc$-finite}, if each $E$-class is $\coc$-finite; otherwise, we say that it is \textit{$\coc$-infinite} (aka \textit{$\coc$-aperiodic}).

We show that $\FinrE$ is also a standard Borel space, which can be naturally viewed as a Borel subset of $X^{< \N} \cup X^\N$. Throughout, we fix a Borel linear order $<_X$ on $X$ and define a partial order $<_\coc$ (linear on every $E$-class) as follows: for any $x,y \in X$
\[
x <_\coc y \defequiv x E y \text{ and } \big(\coc(x) < \coc(y) \text{ or } (\coc(x) = \coc(y) \text{ and } x <_X y)\big).
\]

\begin{obs}\label{Maxr_nonempty_finite}
	For each $A \in \FinrE$ and $a_0 \in A$, the set $\set{a \in A : a >_\coc a_0}$ is finite. In other words, the relation $>_\coc$ on $A \in \FinrE$ is a well-order of type $\le \om$. In particular, $\Maxr A$ is nonempty and finite.
\end{obs}

Thus, we identify $\FinrE$ with the set of all $<_\coc$-decreasing $E$-related sequences $(x_n)_{n < \ell} \subseteq X$, $\ell \le \omega$, that are \textit{$\coc$-summable}, i.e. $\sum_{n < \ell} \coc(x_n) < \infty$. This is clearly a Borel subset of $X^{< \N} \cup X^\N$. Furthermore, we view $\FinE$ as a Borel subset of $\FinrE$.

In our arguments below, we implicitly use the following.

\begin{obs}\label{membership_is_Borel}
    The set $\set{(x,U) \in X \times \FinrE : x \in U}$ is a Borel subset of $X \times \FinrE$.
\end{obs}

The partial order $>_\coc$ induces a (lexicographic) partial order $\cocmore$ on $\FinE$ as follows: for any $A,B \in \FinE$, $A \cocmore B$ if and only if $A \ne B$ yet $[A]_E = [B]_E$ and
\[
\textstyle |A| < |B| \text{ or } (|A|=|B| \text{ and } \max_{<_\coc}(A \symdiff B) \in A),
\]
where $\max_{<_\coc}(A \symdiff B)$ is the maximum element of $A \symdiff B$ with respect to $<_\coc$. For each $E$-class $C$, $\cocmore$ is a linear order on $\FinE[C]$. Moreover, because $>_\coc$ is a well-order on any $A \in \FinrE$, it is not hard to check that $\cocmore$ is also a well-order on $\FinE[A]$. Thus:

\begin{obs}\label{cocord_well-order}
    For every $\coc$-finite $E$-class $C$, $\cocmore$ is a well-order on $\FinE[C]$.
\end{obs}

\subsubsection*{\cocratio of $E$-related sets}\label{subsubsec:coc-ratio}

As mentioned in the introduction, in some constructions below we need to ensure that an increasing union of finite (or $\coc$-finite) sets is $\coc$-infinite, and we noted that just making the cardinality of the sets grow is not enough. However, replacing cardinality with the following works:

\begin{defn}\label{defn:coc-ratio}
	For each $U \in \FinrE$, \cref{Maxr_nonempty_finite} implies that
	\[
	\cocmax(U) \defeq \frac{\coc(U)}{\maxr U} = \min_{x \in U} \coc_x(U),
	\]
	is well-defined and we call it the \emph{\cocratio} of $U$.
\end{defn}

Although $U \subseteq V$ does not imply $\cocmax(U) \le \cocmax(V)$, we still get the desired property:

\begin{obs}\label{increasing_ratio=>infinite}
	For any $U,V \in \FinrE$, $U \subseteq V$ implies $\frac{\cocmax(V)}{\cocmax(U)} \le \frac{\coc(V)}{\coc(U)}$. In particular, for any increasing sequence $(U_n)$ of sets in $\FinrE$, $\cocmax(U_n) \to \w$ implies $\coc(U_n) \to \w$.
\end{obs}

When the increasing union is $\coc$-finite, $\cocmax$ is monotone:

\begin{lemma}\label{incU_finite=>increasing_ratio}
    For any increasing sequence $(U_n)$ of sets in $\FinrE$, if $U \defeq \bigincU_n U_n$ is $\coc$-finite, then $\Maxr U \subseteq U_n$ for all large enough $n$. In particular, the sequence $(\cocmax(U_n))$ is eventually increasing and $\lim_n \cocmax(U_n) = \cocmax(U)$.
\end{lemma}
\begin{proof}
 By \cref{Maxr_nonempty_finite}, $\Maxr U$ is finite, so there is $N$ such that $\Maxr U_N = \Maxr U$. Thus, for any $x \in \Maxr U$ and all $n \ge N$, 
\[
\cocmax(U_n) = \coc_x(U_n) \nearrow \coc_x(U) = \cocmax(U).
\qedhere
\]
\end{proof}

\subsection{Graphs}\label{subsec:graphs}

Let $\vec{G}$ be a locally countable directed Borel graph on $X$, i.e.~$\vec{G}$ is a Borel subset of $X^2$ such that the fibers $\vec{G}_x \defeq \set{y \in X : (x,y) \in \vec{G}}$ and $\vec{G}^x \defeq \set{y \in X : (y,x) \in \vec{G}}$ are countable for each $x \in X$. We drop ``directed'' if $\vec{G}$ is \textit{symmetric}, i.e. $\vec{G} = - \vec{G} \defeq \set{(y,x) : (x,y) \in \vec{G}}$. Let $G$ denote the \textit{symmetrization} of $\vec{G}$, i.e. $G \defeq \vec{G} \cup (-\vec{G})$. For sets $A,B \subseteq X$, the following is standard notation:

\begin{itemize}
    \item $\DirEdgsBtw{\vec{G}}{A,B} \defeq (A \times B) \cap \vec{G}$;
    
	\item $\EdgsBtw{\vec{G}}{A,B} \defeq (A \times B \cup B \times A) \cap \vec{G}$;
	
	\item $\vec{G} \rest{A} \defeq \vec{G} \cap A^2$;
	
	\item $\vec{G}_{-A} \defeq \vec{G} \rest{X \setminus A}$;

	\item $\bdrin{\vec{G}} A \defeq \bdrin{G} A \defeq \set{a \in A : \text{there is $b \in X \setminus A$ with $(a,b) \in G$}}$, the \emph{inner $\vec{G}$-boundary} of $A$;
	
	\item $\bdrout{\vec{G}} A \defeq \bdrout{G} A \defeq \set{b \in X \setminus A : \text{there is $a \in A$ with $(a,b) \in G$}}$, the \emph{outer $\vec{G}$-boundary} of $A$.
\end{itemize}

We denote by $E_G$ the equivalence relation of being in the same $G$-connected component, and put $E_{\vec{G}} \defeq E_G$. The rest of the terminology and most of the paper is about a symmetric graph $G$. Call a set $U \subseteq X$ \emph{$G$-connected} if $G \cap U^2$ is a connected graph on $U$, and let $\FinG$ denote the (Borel) subset of $\FinX_{E_G}$ of $G$-connected sets. Similarly, for a Borel cocycle $\coc : E_G \to \R^+$, let $\FinrG$ denote the corresponding subset of $[X]^{\coc < \w}_{E_G}$. Say that an equivalence relation $F$ on $X$ is \emph{$G$-connected} if each $F$-class is $G$-connected.

We say that $G$ is \emph{component-finite} if each $G$-connected component is finite. For a Borel measure $\mu$ on $X$, we say that $G$ is \emph{(quasi-)pmp, $\mu$-ergodic, Borel hyperfinite, $\mu$-hyperfinite}, etc. if $E_G$ is. In particular, if $G$ is hyperfinite, i.e. $E_G$ is a countable increasing union of finite Borel equivalence relations $E_n$, then $G$ is a countable increasing union of component-finite Borel graphs $G_n \defeq E_n \cap G$.

\subsection{Quotients}\label{subsec:quotients}

For a smooth countable Borel equivalence relation $F$ on $X$, the quotient space $X/F$ is also standard Borel and we denote it by $\Xmod{F}$. We emphasize that although $\Xmod{F}$ can be realized as a Borel subset of $X$, but we think of it as the set of $F$-classes.  We denote the quotient map by $\pi_F : X \onto \Xmod{F}$.

For an $F$-invariant function $f : X \to \R$, we define its quotient $\Xmod{F} \to \R$ by $[x]_F \mapsto f(x)$ and we still denote it by $f$ to not overload notation. For a Borel measure $\mu$ on $X$, we put $\mumod{F} \defeq (\pi_F)_\ast \mu$. For a countable Borel equivalence relation $E \supseteq F$, we denote by $\Emod{F}$ the pushforward of $E$ under $\pi_F$. For $\SC \subseteq \Pow(X)$, let
\[
\SCmod{F} \defeq \set{P \in \Pow(\Xmod{F}) : \pi_F^{-1}(P) \in \SC}.
\]

For a locally countable Borel graph $G$ on $X$ with $F \subseteq E_G$, we form the \emph{quotient graph} $\Gmod{F}$ by contracting the edges between $F$-related vertices, i.e. for $U,V \in \Xmod{F}$, 
\[
(U,V) \in \Gmod{F} \defequiv \EdgsBtw{G}{U,V} \ne \0. 
\]
In our arguments, we need $\pi_F^{-1}(A)$ to be $G$-connected for every $\Gmod{F}$-connected set $A \subseteq \Xmod{F}$; this happens exactly when $F$ is $G$-connected, so we only take quotients of $G$ by $G$-connected equivalence relations.

Moreover, we only take quotients by $\coc$-finite $F$ for a Borel cocycle $\coc : F \to \R^+$, which is allowed by the following:

\begin{lemma}\label{coc-finite_is_smooth}
For any countable Borel equivalence relation $F$ and Borel cocycle $\coc : F \to \R^+$, if $F$ is $\coc$-finite then it is smooth.
\end{lemma}
\begin{proof}
Follows from \cref{Maxr_nonempty_finite} because the map $x \mapsto $ the $<_\coc$-largest element of $[x]_F$ is a Borel selector for $F$.
\end{proof}

Lastly, let $E$ be a quasi-pmp equivalence relation on a standard probability space $(X,\mu)$ and let $\coc : E \to \R^+$ be the Radon--Nikodym cocycle of $E$ with respect to $\mu$.

\begin{obs}\label{coc-quotient}
For any $\coc$-finite Borel subequivalence relation $F \subseteq E$, $\mumod{F}$ is $\Emod{F}$-quasi-invariant, and the Radon--Nikodym cocycle $\rmod{F}$ of $\Emod{F}$ with respect to $\mumod{F}$ is given by
\[
(\rmod{F})_{[x]_F} ([y]_F) \defeq \coc([y]_F) / \coc([x]_F)
\]
for any $(x,y) \in E$. In particular, for any $U \in [\Xmod{F}]^{\rmod{F} < \infty}_{\Emod{F}}$, $\cocmaxF(U) \le \cocmax(\pi_F^{-1}(U))$.
\end{obs}


\section{Finite and hyperfinite averages}\label{sec:hyperfin_averages}

This subsection is the generalization of \cite{Miller-Ts:erg_hyp_dec}*{Subsections 7.A--B} to the quasi-invariant setting. We first state a basic fact about finite averages, which we use repeatedly below. For an abstract set $Y$, treating $Y^2$ as the trivial equivalence relation on $Y$, we let $\coc : Y^2 \to \R^+$ be a cocycle, and for a bounded function $f : Y \to \R$ and a $\coc$-finite nonempty set $U \subseteq Y$, we define the \textit{$\coc$-weighted average of $f$ on $U$} by
\[
\meanrf{U} \defeq \frac{\sum_{x \in U} f(x) \coc(x)}{\coc(U)}.
\]

\begin{lemma}\label{convexity_of_average}
For $Y, \coc$ and $f$ as above, and for nonempty disjoint $\coc$-finite sets $U,V \subseteq Y$,

\begin{enumerate}[(a)]
	\item \label{item:convexity_of_average:disjU_of_averages}
	$\meanrf{U \cup V} = \frac{\coc(U)}{\coc(U) + \coc(V)} \meanrf{U} + \frac{\coc(V)}{\coc(U) + \coc(V)} \meanrf{V}$. 
	
	\item \label{item:convexity_of_average:increment_bound}
	$|\meanrf{U \cup V} - \meanrf{U}| 
	\le 
	2 \, \Linf{f} \frac{\coc(V)}{\coc(U) + \coc(V)}
	\le 
	2 \, \Linf{f} \frac{\coc(V)}{\coc(U)}$.
\end{enumerate}
\end{lemma}
\begin{proof}
One verifies \labelcref{item:convexity_of_average:disjU_of_averages} directly, and \labelcref{item:convexity_of_average:increment_bound} follows from \labelcref{item:convexity_of_average:disjU_of_averages} by the triangle inequality.
\end{proof}

Now let $(X,\mu)$ be a standard probability space. For finite and hyperfinite equivalence relations on $X$, we now give explicit formulas for the conditional expectation of any $f \in L^1(X,\mu)$ with respect to the $\sigma$-algebras of invariant sets.

\smallskip

The following simple statement is what connects local finite averages to the global average (i.e. the integral). The author has been calling statements of this type \textit{local-global bridges} and has been using them in all her proofs of pointwise ergodic theorems as they reduce proving these theorems to solving local combinatorial problems.
 
\begin{lemma}[Local-global bridge]\label{local-global_bridge}
Let $F$ be a quasi-pmp equivalence relation on $(X,\mu)$ and let $\coc : F \to \R^+$ be its Radon--Nikodym cocycle with respect to $\mu$. If $F$ is $\coc$-finite, then for any $f \in L^1(X,\mu)$, the function 
\[
\meanrFf : X \to \R \;\text{ defined by }\; x \mapsto \meanrf{[x]_F}
\]
is the conditional expectation of $f$ with respect to the $\sigma$-algebra of $F$-invariant Borel sets, i.e.
\begin{equation}\label{eq:bridge}
\int_Y f d\mu = \int_Y \meanrFf d\mu,    
\end{equation}
for any $F$-invariant measurable set $Y \subseteq X$. In particular, $\Lp{\meanrFf} \le \Lp{f}$ for any $p \ge 1$.
\end{lemma}
\begin{proof}
The inequality $\Lp{\meanrFf} \le \Lp{f}$ is true in general for any conditional expectation \cite{Durrett}*{Theorem 4.1.11}, but in this case it trivially follows from Jensen's inequality $|\meanrFf|^p \le \meanr{F} |f|^p$ and \labelcref{eq:bridge} applied to $|f|^p$.

As for \labelcref{eq:bridge}, it is enough to prove it assuming that $Y = X$ and all $F$-classes have the same cardinality $n$ for some fixed $n \le |\N|$. To this end, let $S$ be the set of $<_\coc$-maximum elements in each $F$-class (which exist by \cref{Maxr_nonempty_finite}), so $S$ meets every $F$-class in exactly one point. Using this set $S$ and Luzin--Novikov uniformization, we get a Borel automorphism $T : X \to X$ that induces $F$, i.e. $[x]_F = \set{T^i x}_{i \in \Z}$ for every $x \in X$. Put $I \defeq \set{0,1,\dots,n-1}$ if $n$ is finite, and $I \defeq \Z$, otherwise, so for each $x \in X$, the sequence $(T^i x)_{i \in I}$ lists all elements of $[x]_F$ without repetition. Then:
\begin{align*}
\int_X f \; d \mu
&=
\sum_{i \in I} \int_{T^i(S)} f(x) \; d \mu(x)
\\
\eqcomment{change of variable $x$ to $T^i x$}
&= 
\sum_{i \in I} \int_S f(T^i x) \cdot \coc_x(T^i x) \; d \mu(x)
\\
&=
\int_S \sum_{i \in I} f(T^i x) \cdot \coc_x(T^i x) \; d \mu(x)
\\
&=
\int_S \sum_{i \in I} \meanrFf(x) \cdot \coc_x(T^i x) \; d \mu(x)
\\
\eqcomment{$\meanrFf$ is $T$-invariant}
&=
\sum_{i \in I} \int_S \meanrFf(T^i x) \cdot \coc_x(T^i x) \; d \mu(x)
\\
\eqcomment{change of variable $x$ to $T^{-i} x$}
&= 
\sum_{i \in I} \int_{T^i(S)} \meanrFf(x) \; d \mu(x)
=
\int_X \meanrFf \; d \mu
.
\qedhere
\end{align*}
\end{proof}

The following statement is the analogue of the maximal inequality, which makes it enough to prove \cref{ptwise_ergodic_for_graphs} for bounded functions.

\begin{lemma}[Approximate $L^\w$-continuity]\label{approx_Linf-continuity}
Let $F$ be a quasi-pmp equivalence relation on $(X,\mu)$ and let $\coc : F \to \R^+$ be its Radon--Nikodym cocycle with respect to $\mu$. If $F$ is $\coc$-finite, then for each $\e > 0$, there is an $F$-invariant $\mu$-co-$\e$ set $X' \subseteq X$ such that 
\[
\Linfbig{(\meanrFf) \rest{X'}} \le \frac{1}{\e} \Lone{f}.
\]
\end{lemma}

We need the following basic lemma for the proof.

\begin{lemma}\label{Linf<Lone}
For any $f \in L^1(X,\mu)$ and $\e > 0$, there is a $\mu$-co-$\e$ set $X' \subseteq X$ (of the form $\set{x \in X : |f(x)| \le r}$ for some $r \ge 0$) such that $\Linf{f \rest{X'}} \le \frac{1}{\e} \Lone{f}$.
\end{lemma}
\begin{proof}
Letting $r_\e \defeq \inf \set{r \ge 0 : \mu(f > r) \le \e}$, we show that $X' \defeq \set{f \le r_\e}$ works. Then $\set{f > r_\e} = \bigcup_{n \ge 1} \set{f > r_\e + \frac{1}{n}}$, so $\mu(\set{f > r_\e}) \le \e$, hence $X'$ is $\mu$-co-$\e$. On the other hand, $\mu(f \ge r_\e) \ge \e$ because 
\[
\set{f \ge r_\e} = \bigcap_{n \ge 1} \set{f > r_\e - \frac{1}{n}},
\] 
and $\mu(\set{f > r_\e - \frac{1}{n}}) > \e$ for each $n \ge 1$. Thus,
\[
\Linf{f \rest{X'}} \cdot \e \le r_\e \cdot \mu(f \ge r_\e) \le \Lone{f}.\qedhere
\]
\end{proof}

\begin{proof}[Proof of \cref{approx_Linf-continuity}] By \cref{local-global_bridge}, $\Lone{\meanrFf} \le \Lone{f} < \infty$. Thus,
\cref{Linf<Lone} applies to $\meanrFf$ and gives an $F$-invariant $\mu$-co-$\e$ set $X' \subseteq X$ such that $\Linfbig{(\meanrFf) \rest{X'}} \le \frac{1}{\e} \Lone{\meanrFf} \le \frac{1}{\e} \Lone{f}$.
\end{proof}

\begin{theorem}[Pointwise ergodic theorem for hyperfinite equivalence relations]\label{hyperfin_averages}
Let $F$ be a Borel quasi-pmp hyperfinite equivalence relation on $(X,\mu)$ and let $\coc : F \to \R^+$ be its Radon--Nikodym cocycle with respect to $\mu$. For any $f \in L^1(X,\mu)$ and any increasing sequence $(F_n)$ of finite Borel equivalence relations with $F = \bigcup_n F_n$, the pointwise limit
\[
\meanrFf \defeq \lim_{n \to \w} \meanrf{F_n}
\]
exists a.e.~and is equal to the conditional expectation of $f$ with respect to the $\si$-algebra of $F$-invariant measurable sets, i.e.
\begin{enumerate}[(a),series=hyperfin_averages]
	\item \label{item:hyperfin_averages:equality} $\int_Y \meanrFf \, d \mu = \int_Y f \, d \mu$ for any $F$-invariant measurable set $Y \subseteq X$.
\end{enumerate}
In particular,
\begin{enumerate}[resume*=hyperfin_averages]
    \item \label{item:hyperfin_averages:indep-of-witness} Up to null sets, the above limit does not depend on the choice of $(F_n)$.
	
	\item \label{item:hyperfin_averages:Lp} $\Lp{\meanrFf} \le \Lp{f}$ for every $p \ge 1$.
	
	\item \label{item:hyperfin_averages:ergodicity} $F$ is ergodic if and only if for every $f \in L^1(X,\mu)$, $\meanrFf = \int_X f d\mu$ a.e.
\end{enumerate}
\end{theorem}
\begin{proof}
The main part, \labelcref{item:hyperfin_averages:equality}, is just a version of \cite{Miller-Ts:erg_hyp_dec}*{Theorem 7.3} in the quasi-invariant setting, and the proof is the same, given \cref{local-global_bridge}, so we omit it. Parts \labelcref{item:hyperfin_averages:indep-of-witness,item:hyperfin_averages:ergodicity} are immediate from \labelcref{item:hyperfin_averages:equality}. Part \labelcref{item:hyperfin_averages:Lp} is a well-known fact about conditional expectation \cite{Durrett}*{Theorem 4.1.11}, and it can be seen directly here.
\end{proof}

\section{Proofs of \cref{ptwise_ergodic_for_graphs,ergodic_hyperfin_subgraph,ratio_ergodic_for_ergodic_graphs} from Main Lemma \namedthmlabelref{core_lemma}}\label{sec:equiv_of_thms}

We restate the statement of Main Lemma \namedthmlabelref{core_lemma} using the introduced terminology.

\begin{namedthm*}{Main Lemma \namedthmlabelref{core_lemma}}
Let $G$ be a locally countable $\mu$-nowhere hyperfinite quasi-pmp Borel graph on a standard probability space $(X,\mu)$ and let $\coc : E_G \to \R^+$ be the Radon--Nikodym cocycle of $E_G$ with respect to $\mu$. For any $f \in L^\infty(X,\mu)$ and $\e > 0$, there is a component-finite Borel subgraph $H \subseteq G$ such that $\meanrf{E_H} \approx_\e \E(f | \BC_{E_G})$ on a $\mu$-co-$\e$ set.
\end{namedthm*}

Assuming this lemma, we now derive \cref{ptwise_ergodic_for_graphs,ergodic_hyperfin_subgraph,ratio_ergodic_for_graphs}, restating them here using the introduced terminology.

\begin{namedthm*}[Ergodic theorem for quasi-pmp graphs]{\cref{ptwise_ergodic_for_graphs}}
	Let $G$ be a locally countable quasi-pmp Borel graph on a standard probability space $(X,\mu)$ and let $\coc : E_G \to \R^+$ be the Radon--Nikodym cocycle of $E_G$ with respect to $\mu$. There is an increasing sequence $(F_n)$ of $G$-connected finite Borel subequivalence relations of $E_G$ such that for any $p \ge 1$ and $f \in L^p(X,\mu)$, 
	\[
	\lim_{n \to \w} \meanrf{F_n} = \E(f | \BC_{E_G}) \hspace{8pt}\text{a.e. and in $L^p$.}
	\]
	
\end{namedthm*}
\begin{proof}
We first prove convergence in $L^p$ assuming the a.e.~convergence statement. Fix $f \in L^p(X,\mu)$ and $\e > 0$, and let $\hat{f} \in L^\infty(X,\mu)$ be such that $\Lp{f - \hat{f}} < \frac{\e}{3}$. \cite{Durrett}*{Theorem 4.1.11} implies that $\Lp{\E(f | \BC_{E_G}) - \E(\hat{f} | \BC_{E_G})} \le \Lp{f - \hat{f}} < \frac{\e}{3}$ and similarly, \cref{local-global_bridge} implies that $\Lp{\meanrf{F_n} - \meanr{F_n} \hat{f}} \le \Lp{f - \hat{f}} < \frac{\e}{3}$ for any $n \in \N$. Because $|\meanr{F_n} \hat{f}| \le \Linf{\hat{f}}$, it follows from the Dominated Convergence theorem and the a.e. convergence of $\meanr{F_n} \hat{f}$ to $\E(\hat{f} | \BC_{E_G})$ that $\meanr{F_n} \hat{f} \to \E(\hat{f} | \BC_{E_G})$ in $L^p$, so for all large enough $n$, $\Lp{\meanr{F_n} \hat{f} - \E(\hat{f} | \BC_{E_G})} < \frac{\e}{3}$. Thus, $\Lp{\meanrf{F_n} - \E(f | \BC_{E_G})} < \e$ by the triangle inequality.

We now prove the a.e.~convergence statement for a fixed $f \in L^1(X,\mu)$. If $G$ is $\mu$-hyperfinite, then \cref{hyperfin_averages} implies that any increasing sequence $(F_n)$ of $G$-connected finite Borel equivalence relations with $E_G = \bigcup_n F_n$ satisfies the desired conclusion. Thus, we assume that $G$ is not $\mu$-hyperfinite; in fact, by discarding a maximal (up to $\mu$-null sets) $E_G$-invariant Borel set on which $G$ is hyperfinite, we may assume that $G$ is $\mu$-nowhere hyperfinite.

Let $\DC \subseteq L^\infty(X,\mu)$ be a countable set that is dense in $L^1(X,\mu)$, and let $(f_n)_{n \ge 1}$ be an enumeration of $\DC$ such that each $f \in \DC$ is equal to $f_n$ for infinitely-many $n \in \N$. Furthermore, let $(\e_n)_{n \ge 1}$ be a sequence of positive reals converging to $0$. We inductively build an increasing sequence $(F_n)$ of $G$-connected finite Borel equivalence relations by taking $F_0 \defeq \Id_X$ and letting $F_{n+1}$ be the pullback to $X$ of $E_{H_{n+1}}$, where $H_{n+1}$ is a component-finite Borel subgraph of $\Gmod{F_n}$ given by Main Lemma \namedthmlabelref{core_lemma} when applied to the quotient graph $\Gmod{F_n}$ on $(\Xmod{F_n}, \mumod{F_n})$, the function $\meanr{F_n} f_{n+1}$, and $\e \defeq \e_{n+1}$. Thus, for each $n \ge 1$, $\meanr{F_n}f_n \approx_{\e_n} \E(f_n | \BC_{E_G})$ on an $F_n$-invariant $\mu$-co-$\e_n$ set $X_n \subseteq X$. 

Now fix $f \in L^1(X,\mu)$. By \cref{hyperfin_averages}\labelcref{item:hyperfin_averages:indep-of-witness} applied to $F \defeq \bigcup_n F_n$, it is enough to extract a subsequence $(F_{n_k})$ such that
\[
	\lim_{n \to \w} \meanrf{F_{n_k}} = \E(f | \BC_{E_G}) \; \text{a.e.}
\]
By the Borel--Cantelli lemma, it is enough to show that for any $\e \in (0,1)$, there is $n \in \N$ and an $F_n$-invariant $\mu$-co-$\e$ set $X' \subseteq X$ such that 
$
	\meanrf{F_n} \approx_{\frac{\e}{3}} \E(f | \BC_{E_G})
$
for all $x \in X'$. 

Because there is a subsequence $(f_{n_k})$ converging to $f$ in $L^1$ and conditional expectation is an $L^1$-contraction \cite{Durrett}*{Theorem 4.1.11}, we also have that $\E(f_{n_k} | \BC_{E_G}) \to \E(f | \BC_{E_G})$ in $L^1$, so there is a further subsequence that converges a.e. Thus, for any fixed $\e$, there is $n \in \N$ such that $\e_n$ and $\Lone{f - f_n}$ are both less than $\frac{\e^2}{3}$, and $\E(f_n | \BC_{E_G}) \approx_{\frac{\e}{3}} \E(f | \BC_{E_G})$ on a $\mu$-co-$\frac{\e}{3}$ set. By \cref{approx_Linf-continuity},
\[
	\left|\meanr{F_n} f_n - \meanrf{F_n}\right|
	\le 
	\meanr{F_n} |f_n - f| 
	\le 
	\frac{1}{\e}\Lone{f_n - f} < \frac{\e}{3},
\]
on a $\mu$-co-$\frac{\e}{3}$ set. Thus, by the triangle inequality, $\meanrf{F_n} \approx_\e \E(f | \BC_{E_G})$ on a $\mu$-co-$\e$ set.
\end{proof}

\begin{namedthm*}[Ergodic hyperfinite subgraph]{\cref{ergodic_hyperfin_subgraph}}
	Every locally countable ergodic Borel graph $G$ on a standard probability space $(X,\mu)$ admits an ergodic hyperfinite Borel subgraph $H \subseteq G$.
\end{namedthm*}
\begin{proof}
We first note that without loss of generality, we may assume that $\mu$ is $E_G$-quasi-invariant by replacing it with $\sum_{m \ge 1} 2^{-n} (\ga_m)_\ast \mu$, where $(\ga_m)_{m \ge 1} \subseteq [E_G]$ is a sequence of Borel automorphisms such that $E = \bigcup_{m \ge 1} \Graph(\ga_m)$, which exists by the Feldman--Moore theorem \cite{Feldman-Moore}. Now let $(F_n)$ be an increasing sequence of finite Borel equivalence relations provided by \cref{ptwise_ergodic_for_graphs}. Then $F \defeq \bigcup_{n \in \N} F_n$ is hyperfinite and $G$-connected. Furthermore, for every $f \in L^1(X,\mu)$, $\meanrFf = \int_X f d\mu$ a.e., so by \cref{hyperfin_averages}\labelcref{item:hyperfin_averages:ergodicity}, $F$ is ergodic.
\end{proof}

Lastly, we state and prove the general version of \cref{ratio_ergodic_for_ergodic_graphs}.

\begin{theorem}[Ratio ergodic theorem for quasi-mp graphs]\label{ratio_ergodic_for_graphs}
Let $G$ be a locally countable quasi-mp Borel graph on a $\sigma$-finite standard measure space $(X,\mu)$ and let $\coc : E_G \to \R^+$ be the Radon--Nikodym cocycle of $E_G$ with respect to $\mu$. There is an increasing sequence $(F_n)$ of $G$-connected finite Borel equivalence relations $F_n$ such that for any $f,g \in L^1(X,\mu)$ with $g > 0$,
\[
r \defeq \lim_{n \to \w} 
\frac{\meanrf{F_n}}{\meanr{F_n} g}
= 
\E_{\mu_g}(fg^{-1} | \BC_{E_G}) \hspace{8pt}\text{a.e.}
\]
\noindent where $\E_{\mu_g}$ denotes the conditional expectation with respect to the measure $\mu_g$ defined by $d \mu_g \defeq \frac{1}{\Lone{g}} g d \mu$. In other words, for any $E_G$-invariant Borel set $Y \subseteq X$,
\[
\int_Y rg d\mu = \int_Y f d\mu.
\]
In particular, when $G$ is $\mu$-ergodic, 
\[
\lim_{n \to \w} 
\frac{\meanrf{F_n}}{\meanr{F_n} g}
= 
\frac{\int_X f d\mu}{\int_X g d\mu} \hspace{8pt}\text{a.e.}
\]
\end{theorem}

\begin{proof}
Observe that $\mu_g \sim \mu$ and the Radon--Nikodym cocycle of $E_G$ with respect to $\mu_g$ is $\~\coc_x(y) = \coc_x(y) g(y) g(x)^{-1}$, so
\[
\frac{\meanrf{F_n}(x)}{\meanr{F_n} g(x)}
=
\frac{\sum_{y \in [x]_{F_n}} f(y) \coc_x(y)}{\sum_{y \in [x]_{F_n}} g(y) \coc_x(y)} 
=
\frac{\sum_{y \in [x]_{F_n}} f(y)g(y)^{-1} \coc_x(y) g(y) g(x)^{-1}}{\sum_{y \in [x]_{F_n}} \coc_x(y) g(y) g(x)^{-1}} 
= 
\mean{\~\coc}{F_n} (f g^{-1})(x).
\]
It is clear that $fg^{-1} \in L^1(X,\mu_g)$, so \cref{ptwise_ergodic_for_graphs} applied to $G$ and $fg^{-1}$ on $(X,\mu_g)$ gives the a.e. convergence of $\frac{\meanrf{F_n}}{\meanr{F_n} g}$ to $\E_{\mu_g}(fg^{-1} | \BC_{E_G})$. This means that for any $E_G$-invariant Borel set $Y \subseteq X$,
\[
\int_Y r g d\mu = \Lone{g} \cdot \int_Y r d\mu_g = \Lone{g} \cdot \int_Y fg^{-1} d\mu_g = \int_Y f d\mu.
\]
Lastly, if $G$ is $\mu$-ergodic then it is also $\mu_g$-ergodic because $\mu \sim \mu_g$, hence 
\[
\E_{\mu_g}(fg^{-1} | \BC_{E_G}) = \int_X fg^{-1} d\mu_g = \frac{\int_X f d\mu}{\int_X g d\mu} \;\;\text{ a.e.}
\qedhere
\]
\end{proof}

\section{Mass transport along a cocycle}\label{sec:flows}

Throughout this section, let $X$ be a standard Borel space, $E$ a countable Borel equivalence relation on $X$, and $\coc : E \to \R^+$ a Borel cocycle.

\subsection{Preliminaries}

We refer to a function $\phi : E \to [0,\w)$ as a \textit{mass transport} or just \textit{transport} (in $E$), and we think of the value $\phi(x,y)$ as the fraction of the $\coc$-mass of $x$ that gets transported from $x$ to $y$; thus, what $y$ receives in this transaction is $\phi(x,y) \coc_y(x)$. With this in mind, we define functions $\infl,\, \outfl : X \to [0,\w]$ by
\begin{align*}
	\outfl (x) &\defeq \sum_{y \in [x]_E} \phi(x,y)
	\\
	\infl (x) &\defeq \sum_{y \in [x]_E} \phi(y,x) \coc_x(y).
\end{align*}

\begin{defn}
We say that a transport $\phi : E \to [0,\w)$ is \emph{$\coc$-bounded} if $\outfl$ and $\infl$ are bounded by $1$.
\end{defn}

The definition of a $\coc$-bounded transport in $E$ is exactly the same as that of a \emph{$\coc$-invariant fuzzy partial injection} defined in \cite{Miller:meas_with_cocycle_II}, and in our arguments we will only use such transports. 

When defining a transport, we will only partially specify its values with the convention that the undefined values are treated as $0$. We also use the following terminology:
\begin{itemize}
    \item The \emph{domain} of $\phi$ is $\dom(\phi) \defeq \set{(x,y) \in E : \phi(x,y) \ne 0}$.
    
	\item The \emph{vertex-domain} of $\phi$ is $\vdom(\phi) \defeq \proj_1(\dom(\phi)) \cup \proj_2(\dom(\phi))$. 
	
	\item For sets $Y,Z \subseteq X$, we say $Y \times Z$ is \emph{$\phi$-closed} if $\dom(\phi)$ is disjoint from $(Y \times Z^c) \cup (Y^c \times Z)$. 
	
\end{itemize}

\subsection{Mass transport principle}

\begin{lemma}\label{in-flow=out-flow}
	For a transport $\phi$ in $E$, and for any subsets $U,V$ of the same $E$-class in $X$ such that $U \times V$ is $\phi$-closed, we have
	\[
	\int_U \outfl d\coc = \int_V \infl(v) d\coc.
	\]
\end{lemma}
\begin{proof}
	Letting $x$ be any element in $[U \cup V]_E$, we compute:
	\begin{align*}
		\sum_{u \in U} \outfl(u) \cdot \coc_x(u)
		&=
		\sum_{u \in U} \sum_{v \in V} \phi(u,v) \cdot \coc_x(u)
		\\
		\eqcomment{Fubini}
		&=
		\sum_{v \in V} \sum_{u \in U} \phi(u,v) \cdot \frac{\coc_x(u)}{\coc_x(v)} \cdot \coc_x(v)
		\\
		\eqcomment{cocycle identity}
		&=
		\sum_{v \in V} \sum_{u \in U} \phi(u,v) \cdot \coc_v(u) \cdot \coc_x(v)
		\\
		&=
		\sum_{v \in V} \infl(v) \cdot \coc_x(v).
		\qedhere
	\end{align*}
\end{proof}

The following is the measure version of \cref{in-flow=out-flow}, which is just \cite{Miller:meas_with_cocycle_II}*{Proposition 6.4} or \cite{Kechris-Miller}*{Proposition 8.2} phrased in our terminology.

\begin{prop}\label{in-flow=out-flow_for_measure}
	For a Borel transport $\phi$ in $E$ and any $\coc$-invariant Borel probability measure $\mu$ on $X$,
	
	\[
	\int_X \outfl \,d\mu = \int_X \infl \,d\mu.
	\]
\end{prop}
\begin{proof}
	We assume without loss of generality that $\phi(x,x) = 0$ for all $x \in X$ because changing the value of $\phi(x,x)$ does not change the validity of the desired equality.
	
	\begin{claim+}
	    $E \setminus \Id_X$ is the disjoint union of the sets $\Graph(\ga_n) \setminus \Id_X$, $n \in \N$, where each $\ga_n$ is a Borel involution on $X$.
	\end{claim+}
	\begin{pf}
	    By the Feldman--Moore theorem \cite{Feldman-Moore}, $E$ is a union of $\Graph(\beta_n)$, $n \in \N$, where each $\beta_n$ is a Borel involution on $X$. Recalling that set-theoretically, $\beta_n$ is the same as $\Graph(\beta_n)$, we define partial involutions $\beta_n' \defeq \beta_n \setminus \bigcup_{k < n} \beta_k$ whose graphs are pairwise disjoint. Let $\gamma_n : X \to X$ be equal to $\beta_n'$ on the domain of $\beta_n'$ and to the identity function, elsewhere.
	\end{pf}
	Putting $X_n \defeq \set{x \in X : \gamma_n x \ne x}$, we compute:
	\begin{align*}
		\int_X \outfl(x) d\mu(x) &= \int_X \sum_{n \in \N : \gamma_n x \ne x} \phi(x, \gamma_n x) \,d\mu(x)
		\\
		\eqcomment{Fubini}
		&=
		\sum_n \int_{X_n} \phi(x,\gamma_n x) \,d\mu(x) 
		\\
		\eqcomment{change of variable $x$ to $\gamma_n x$}
		&=
		\sum_n \int_{X_n} \phi(\gamma_n x, x) \cdot \coc_x(\gamma_n x) \,d\mu(x)
		\\
		\eqcomment{Fubini}
		&=
		\int_X \sum_{n \in \N : \gamma_n x \ne x} \phi(\gamma_n x, x) \cdot \coc_x(\gamma_n x) \,d\mu(x) 
		\\
		&=
		\int_X \infl(x) \,d\mu(x).
		\qedhere
	\end{align*}
\end{proof}

\subsection{Deficiency and $\coc$-invariant measures}

For a $\coc$-bounded transport $\phi$ in $E$,

\begin{itemize}
	\item call a point $x \in X$ a \emph{$\phi$-source} (resp. \emph{$\phi$-sink}) if $\outfl(x) - \infl(x)$ is positive (resp. negative). We denote the sets of $\phi$-sources and $\phi$-sinks by $\Sources(\phi)$ and $\Sinks(\phi)$, respectively. 
	
	\item say that $\phi$ \emph{disbalances} an $E$-class $C$ if $C$ contains at least one $\phi$-source but no $\phi$-sink, or vice versa, at least one $\phi$-sink but no $\phi$-source.
	
	\item say that $\phi$ \emph{disbalances} an $E$-invariant set $Z \subseteq X$ if it disbalances every $E$-class of $Z$.
	
	\item call $Y \subseteq X$ \emph{$\coc$-deficient} if there is a $\coc$-bounded Borel transport $\phi$ in $E$ disbalancing $[Y]_E$.
\end{itemize}

\smallskip

\noindent Taking sums of transports on disjoint $E$-invariant domains, we see that:
\begin{obs}\label{deficient-ideal}
	$\coc$-deficient sets form a $\si$-ideal.
\end{obs}

Thus, we say that a statement holds \emph{modulo $\coc$-deficient} if it holds on $X \setminus D$ for some $E$-invariant $\coc$-deficient Borel set $D$.

\cref{in-flow=out-flow_for_measure} immediately gives the following corollary.

\begin{cor}\label{deficient=>no_prob_meas}
	If a Borel $Y \subseteq X$ is $\coc$-deficient, then $\mu(Y) = 0$ for every $\coc$-invariant Borel probability measure $\mu$ on $X$.
\end{cor}

This corollary is all we need about mass transports in our proofs. However, it is well worth pointing out that its converse is also true (much more difficult to prove) and it is the content of \cite{Miller:meas_with_cocycle_II}*{Theorem 3}. We restate this here in our terms for the sake of completeness:

\begin{theorem}[Miller 2008]
	For a countable Borel equivalence relation $E$ on a standard Borel space $X$ and a Borel cocycle $\coc : E \to \R^+$, 
	\begin{enumerate}[(I),widest=II]
		\item \underline{either}: $X$ is $\coc$-deficient,
		
		\item \underline{or else}: there is a $\coc$-invariant Borel probability measure on $X$.
	\end{enumerate}
\end{theorem}

\subsection{Transporting a prescribed fraction of mass}

\begin{lemma}\label{finite-flow-definer}
	For any disjoint nonempty $\coc$-finite subsets $U,V$ of the same $E$-class in $X$ and a function $\p : U \to [0,1]$, if $\int_U \p \; d \coc \le \coc(V)$, then there is a $\coc$-bounded transport $\phi$ in $E$ with $\dom(\phi) \subseteq U \times V$ such that $\outfl \rest{U} = \p$.
\end{lemma}
\begin{proof}
    For each $(u,v) \in U \times V$, define $\phi(u,v) \defeq \p(u) \frac{\coc(v)}{\coc(V)}$. Then for each $u \in U$,
    \[
    \outfl(u) = \sum_{v \in V} \p(u) \frac{\coc(v)}{\coc(V)} = \p(u),
    \]
    while for each $v \in V$,
    \[
    \infl(v) = \sum_{u \in U} \p(u) \frac{\coc_v(v)}{\coc_v(V)} \coc_v(u) = \sum_{u \in U} \p(u) \frac{\coc_v(u)}{\coc_v(V)} = \frac{1}{\coc(V)} \int_U \p \; d\coc \le 1.
    \qedhere
    \]
\end{proof}

The proof of the last lemma can be carried out in a uniformly Borel fashion, yielding:

\begin{lemma}\label{flow-definer}
	Let $F \subseteq E$ be a $\coc$-finite Borel subequivalence relation, $U,V \subseteq X$ be disjoint Borel sets, and $\p : U \to [0,1]$ a Borel function. If for each $F$-class $Y$, $\int_{U \cap Y} \p \; d\coc \le \coc(V \cap Y)$, then there exists a $\coc$-bounded Borel transport $\phi$ with $\dom(\phi) \subseteq (U \times V) \cap F$ and $(\outfl) \rest{U} = \p$.
\end{lemma}

\section{Hyperfiniteness via cuts}\label{sec:finitizing_cuts}

In this section we present an easy but very useful characterization of $\mu$-hyperfinitess that turns non-$\mu$-hyperfiniteness into a positive property. Let $X$ be a standard Borel space and let $G$ be a locally countable Borel graph on it.

\begin{defn}
    Call a set $V \subseteq X$ a \emph{finitizing} (resp. \emph{hyperfinitizing}) \emph{vertex-cut for $G$} if $G_{-V}$ is component-finite (resp. hyperfinite). Likewise, call a subgraph $H \subseteq G$ a \emph{finitizing} (resp. \emph{hyperfinitizing}) \emph{edge-cut for $G$} if $G \setminus H$ is component-finite (resp. hyperfinite).
\end{defn}

Call a sequence of sets \emph{vanishing} if it is decreasing and has empty intersection.

\begin{prop}\label{hyperfin=vanish_edge-cuts}
	A locally countable Borel graph $G$ is hyperfinite if and only if it admits a vanishing sequence $(H_n)_n$ of finitizing Borel edge-cuts.
\end{prop}
\begin{proof}
	\lefttoright Letting $(G_n)$ be an increasing sequence of component-finite Borel graphs with $G = \bigincU_n G_n$, we see that the graphs $G \setminus G_n$ are vanishing finitizing edge-cuts for $G$.
	
	\righttoleft For each $n \ge 1$, $G_n \defeq G \setminus H_n$ is component-finite and $G = \bigincU_{n = 1}^\w G_n$.
\end{proof}

\begin{remark}
The analogue for vertex-cuts of \cref{hyperfin=vanish_edge-cuts} also holds for locally finite graphs. For locally countable graphs, the implication $\pmi$ still holds, but $\imp$ may fail: take any aperiodic hyperfinite equivalence relation $E$ and let $G \defeq E \setminus \Id_X$.
\end{remark}

Let $\mu$ and $\nu$ be finite Borel measures on $X$ and $G$, respectively.

\begin{defn}
	The \emph{finitizing vertex-price} (with respect to $\mu$) and \emph{finitizing edge-price} (with respect to $\nu$) of $G$ are the quantities:
	\begin{align*}
		\fvp(G) 
		&\defeq
		\inf \set{\mu(V) : V \subseteq X \text{ is a Borel finitizing vertex-cut for $G$}},
		\\
		\fep{\nu}(G) 
		&\defeq 
		\inf \set{\nu(H) : H \subseteq G \text{ is Borel finitizing edge-cut for $G$}}.
	\end{align*}
	Replacing ``finitizing'' with ``hyperfinitizing'' in the above definitions, we obtain \emph{hyperfinitizing vertex} and \emph{edge prices} denoted by $\hvp(G)$ and $\hep{\nu}(G)$.
\end{defn}

We characterize $\mu$-hyperfiniteness in terms of $\hvp(G)$, but the proof of this goes through finitizing edge-cuts and $\hep{\lift{\mu}}(G)$ for a measure $\lift{\mu}$ that is a \textit{lift} of $\mu$, that is: for any Borel (symmetric) subgraph $H \subseteq G$, $H$ is $\lift{\mu}$-null if and only if $\proj_1(H)$ is $\mu$-null; in particular, $(\proj_1)_\ast \lift{\mu} \sim \mu$. For example, writing $G$ as a countable union of the graphs of Borel maps $\ga_n : X \to X$ (by Luzin--Novikov uniformization), we define, for a Borel set $H \subseteq G$,
\[
\lift{\mu}(H) \defeq \sum_{n \ge 1} 2^{-n} \int_X \1_H(x, \ga_n x) d\mu(x).
\]

\begin{prop}\label{char_of_hyperfin_via_price}
	For a locally countable quasi-pmp Borel graph $G$ on a standard probability space $(X,\mu)$ and a lift $\lift{\mu}$ of $\mu$ to a \textbf{finite} Borel measure on $G$, the following are equivalent:
	\begin{enumerate}[(I),series=char_of_hyperfin_via_price,widest=III]
		\item \label{item:char_of_hyperfin:defn} $G$ is $\mu$-hyperfinite.
		
		\item \label{item:char_of_hyperfin:hyperfin_vertex-cuts} $\hvp(G) = 0$.
		
		\item \label{item:char_of_hyperfin:hyperfin_edge-cuts} $\hep{\lift{\mu}}(G) = 0$.
		
		\item \label{item:char_of_hyperfin:fin_edge-cuts} $\fep{\lift{\mu}}(G) = 0$.
	\end{enumerate}
	These statements are implied by the following, and are equivalent to it when $G$ is locally finite:
	\begin{enumerate}[resume*=char_of_hyperfin_via_price]
	    \item \label{item:char_of_hyperfin:fin_vertex-cuts} $\fvp(G)=0$.
	\end{enumerate}
\end{prop}
\begin{proof}
	\labelcref{item:char_of_hyperfin:defn} $\imp$ \labelcref{item:char_of_hyperfin:hyperfin_vertex-cuts} is trivial. For \labelcref{item:char_of_hyperfin:hyperfin_vertex-cuts} $\imp$ \labelcref{item:char_of_hyperfin:hyperfin_edge-cuts}, if $V \subseteq X$ is an arbitrarily $\mu$-small hyperfinitizing vertex-cut then $\EdgsBtw{G}{V,X}$ is an arbitrarily $\lift{\mu}$-small hyperfinitizing edge-cut, by $(\proj_1)_\ast \lift{\mu} \ll \mu$. \labelcref{item:char_of_hyperfin:hyperfin_edge-cuts} $\imp$ \labelcref{item:char_of_hyperfin:fin_edge-cuts} follows by an $\frac{\e}{2}+\frac{\e}{2}$ argument using that hyperfinite graphs admit $\lift{\mu}$-arbitrarily small finitizing edge-cuts, by \cref{hyperfin=vanish_edge-cuts} and the finiteness of $\lift{\mu}$. 
	
	For \labelcref{item:char_of_hyperfin:fin_edge-cuts} $\imp$ \labelcref{item:char_of_hyperfin:defn}, suppose $\fep{\lift{\mu}}(G) = 0$, and let $(H_n)$ be a sequence of finitizing Borel edge-cuts for $G$ such that $\mu(H_n)$ is summable. Then the graphs $H_n' \defeq \bigcup_{k \ge n} H_k$ are decreasing edge-cuts for $G$ and $H = \bigcap_n H_n'$ is $\lift{\mu}$-null. Hence, $Z \defeq \proj_1(H)$ is $\mu$-null because $\lift{\mu}$ is a lift of $\mu$, and by the quasi-invariance of $\mu$, $[Z]_{E_G}$ is still $\mu$-null. Throwing it out, makes $(H_n')$ a vanishing sequence of finitizing edge-cuts, so $G$ becomes hyperfinite by \cref{hyperfin=vanish_edge-cuts}. The same argument applied to vertices shows \labelcref{item:char_of_hyperfin:fin_vertex-cuts} $\imp$ \labelcref{item:char_of_hyperfin:defn}.
	
	It remains to show \labelcref{item:char_of_hyperfin:defn} $\imp$ \labelcref{item:char_of_hyperfin:fin_vertex-cuts}, assuming $G$ is locally finite. To this end, removing a $\mu$-null set from $X$, we may write $G$ as an increasing union of component-finite Borel graphs $G_n$. Then the sets $V_n \defeq \set{x \in X : x \in \bdrin{G} [x]_{G_n}}$ form a vanishing sequence of finitizing vertex-cuts because $\bdrout{G} [x]_{G_n}$ is finite for every $x \in X$. Thus, $\mu(V_n) \to 0$ by the finiteness of $\mu$.
\end{proof}

\begin{remark}\label{Dye-Krieger}
Although \cref{char_of_hyperfin_via_price} has a simple proof, it immediately implies the Dye--Krieger theorem \cite{Kechris-Miller}*{Theorem 6.11} that the increasing union $E \defeq \bigincU_n E_n$ of $\mu$-hyperfinite equivalence relations is $\mu$-hyperfinite. Indeed, if $H_n$ is a finitizing edge-cut for $E_n$, then $H \defeq \bigcup_n H_n$ a hyperfinitizing edge-cut for $E$ because $E \setminus H = \bigincU_n (E_n \setminus H)$ and each $E_n \setminus H$ is a component-finite graph. Taking a lift $\lift{\mu}$ to $E$ of the measure $\mu$, we can make $\lift{\mu}(H)$ arbitrarily small by taking the $H_n$ so that $\sum_n \lift{\mu}(H_n)$ is sufficiently small, hence $\hep{\lift{\mu}}(E) = 0$.
\end{remark}

\section{Packed and approximately saturated tilings}\label{sec:packing_and_saturation}

Throughout this section, let $X$ be a standard Borel space, $E$ be a countable Borel equivalence relation on $X$, and $\coc : E \to \R^+$ a Borel cocycle. Within a given Borel collection $\SC$ of $\coc$-finite subsets of $X$, we will build Borel tilings of various degrees of maximality.

When $\coc \equiv 1$, i.e. in the pmp setting, all of the results below are simpler and are proven in \cite{Miller-Ts:erg_hyp_dec}*{Section 4}.

\subsection{Tilings}

For any $\SC \subseteq \Pow(X)$, we put $\dom(\SC) \defeq \bigcup \SC$ and call it the \emph{domain} of $\SC$. For $Y \subseteq X$, put 
\[
\SC \rest{Y} \defeq \set{U \in \SC : U \subseteq Y}.
\]
We call $\SC$ a \textit{tiling} if the sets in $\SC$ are pairwise disjoint. In this case, we refer to the sets in $\SC$ as \textit{tiles} and denote by $E(\SC)$ the equivalence relation that is the identity outside of $\dom(\SC)$ and on $\dom(\SC)$ its classes are exactly the tiles in $\SC$.

We say that a tiling $\PC \subseteq \Pow(X)$ is \textit{maximal} within $\SC \subseteq \Pow(X)$, if no set in $\SC$ is disjoint from $\dom(P)$.

\begin{defn}\label{defn:prepart_and_properties}
	Let $\SC \subseteq \FinrE$.
	\begin{itemize}
	    \item Let $\^\SC \subseteq \FinrE$ denote the \textit{upward closure} of $\SC$, i.e. the collection of all $\coc$-finite sets that are countable increasing unions of sets in $\SC$. Call $\SC$ \textit{upward closed} if $\^\SC = \SC$.
	
		\item Say that $\SC$ is \textit{finitely based} if for any $A \in \SC$ and any finite $B \subseteq A$ there is a finite $A' \in \SC$ with $B \subseteq A' \subseteq A$. This immediately implies a stronger statement that for any $\e > 0$, $A'$ can be taken such that $\coc(A') \ge (1 - \e) \coc(A)$.
		
		\item We say that $\SC$ has \textit{finitely based quotients} if for each Borel tiling $\PC \subseteq \SC$, the quotient $\SCmod{E(\PC)}$ is finitely based.
	\end{itemize}
\end{defn}

When $\SC \subseteq \FinrE$ is Borel, $\dom(\SC)$ is analytic (hence measurable) in general. However:

\begin{lemma}\label{dom_Borel}
	If a Borel collection $\SC \subseteq \FinrE$ is either finitely based or a tiling, then $\dom(\SC)$ is Borel.
\end{lemma}
\begin{proof}
	In the first case, $\dom(\SC) = \dom(\SC \cap \FinE)$ and every $x \in \dom(\SC \cap \FinE)$ is contained in only countably-many sets in $\SC \cap \FinE$. In the second case, every $x \in \dom(\SC)$ is contained in exactly one set in $\SC$. Thus, in either case, \cref{membership_is_Borel} and Luzin--Novikov uniformization imply that $\dom(\SC)$ is Borel.
\end{proof}

Let $\PC, \PC' \subseteq \Pow(X)$ be tilings. We say that $\PC'$ is a \emph{partial extension} of $\PC$, noted $\PC' \gtrsim \PC$, if each set in $\PC'$ is $E(\PC)$-invariant. If moreover, $\dom(\PC') \supseteq \dom(\PC)$, we say that $\PC'$ is an \emph{extension} of $\PC$ and write $\PC' \ge \PC$. 

For $N \le \infty$, we call a sequence $(\PC_n)_{n < N}$ of tilings \emph{coherent} if $\PC_i \lesssim \PC_j$ for all $i \le j$. In this case, observe that although the $E(\PC_n)$ are not increasing, $\Id_X \cup \bigcup_{n < m} E(\PC_n)$ is an equivalence relation for each $m \le N$. It thus makes sense to define 
\[
\bigvee_{n < N} \PC_n
\]
as the collection of all classes of the equivalence relation $\Id_X \cup \bigcup_{n < N} E(\PC_n)$ that are contained in $\bigcup_{n < N} \dom(\PC_n)$; in particular, $\bigvee_{n < 0} \PC_n = \0$. When $N = \infty$, we say that the sequence $(\PC_n)_{n < \infty}$ \emph{stabilizes} if $\bigvee_{n < \infty} \PC_n \subseteq \bigcup_{n < \infty} \PC_n$.

\subsection{Packed tilings}

\begin{defn}\label{defn:packed}
	Let $p \in \R^+$.
	
	\begin{itemize}
		\item For a tiling $\PC \subseteq \FinrE$, call a set $A \in \FinrE$ a \emph{$p$-pack} over $\PC$ if $A$ is $E(\PC)$-invariant and $\coc\big(A \setminus \dom(\PC)\big) \ge p \cdot \coc\big(\dom(\PC) \cap A\big)$, equivalently, $\coc(A) \ge (1 + p) \cdot \coc\big(\dom(\PC) \cap A\big)$.
		
		\item For a collection $\SC \subseteq \FinrE$, call a tiling $\PC \subseteq \FinrE$ \emph{$p$-packed} within $\SC$ if $\SC$ has no $p$-pack over $\PC$.
		
		\item A sequence $(\PC_n)_{n \in \N}$ of tilings contained in $\FinrE$ is \textit{$p$-packing} if it is coherent and for all $n \in \N$, each tile in $\PC_n \setminus \left( \bigvee_{m < n} \PC_m \right)$ is a $p$-pack over $\bigvee_{m < n} \PC_m$.
	\end{itemize}
\end{defn}

\begin{obs}\label{packed_properties}
	Let $p \in \R^+$, $\SC \subseteq \FinrE$, and let $\PC \subseteq \FinrE$ be a $p$-packed tiling within $\SC$. 
	\begin{enumerate}[(a)]
	    \item\label{packed_properties:maximal} $\PC$ is maximal within $\SC$.
	    
	    \item\label{packed_properties:extensions} Any tiling $\PC' \subseteq \FinrE$ with $\PC' \ge \PC$ is still $p$-packed within $\SC$.
	\end{enumerate}
\end{obs}

\begin{lemma}\label{packing_stabilizes}
	For any $p \in \R^+$, any $p$-packing sequence $(\PC_n)$ of Borel tilings within $\FinrE$ stabilizes modulo $\coc$-deficient.
\end{lemma}
\begin{proof}
	We assume without loss of generality that $p \le 1$. We also assume without loss of generality that $(\PC_n)$ is extension-increasing by replacing each $\PC_n$ with $\bigvee_{m \le n} \PC_m$.
	
	Put $E_n \defeq E(\PC_n)$ for each $n \in \N$, and $E_\infty \defeq \bigincU_n E_n$. Letting $Z$ be the union of all $E_\infty$-classes $C$ with $C \notin \bigcup_n \PC_n$, it is enough to build a $\coc$-bounded Borel transport in $E_\infty \rest{Z}$ disbalancing $Z$. Thus, we assume without of generality that $X = Z$, so for every $n \in \N$, each $E_n$-class is strictly contained in an $E_\infty$-class.
	
	For each $n \in \N$, putting $D_n \defeq \dom(\PC_n)$ and $D_{-1} \defeq \0$, let $\PC_n' \defeq \set{P \in \PC_n : P \cap D_{n-1} = \0}$. The domains $D_n' \defeq \dom(\PC_n')$ are pairwise disjoint and the set $D' \defeq \bigcup_n D_n'$ is an $E_\infty$-complete section. We build a $\coc$-bounded transport $\phi$ in $E_\infty$ with no sinks but with $\Sources(\phi) = D'$, thus disbalancing $X$. 
	
	To get an intuitive idea, take $U \in \bigcup_n \PC_n'$ and let $n \in \N$ be the least such that $V \defeq [U]_{E_n} \supsetneq U$. The packing condition ensures that $V \setminus D_{n-1}$ has at least as much mass as the \onum{p} fraction of the mass of $V \cap D_{n-1}$, so we can transport the \onum{p} fraction of the mass of $U$ into $V \setminus D_{n-1}$. Maybe other $E_{n-1}$-classes in $V \cap D_{n-1}$ decide to do the same, but that is still alright since $V \setminus D_{n-1}$ has enough mass to accommodate the \onum{p} fraction of the mass of all of $V \cap D_{n-1}$. We repeat this with $V$ instead of $U$, to drain the sinks in $V$ into new points that join the equivalence class of $V$ in a later stage. Thus, sinks move to infinity, while $U$ stays a source.
	
	Formally, we recursively define a sequence $(\phi_n)$ of $\coc$-bounded Borel transports with pairwise disjoint domains, and take $\phi \defeq \sum_n \phi_n$. For each $n \ge 1$, putting $\phib_n \defeq \sum_{k < n} \phi_k$, we ensure that
	\begin{enumerate}[(i),widest=ii),series=packing_stabilizes]
		\item \label{packing_stabilizes:domain} $\dom(\phib_n) \subseteq E_n \cap \big(D_{n-1} \times (D_n \setminus D')\big)$;
		
		\item \label{packing_stabilizes:sources} $\Sources(\phib_n) = \bigcup_{k < n} D_k'$ and the net of $\phib_n$ on $\Sources(\phib_n)$ is constant $-p$;
		
		\item \label{packing_stabilizes:sinks} $\Sinks(\phib_n)$ is disjoint from $\bigcup_{k < n} U_k$, where $U_k \defeq X_k \cap D_k$ and
		\[
		X_k \defeq \set{x \in D_{k+1} : [x]_{E_{k+1}} \cap D_k \ne \0 \text{ and } [x]_{E_{k+1}} \setminus D_k \ne \0}.
		\]
	\end{enumerate}
	Condition \labelcref{packing_stabilizes:sources} guarantees that $\Sources(\phi) = D'$, and \labelcref{packing_stabilizes:sinks} implies that $\Sinks(\phi) = \0$ because $X = \bigcup_n U_n$ (each $E_n$-class is strictly contained in an $E_m$-class for some $m > n$). It remains to construct such a sequence $(\phi_n)$. 
	
	Fixing $n \in \N$, suppose that $\phib_n$ is defined and satisfies \labelcref{packing_stabilizes:domain}--\labelcref{packing_stabilizes:sinks}. Putting $V_n \defeq X_n \setminus D_n$, so $X_n = U_n \sqcup V_n$, we will define $\phi_n$ only on $E_{n+1} \cap (U_n \times V_n)$, so $\dom(\phi_n)$ is disjoint from $\dom(\phib_n)$ by \labelcref{packing_stabilizes:domain} for $\phib_n$. Moreover, because $V_n \cap D' = \0$, condition \labelcref{packing_stabilizes:domain} will also hold for $\phib_{n+1}$.
	
	By \labelcref{packing_stabilizes:domain} for $\phib_n$, $D_n'$ is disjoint from $\Sinks(\phib_n)$, so we define $\p_n : U_n \to [0, 1]$ by 
	\[
	\p_n(u) \defeq 
	\begin{cases}
	p & \text{if $u \in D_n'$}
	\\
	\infl[\phib_n](u) - \outfl[\phib_n](u) & \text{if $u \in \Sinks(\phib_n)$}
	\\
	0 & \text{otherwise}.
	\end{cases}
	\]
	
	\begin{claim*}
	$\int_{U_n \cap Y} \p_n d\coc \le \coc(V_n \cap Y)$ for each $E_{n+1}$-class $Y \subseteq X_n$.
	\end{claim*}
	\begin{pf}
	Fix an $E_n$-class $P \subseteq U_n$. If $P \cap D_n' \ne \0$, then $P \subseteq D_n'$ because $D_n'$ is $E_n$-invariant, and $\int_{P} \p_n d\coc = p \cdot \coc(P)$. If $P \cap D_n' = \0$, then
	\begin{align*}
	\int_P \p_n d\coc &= \int_{P \cap \Sinks(\phib_n)} (\infl[\phib_n] - \outfl[\phib_n]) \,d\coc
	\\
	\eqcomment{\cref{in-flow=out-flow}}
	&=
	- \int_{P \cap \Sources(\phib_n)} (\infl[\phib_n] - \outfl[\phib_n]) \,d\coc 
	\\
	\eqcomment{\labelcref{packing_stabilizes:sources} for $\phib_n$}
	&= p \cdot \coc(P \cap \Sources(\phib_n))
	\le 
	p \cdot \coc(P).
	\end{align*}
	Thus, for each $E_{n+1}$-class $Y \subseteq X_n$, $\int_{U_n \cap Y} \p_n d\coc \le p \cdot \coc(U_n \cap Y) \le \coc(V_n \cap Y)$, where the last inequality is because $Y$ is a $p$-pack over $\PC_n$. 
	\end{pf}
	Hence, \cref{flow-definer} applies to $F \defeq E_{n+1} \rest{X_n}$, $U \defeq U_n$, $V \defeq V_n$, and $\p \defeq \p_n$, yielding a transport $\phi_n$ with $\dom(\phi_n) \subseteq E_{n+1} \cap (U_n \times V_n)$ and $\outfl[\phi_n] \rest{U_n} \equiv \p_n$. In particular, $\Sources(\phib_{n+1}) = D_n' \cup \Sources(\phib_n)$, so \labelcref{packing_stabilizes:sources} is satisfied for $\phib_{n+1}$. Furthermore, $\Sinks(\phib_{n+1}) \cap X_n = \Sinks(\phi_n) \subseteq V_n$, so \labelcref{packing_stabilizes:sinks} holds too.
\end{proof}

\begin{theorem}\label{existence-of-packed}
	For any countable Borel equivalence relation $E$ on $X$, $p \in \R^+$, and a Borel $\SC \subseteq \FinE$, there is a Borel tiling $\PC \subseteq \SC$ that is $p$-packed within $\SC$ modulo $\coc$-deficient.
\end{theorem}
\begin{proof}
	By \cite{Kechris-Miller}*{Proof of Lemma 7.3}, the intersection graph on $\FinE$ admits a countable Borel coloring and we fix one so that the colors are natural numbers.
	
	We recursively build an extension-increasing $p$-packing sequence $(\PC_n)$ of Borel tilings contained in $\SC$. Take $\PC_0 \defeq \0$ and fixing $n \in \N$, suppose that $\PC_n$ is defined. Let $\PC_n'$ be the collection of all sets in $\SC$ of color $n$ that are $p$-packs over $\PC_n$ and let $\PC_{n+1} \defeq \PC_n' \cup (\PC_n \rest{X \setminus \dom(\PC_n')})$.
	
	By \cref{packing_stabilizes}, we may assume that the sequence $(\PC_n)$ stabilizes, so $\PC \defeq \bigvee_n \PC_n \subseteq \SC$ and it remains to show that $\PC$ is $p$-packed within $\SC$. Suppose towards a contradiction that $U \in \SC$ is a $p$-pack over $\PC$. Then $U$ is a $p$-pack over $\PC_n$ for every $n \in \N$. Letting $n$ be the color of $U$, the construction puts $U$ into $\PC_{n+1}$, a contradiction.
\end{proof}

\begin{lemma}\label{finitely-based=>more_packed}
	Let $\SC \subseteq \FinrE$ be a Borel collection with finitely based quotients and let $p \in \R^+$. For any Borel tiling $\PC \subseteq \SC' \defeq \SC \cap \FinE$, if $\PC$ is $p$-packed within $\SC'$, then $\PC$ is $p'$-packed within $\SC$ for any $p' > p$.
\end{lemma}
\begin{proof}
	Let $p' > p$ and suppose towards a contradiction that there is a $p'$-pack $A \in \SC$ over $\PC$. Taking $\e \in (0,1)$ such that $(1 - \e) (p' + 1) \ge (p + 1)$, the finite basedness of $\SCmod{E(\PC)}$ yields an $E(\PC)$-invariant set $A' \in \SC'$ such that $A' \subseteq A$ and $\coc(A') \ge (1 - \e) \coc(A)$. But then
	\[
	\coc(A') \ge (1 - \e) \cdot \coc(A) \ge (1 - \e) (p' + 1) \cdot \coc\big(A \cap \dom(\PC)\big) \ge (p + 1) \cdot \coc\big(A' \cap \dom(\PC)\big),
	\]
	so $A'$ is a $p$-pack over $\PC$, contradicting the $p$-packedness of $\PC$ within $\SC'$.
\end{proof}

\subsection{Approximately saturated tilings}

\begin{defn}
	\begin{itemize}
		\item For a tiling $\PC \subseteq \FinrE$, call a set $A \in \FinrE$ \emph{injective} over $\PC$ if $A$ is $E(\PC)$-invariant and contains at most one tile from $\PC$.
		
		\item A sequence $(\PC_n)$ of tilings is \textit{injective} if it is extension-increasing and for all $n \in \N$, each tile in $\PC_{n+1}$ is injective over $\PC_n$.
	\end{itemize}
\end{defn}

For a collection $\SC \subseteq \FinE$ and a tiling $\PC \subseteq \SC$, a notion of saturation was defined in \cite{Miller-Ts:erg_hyp_dec}*{Section 4.C} as follows: $\PC$ is \textit{saturated} within $\SC$ if there is no set in $\SC$ that is injective over $\PC$. It was also shown in \cite{Miller-Ts:erg_hyp_dec}*{Section 4.C} that if $\SC$ is Borel, then a saturated Borel tiling exists off of an $E$-invariant compressible Borel set. We would like to have the same statement, but with compressible replaced with $\coc$-deficient. However, this is not true in general: let $X \defeq \N$, $E \defeq \N^2$, $\coc_n(m) \defeq 2^{n-m}$, and $\SC \defeq \set{U \in \FinE : 0 \in U}$. Thus, we define an approximate notion of saturation for a given cocycle $\coc$, which coincides with the original notion when $\coc \equiv 1$, and prove existence modulo $\coc$-deficient.

To motivate this new definition, we rephrase the original definition: $\PC$ is saturated within $\SC$ if and only if $\PC \subseteq \SC$ and there is no $Q \in \FinE$ disjoint from $\dom(\PC)$ such that either $Q \in \SC$ or there is $x \in \dom(\PC)$ such that $[x]_{E(\PC)} \cup Q \in \SC$.

\begin{defn}\label{defn:saturated}
For a collection $\SC \subseteq \FinrE$, call a tiling $\PC \subseteq \FinrE$ \emph{approximately saturated} within $\SC$ if $\PC = \bigvee_n \PC_n$ for some injective sequence $(\PC_n)_{n \in \N}$ of tilings contained in $\SC$ with the following property: there is no $Q \in \FinrE$ disjoint from $\dom(P)$ such that either $Q \in \SC$ or there is $x \in \dom(\PC)$ such that $[x]_{E(\PC_n)} \cup Q \in \SC$ for all large enough $n \in \N$. We refer to $(\PC_n)$ as a \textit{saturating $\SC$-approximation} for $\PC$.
\end{defn}

When $\coc \equiv 1$, approximately saturated is indeed the same as saturated because any saturating $\SC$-approximation stabilizes.

\begin{obs}\label{saturated_properties}
Let $\PC \subseteq \FinrE$ be a tiling approximately saturated within $\SC \subseteq \FinrE$.
\begin{enumerate}[(a)]
    \item \label{saturated-are-max} $\PC$ is maximal within $\SC$.
    
    \item \label{saturated-contained-in-upward-closure} $\PC \subseteq \^\SC$.
\end{enumerate}
\end{obs}

To build approximately saturated tilings, we first need the following:

\begin{lemma}\label{injective_is_coc-finite}
	For an injective sequence $(\PC_n)$ of Borel tilings $\PC_n \subseteq \FinrE$, the $\coc$-infinite part of $E_\infty \defeq \bigincU_n E(\PC_n)$ is $\coc$-deficient.
\end{lemma}
\begin{proof}
    Letting $Z$ be the union of $\coc$-infinite $E_\infty$-classes, it is enough to build a $\coc$-bounded Borel transport in $E_\infty \rest{Z}$ disbalancing $Z$. Thus, we assume without of generality that $X = Z$, so $E_\infty$ is $\coc$-infinite. Let $E_n \defeq E(\PC_n)$ for each $n \in \N$.
    
    Let $D$ be the union of inclusion-minimal sets in $\bigcup_n \PC_n$; in particular, $[D]_{E_\w} = X$. We define $k_m : D \to \N$ by induction on $m$ as follows. For each $x \in D$, let $k_0(x)$ be the least $k \in \N$ with $[x]_{E_k} \in \PC_k$ and $[x]_{E_k} \subseteq D$. Supposing that $k_m(x)$ is defined, let $k_{m+1}(x)$ be the least $k > k_m(x)$ such that $\coc([x]_{E_k} \setminus [x]_{E_{k_m(x)}}) \ge \coc([x]_{E_{k_m(x)}})$. For each $m \ge 0$, $\QC_m \defeq \set{[x]_{E_{k_m(x)}} : x \in D}$ is a Borel tiling, and $(\QC_m)$ is an injective sequence. Moreover, for each $Q \in \QC_{m+1}$, $Q = [x]_{E_{k_{m+1}}(x)}$ for some $x \in D$ and
    \[
    \coc(Q \setminus \dom(\QC_m)) = \coc([x]_{E_{k_{m+1}}(x)} \setminus [x]_{E_{k_m(x)}}) \ge \coc([x]_{E_{k_m(x)}}) = \coc(Q \cap \dom(\QC_m)),
    \]
    so $Q$ is a $1$-pack over $\QC_m$. Thus, $(\QC_m)$ is a $1$-packing sequence, so by \cref{packing_stabilizes}, $(\QC_m)$ stabilizes modulo $\coc$-deficient. But for each $x \in D$, $\bigcup_m [x]_{E(\QC_m)} = [x]_{E_\infty}$ is $\coc$-infinite, so $[D]_{E_\infty}$ is $\coc$-deficient, and $X = [D]_{E_\infty}$.
\end{proof}

When building an approximately saturated tiling (in \cref{existence-of-saturated}), we need label-maximizing Borel maximal matchings in labeled Borel bipartite graphs. We prove their existence now:

\begin{lemma}\label{maximal_labeled_matchings}
Let $Y,Z$ be disjoint standard Borel spaces, and let $H \subseteq Y \times Z$ be a locally countable directed Borel graph and $\ell : H \to \R^+$ be a Borel function. There is a Borel maximal matching\footnote{A set of edges $M$ is a \textit{matching} if no two edges are adjacent (ignoring the direction of the edges).} $M \subseteq H$ such that for each edge $(y,z) \in M$,
\[
\ell(y,z) \ge 1 \straighttext{ or } \ell(y,z) \ge \frac{1}{2} \sup_{z' \in H_y \setminus Z_M} \ell(y,z'), 
\]
where $Z_M$ is the set of $M$-matched\footnote{A vertex is \textit{$M$-matched} if there is an edge in $M$ incident to it.} points in $Z$.
\end{lemma}
\begin{proof}
    Because $E_H$ is a countable Borel equivalence relation on $Y \cup Z$, \cite{Kechris-Miller}*{Lemma 7.3} implies that $H$, as well as its Borel subgraphs, admit Borel maximal matchings. We recursively define a sequence $(M_n)_{n \in \N}$ of pairwise \textit{vertex-disjoint} Borel matchings, i.e. no vertex in $Y \cup Z$ is $M_k$-matched and $M_n$-matched for $k \ne n$, so the union $M \defeq \bigcup_n M_n$ is a Borel matching.
    
    Let $H_0 \defeq \set{(y,z) \in H : \ell(y,z) \ge 1}$, and let $M_0 \subseteq H_0$ be a Borel maximal matching. Suppose that the matchings $M_k$, $0 \le k < n$, are defined and are pairwise vertex-disjoint, so $\^M_n \defeq \bigcup_{k<n} M_k$ is a matching, and let $Y_{\^M_n},Z_{\^M_n}$ denote the subsets of $Y$, $Z$, respectively, of all points that are $\^M_n$-matched. Let
    \[
    H_n \defeq \set{(y,z) \in H \cap (Y \setminus Y_{\^M_n}) \times (Z \setminus Z_{\^M_n}) : \ell(y,z) \ge \frac{1}{2} \sup\set{\ell(y,z') : z' \in H_y \setminus Z_{\^M_n}}},
    \]
    and let $M_n \subseteq H_n$ be a Borel maximal matching. This completes the construction of $(M_n)$. 
    It is immediate from the definitions that for any $(y,z) \in M$, either $(y,z) \in M_0$, so $\ell(y,z) \ge 1$, or $(y,z) \in M_n$ for some $n \ge 1$, so $\ell(y,z) \ge \frac{1}{2} \sup_{z' \in H_y \setminus Z_{\^M_n}} \ell(y,z') \ge \frac{1}{2} \sup_{z' \in H_y \setminus Z_M} \ell(y,z')$. 
    
    It remains to show that $M$ is maximal. We take a $y \in Y$ that it is not $M$-matched and aim to show that $H_y \setminus Z_M = \0$, where $H_y \defeq \set{z' \in Z : (y,z') \in H}$. By the maximality of $M_0$, $s_0 \defeq \sup_{z \in H_y \setminus Z_{\^M_1}} \ell(y,z) \le 1$. Then the maximality of each $M_k$ implies, by induction on $n$, that $s_n \defeq \sup\set{\ell(y,z) : z \in Z \setminus Z_{\^M_{n+1}}} \le \frac{1}{2} s_{n-1} \le 2^{-n}$ for each $n \ge 1$. Because $H_y \setminus Z_M = \bigcup_{n \ge 1} (H_y \setminus Z_{\^M_n})$, $\sup_{z \in H_y \setminus Z_M} \ell(y,z) \le \inf_{n \ge 0} s_n = 0$, so $H_y \setminus Z_M = \0$.
\end{proof}

\begin{theorem}\label{existence-of-saturated}
    Any Borel $\SC \subseteq \FinE$ admits a Borel tiling $\PC \subseteq \^\SC$ approximately saturated within $\SC$ modulo $\coc$-deficient. In fact, given a Borel tiling $\PC_0 \subseteq \SC$, the tiling $\PC$ can be chosen so that it admits a saturating $\SC$-approximation starting with $\PC_0$.
\end{theorem}
\begin{proof}
	Fix a Borel coloring of the intersection graph on $\FinE$ with natural numbers \cite{Kechris-Miller}*{Proof of Lemma 7.3}, and let $(k_n)$ be a sequence of natural numbers in which each (color) $k \in \N$ appears infinitely-many times.
	
	Having $\PC_0$ given (otherwise, take $\PC_0 \defeq \0$), we recursively build an injective sequence $(\PC_n)_{n \ge 1}$ of Borel tilings contained in $\SC$. Fixing $n \in \N$, suppose that $\PC_n$ is defined. Let 
	
	\begin{itemize}
	    \item $\SC_n^\perp \defeq $ the collection of all sets in $\SC$ of color $k_n$ that are disjoint from $\dom(\PC_n)$;
	    
	    \item $\QC_n \defeq$ the collection of all $Q \in \FinE \setminus \SC$ of color $k_n$ that are disjoint from $\dom(\PC_n)$.
	\end{itemize}
	
	Let $\GF_n \defeq \set{(P,Q) \in \PC_n \times \QC_n : P \cup Q \in \SC}$. This graph is locally countable because $P \cup Q \in \SC$ implies that $[P]_E = [Q]_E$, and it is Borel (using Luzin--Novikov uniformization), so by \cref{maximal_labeled_matchings} applied to $\GF_n$ with $\ell(P,Q) \defeq \frac{\coc(Q)}{\coc(P)}$, there is a Borel maximal matching $\MF_n \subseteq \GF_n$ such that, for each $(P,Q) \in \MF_n$,
	\begin{equation}\label{eq:max-label}
    \coc(Q) \ge \coc(P) \straighttext{ or } \coc(Q) \ge \frac{1}{2} \sup\set{ \coc(Q') : Q' \in (\GF_n)_P \cap \QC_n'},
    \end{equation}
    where $\QC_n'$ denotes set of all $\MF_n$-unmatched vertices in $\QC_n$. Note that because $\PC_n, \QC_n$ are tilings with disjoint domains and $\MF_n$ is a matching, $\MC_n \defeq \set{P \cup Q : (P,Q) \in \MF_n}$ is a tiling as well. Take
	$
	\PC_{n+1} \defeq \SC_n^\perp \cup \MC_n \cup (\PC_n \rest{D}),
	$
	where $D \defeq \dom(\PC_n) \setminus \dom(\SC_n^\perp \cup \MC_n)$.
	
	Applying \cref{injective_is_coc-finite} to $(\PC_n)$ and ignoring a $\coc$-deficient $E$-invariant Borel set, we assume that all the tiles in $\PC \defeq \bigvee_n \PC_n$ are $\coc$-finite. Thus, $\PC \subseteq \^\SC$.
	
	It remains to show that $(\PC_n)$ is a saturating $\SC$-approximation for $\PC$. Let $Q' \in \FinE$ be disjoint from $\dom(\PC)$. If $Q' \in \SC$, then for any $n \in \N$ such that $k_n$ is the color of $Q'$, the construction would put $Q'$ in $\PC_{n+1}$, contradicting $Q' \cap \dom(\PC) = \0$. Thus, $Q' \notin \SC$. Suppose towards a contradiction that there is $x \in X$ such that $P_n \cup Q' \in \SC$ for all $n \ge N$, for some $N \in \N$, where $P_n \defeq [x]_{E(\PC_n)}$. Let $n \ge N$ be large enough so that
	\begin{equation}\label{eq:saturated-choice-of-n}
	\coc([x]_{E(\PC)} \setminus P_n) < \min \set{\coc(P_n), \tfrac{1}{2}\coc(Q')},
	\end{equation}
	and moreover, choose $n$ so that $k_n$ is the color of $Q'$. Then $(P_n,Q') \in \GF_n$ and yet $Q'$ is not $\MF_n$-matched because all $\MF_n$-matched tiles are contained in $\dom(\PC_{n+1})$, while $Q' \cap \dom(\PC_{n+1}) = \0$. By the maximality of the matching $\MF_n$, $P_n$ must be $\MF_n$-matched with some $Q \in \QC_n$. By \labelcref{eq:max-label}, $\coc(Q) \ge \coc(P_n)$ or $\coc(Q) \ge \frac{1}{2} \coc(Q')$. But $[x]_{E(\PC)} \setminus P_n \supseteq Q$, so $\coc([x]_{E(\PC)} \setminus P_n) \ge \coc(P_n)$ or $\coc([x]_{E(\PC)} \setminus P_n) \ge \tfrac{1}{2}\coc(Q')$, contradicting \labelcref{eq:saturated-choice-of-n}.
\end{proof}

\subsection{Packed and approximately saturated tilings}

\begin{prop}\label{packed-and-saturating-approximation}
	For any Borel collection $\SC \subseteq \FinE$ and $p \in \R^+$, there is a Borel tiling $\PC \subseteq \^\SC$ such that, modulo $\coc$-deficient, it is approximately saturated within $\SC$ and moreover, admits a saturating $\SC$-approximation of $p$-packed tilings within $\SC$.
\end{prop}
\begin{proof}
    Applying \cref{existence-of-packed} and throwing out an $\coc$-deficient set, we get a Borel tiling $\PC_0 \subseteq \SC$ $p$-packed within $\SC$. In particular, any extension $\PC_0' \subseteq \SC$ of $\PC_0$ is $p$-packed within $\SC$ by \cref{packed_properties}\labelcref{packed_properties:extensions}. Applying \cref{existence-of-saturated} to $\SC$ with the initial tiling $\PC_0$ gives a Borel tiling $\PC \subseteq \^\SC$ with a desired saturating $\SC$-approximation.
\end{proof}

\begin{theorem}\label{existence-of-packed-and-saturated}
For any Borel collection $\SC \subseteq \FinrE$ with finitely based quotients and $p \in \R^+$, there is a Borel tiling $\PC \subseteq \^\SC$ that is $p$-packed within $\SC$ and approximately saturated within $\SC' \defeq \SC \cap \FinE$, modulo $\coc$-deficient.
\end{theorem}
\begin{proof}
Applying \cref{packed-and-saturating-approximation} to $\SC'$ and $\frac{p}{2}$, we get a Borel tiling $\PC \subseteq \^{\SC'}$ with a saturating $\SC'$-approximation $(\PC_n)$ such that $\PC_0$ is $\frac{p}{2}$-packed within $\SC'$. Because $\SC$ has finitely based quotients, \cref{finitely-based=>more_packed} implies that $\PC_0$ is $p$-packed within $\SC$. By \cref{packed_properties}\labelcref{packed_properties:extensions}, same is true for $\PC$ because $\PC \ge \PC_0$.
\end{proof}

\section{Cocycled graph visibility}\label{sec:cocycle_visibility}

Throughout this section, let $G$ be a locally countable Borel graph on a standard Borel space $X$ and let $\coc : E_G \to \R^+$ be a Borel cocycle.

\subsection{Definitions and basic properties}

In our proofs, we would like to obtain $G$-connected $\coc$-finite sets of arbitrarily large \cocratio. The following notion helps replace $\cocmax$ with $\coc$.

\begin{defn}
	Let $x \in X$.
	\begin{itemize}
		\item A \emph{$(G,\coc)$-visible neighborhood} of $x$ is any $G$-connected set $V \ni x$ such that $\coc(x) \ge \coc(v)$ for each $v \in V$.
		
		\item A point $y \in X$ is said to be \emph{$(G, \coc)$-visible} from $x$, denoted by $y \viscocG x$, if $x$ admits a $(G,\coc)$-visible neighborhood containing $y$. (The intuition comes from thinking of $\coc$ as relative heights, so $y$ is visible from $x$ if on some path from $x$ to $y$, no point is taller than $x$.)
		
		\item We refer to the sets $(\viscocG)^x \defeq \set{y \in X : y \viscocG x}$ and $(\viscocG)_x \defeq \set{y \in X : x \viscocG y}$ as the \textit{$\viscocG$-downward} and \textit{$\viscocG$-upward cones} of $x$.
		
		\item Call a set $C$ a \textit{$\viscocG$-downward} (resp. \textit{upward}) \textit{cone} if $C$ is a $\viscocG$-downward (resp. upward) cone of some $x \in X$.
				
		\item We say that $G$ has \emph{finite $\coc$-visibility} if every $\viscocG$-downward cone is $\coc$-finite.
	\end{itemize}
	
	\noindent We omit $G, \coc$ from the notation if they are understood from the context. Similarly, we just say downward (resp. upward) cone.
\end{defn}

\begin{prop}\label{vis:properties}
	Let $G, \coc$ be as above, and let $x,y \in X$ be points in the same $G$-connected component.
	\begin{enumerate}[(a)]
	    \item \label{item:vis:transitive}\emph{Partial quasi-order:} $\vis$ is a transitive relation (a partial quasi-order).
	    
	    \item \label{item:vis:amalgamation} \emph{Amalgamation:} There is $z \in X$ such that $x \vis z$ and $y \vis z$. In other words, any two downward cones in the same $G$-connected component are contained in a downward cone.
	    
		\item \label{item:vis:linear_on_upward_cones} \emph{Linearity on upward cones:} The quasi-order $\vis$ is linear on each upward cone, i.e. if $x \vis y$ and $x \vis z$, then $y \vis z$ or $z \vis y$.

		\item \label{item:vis:containing_coc-finite} \emph{Covering with downward cones:} Every $U \in \FinrG$ is contained in a downward cone.
		
		\item \label{item:vis:upward-cone-cofinal} \emph{Cofinality of upward cones:} For each $x \in X$, $\sup_{y \in [x]_{E_G}} \coc(y) = \sup_{z \sees x} \coc(z)$.
	\end{enumerate}
\end{prop}
\begin{proof}
    \labelcref{item:vis:transitive} is immediate. 
    For \labelcref{item:vis:amalgamation}, letting $P$ be a $G$-path from $x$ to $y$, we see that $x$ and $y$ are visible from any $z \in \Maxr P$. 
    For \labelcref{item:vis:linear_on_upward_cones}, if $x \vis y$ and $x \vis z$ then $\coc(y) \le \coc(z)$ implies $y \vis z$. 
    %
    %
    For \labelcref{item:vis:containing_coc-finite}, because $\coc(U) < \infty$, $\Maxr U$ is nonempty by \cref{Maxr_nonempty_finite}, so $U \subseteq \dcone{x}$ for any $x \in \Maxr U$. 
    For \labelcref{item:vis:upward-cone-cofinal}, note that for any $y \in [x]_{E_G}$, \labelcref{item:vis:amalgamation} gives a $z \sees x,y$, in particular, $\coc(z) \ge \coc(y)$.
\end{proof}

\subsection{Finite visibility and hyperfiniteness}

Here, we prove the following sufficient condition for hyperfiniteness in terms of a cocycle mentioned in the introduction:

\begin{namedthm*}{\cref{finite_visibility=>hyperfinite}}
	Let $G$ be a locally countable Borel graph on a standard Borel space $X$ and let $\coc : E_G \to \R^+$ be a Borel cocycle. If $G$ has finite $\coc$-visibility, then it is hyperfinite.
\end{namedthm*}
\begin{proof}
This proof was significantly simplified by the suggestion of an anonymous referee.

Suppose that $G$ has finite $\coc$-visibility. By \cref{coc-finite_is_smooth}, the $\coc$-finite part of $E_G$ is smooth, hence hyperfinite, so we assume without loss of generality that $E_G$ is $\coc$-infinite.

\begin{claimlemma}\label{claim:upward_cone_infinite}
    Every upward cone is $\coc$-infinite.
\end{claimlemma}
\begin{pf}
    Suppose towards a contradiction that $U \defeq \ucone{x}$ is $\coc$-finite for some $x \in X$, so $\Maxr U \ne \0$, by \cref{Maxr_nonempty_finite}. But then \cref{vis:properties}\labelcref{item:vis:upward-cone-cofinal} implies that every point in $[x]_{E_G}$ is visible from any $y \in \Maxr U$, contradicting finite visibility because $[x]_{E_G}$ is $\coc$-infinite.
\end{pf}

\begin{claimlemma}\label{claim:Min_nonempty_finite}
    $\vis$ is a well-order on any upward cone $U$. Moreover, for any nonempty $A \subseteq U$, $\Minr A$ is nonempty and finite.
\end{claimlemma}
\begin{pf}
By \cref{vis:properties}\labelcref{item:vis:linear_on_upward_cones}, any two points in $\Minr A$ are visible to each other. Thus, finite visibility implies that $\Minr A$ is $\coc$-finite, and hence finite.

Suppose towards a contradiction that $\Minr A = \0$, so there is a sequence $(x_n) \subseteq A$ such that $\coc(x_n) > \coc(x_{n+1})$. This implies that $x_n \sees x_{n+1}$ by the linearity of $\sees$ on $U$ (\cref{vis:properties}\labelcref{item:vis:linear_on_upward_cones}). But $U = \ucone{x}$ for some $x \in X$, so $\coc(x_n) \ge \coc(x)$ for all $n$, and every $x_n$ is visible from $x_0$, contradicting finite visibility (at $x_0$).
\end{pf}

For each $x \in X$, put $C^+(x) \defeq \ucone{x} \setminus \Minr \ucone{x}$. \cref{claim:Min_nonempty_finite,claim:upward_cone_infinite} imply that $C^+(x) \ne \0$, so fixing a Borel linear order $<$ on $X$, we define a transformation $T : X \to X$ by $x \mapsto$ the $<$-least element of the set $\Minr C^+(x)$, which exists by \cref{claim:Min_nonempty_finite}. Clearly, $x \vis T(x)$ and $\coc(x) < \coc(T(x))$.

\begin{claimlemma}
    For any $E$-related $x,y \in X$, there are $n,m \in \N$ such that $T^n(x) = T^m(y)$.
\end{claimlemma}
\begin{pf}
    By the amalgamation property (\cref{vis:properties}\labelcref{item:vis:amalgamation}), it is enough to show the statement assuming $x \vis y$. By finite visibility (at $y$), there are only finitely-many points $z \in \ucone{x}$ with $\coc(z) \le \coc(y)$. Because the sequence $(\coc_x(T^k(x)))_{k \in \N}$ is strictly increasing and $(T^k(x))_{k \in \N} \subseteq \ucone{x}$, there is a least $n \in \N$ such that $\coc(T^n(x)) > \coc(y)$. Then $n \ge 1$ and $\coc(T^{n-1}(x)) = \coc(y)$, so $T^{n-1}(x)$ and $y$ are visible to each other by \cref{vis:properties}\labelcref{item:vis:linear_on_upward_cones}. But then the upward cones of $T^{n-1}(x)$ and $y$ coincide, so $T^n(x) = T(y)$.
\end{pf}

Thus, the orbit equivalence relation of the (semigroup) action of $T$ on $X$ is exactly $E_G$, so by \cite{DJK}*{Corollary 8.2}, $E_G$ is hyperfinite.
\end{proof}

\subsection{Tiling with sets of large \cocratio}

As before, throughout this subsection, we let $G$ be a locally countable Borel graph on a standard Borel space $X$ and let $\coc : E_G \to \R^+$ be a Borel cocycle. Moreover, we let $\mu$ be a $\coc$-invariant Borel probability measure on $X$. 

The goal of this subsection (\cref{tiling_with_large_visibility}) is to build a Borel tiling $\PC$ with a large domain whose tiles are of large \cocratio and perhaps satisfy other properties. This is not hard in the pmp setting, i.e. when \cocratio is just cardinality; indeed, taking a saturated tiling works. However, typically \cocratio is not equal to cardinality and it is not monotone (under subsets), which makes such a tiling harder to build.

We say that $V \subseteq X$ is a \textit{$(G,\coc)$-visible neighborhood} of $U \in \FinrG$ if $V$ is $G$-connected, contains $U$, and $\coc(U) \ge \coc(v)$ for each $v \in V$. In this case, for a $\coc$-finite $V$,
\begin{equation}\label{eq:vis-nbhd_cocmax=coc}
    \frac{\coc(V)}{\coc(U)} \le \cocmax(V) \le \cocmax(U) \frac{\coc(V)}{\coc(U)}.
\end{equation}
Indeed, $\cocmax(V) = \frac{\coc(V)}{\maxr V}$ and $\frac{\coc(V)}{\coc(U)} \le \frac{\coc(V)}{\maxr V} \le \frac{\coc(V)}{\maxr U} = \cocmax(U) \frac{\coc(V)}{\coc(U)}$.

The following captures the kind of properties we will allow the tiles to satisfy.

\begin{defn}\label{defn:conic}
Call $\SC \subseteq \FinrG$ \textit{$(G,\coc)$-conic} (resp. \textit{finitely $(G,\coc)$-conic}) if for each $U \in \FinrG$ (resp. $U \in \FinG$), all of its $\coc$-large enough $\coc$-finite (resp. finite) $(G,\coc)$-visible neighborhoods are in $\SC$. By \labelcref{eq:vis-nbhd_cocmax=coc}, ``$\coc$-large enough'' can be replaced with ``large enough \cocratio''.
\end{defn}

Note that $(G,\coc)$-conic implies finitely $(G,\coc)$-conic.

\begin{example}
For any $L > 0$, the collection of all $V \in \FinrG$ (resp. $V \in \FinG$) with $\cocmax(V) > L$ is $(G,\coc)$-conic (resp. finitely $(G,\coc)$-conic).
\end{example}

\begin{lemma}\label{maximal_in_conic_is_hyperfinitizing}
    If a Borel tiling $\PC \subseteq \FinrG$ is maximal within a finitely $(G,\coc)$-conic $\SC \subseteq \FinG$, then $G_{- \dom(\PC)}$ has finite $\coc$-visibility. In particular, $\dom(\PC)$ is a hyperfinitizing vertex-cut for $G$.
\end{lemma}
\begin{proof}
The last part is due to \cref{finite_visibility=>hyperfinite}. As for finite $\coc$-visibility, suppose towards a contradiction that there is an $x \in X \setminus \dom(\PC)$ that admits arbitrarily $\coc$-large finite $(G,\coc)$-visible neighborhoods $V \subseteq X \setminus D$. Then a $\coc$-large enough such $V$ belongs to $\SC$, contradicting the maximality of $\PC$.
\end{proof}

The following is the reason why we use $(G,\coc)$-conic and not just finitely $(G,\coc)$-conic.

\begin{prop}\label{quotient_of_conic}
Let $F$ be a $G$-connected $\coc$-finite Borel equivalence relation. If $\SC \subseteq \FinrG$ is $(G,\coc)$-conic, then $\SCmod{F}$ is $(\Gmod{F}, \rmod{F})$-conic.
\end{prop}
\begin{proof}
For any $U \in [\Xmod{F}]^{\rmod{F} < \infty}_{\Gmod{F}}$, the $\pi_F$-preimage of any $(\Gmod{F},\rmod{F})$-visible neighborhood $V$ of $U$ is a $(G,\coc)$-visible neighborhood of $\pi^{-1}(U)$ and $\frac{\coc(\pi_F^{-1}(V))}{\coc(\pi_F^{-1}(U))} = \frac{\rmod{F}(V)}{\rmod{F}(U)}$.
\end{proof}

We are now ready to prove the main result of this subsection.

\begin{lemma}[$\mu$-co-$\e$ tiling]\label{tiling_with_large_visibility}
    If $G$ is $\mu$-nowhere hyperfinite, then for any $(G,\coc)$-conic $\SC \subseteq \FinrG$ and $\e > 0$, there is a Borel tiling $\PC \subseteq \SC$ with a $\mu$-co-$\e$ domain.
\end{lemma}
\begin{proof}
For each $U \in \FinrG$, let $r(U)$ be the infimum of all $r > 0$ such that any $(G,\coc)$-visible neighborhood $V$ of $U$ with $\cocmax(V) > r$ belongs to $\SC$; write $r(x)$ if $U = \set{x}$. For any $G$-connected $\coc$-finite Borel equivalence relation $F$ and $L > 0$, let $\SCmod{F}'(L)$ denote the collection of all finite $P \in \SCmod{F}$ with $\cocmaxF(P) > \max \set{L, r(U)}$ for all $U \in \MaxrF P$. We just write $\SC'(L)$ if $F = \Id_X$.

Note that for any $x' \in \Xmod{F}$, any $(\Gmod{F}, \rmod{F})$-visible neighborhood $V$ of $x'$ with $\cocmaxF(V) > r(x')$ belongs to $\SCmod{F}$ because $\cocmax(\pi_F^{-1}(V)) \ge \cocmaxF(V)$ (\cref{coc-quotient}). We call this the \textit{main property} of $r$. This and the fact that \cocratio[cocycle] can only decrease in the quotient (\cref{coc-quotient}) allow us below to work mod $F$ and assume without loss of generality that $F = \Id_X$.

\begin{claimlemma}\label{SCmod_is_finitely_conic}
$\SCmod{F}'(L)$ is finitely $(\Gmod{F}, \rmod{F})$-conic.
\end{claimlemma}
\begin{pf}
Follows from \cref{quotient_of_conic} and the fact that the additional condition in the definition of $\SCmod{F}'(L)$ is a lower bound on \cocratio[$\rmod{F}$].
\end{pf}

\begin{claimlemma}\label{closure_of_SC}
$\^{\SCmod{F}'(L)} \subseteq \SCmod{F}$.
\end{claimlemma}
\begin{pf}
We work mod $F$, so assume without loss of generality that $F = \Id_X$. Let $Q_m \in \SC'(L)$ be an increasing sequence of sets such that $Q = \bigcup_m Q_m$ is still $\coc$-finite. Then by \cref{incU_finite=>increasing_ratio}, there is $m$ such that $\Maxr Q_m = \Maxr Q$ and hence $\cocmax(Q) \ge \cocmax(Q_m) > r(x)$ for all $x \in \Maxr Q$, so $Q \in \SC$ by the main property of $r$.
\end{pf}

Let $(L_n)_{n \ge 0}$ be an increasing unbounded sequence of positive reals. We recursively define a coherent sequence $(\PC_n)_{n \ge 0}$ of Borel tilings contained in $\FinrG$; in particular, the equivalence relations $F_n \defeq \bigcup_{k<n} E(\PC_k)$ are increasing, where $F_0 \defeq \Id_X$. For $n \ge 0$, suppose that $F_n$ is defined. \cref{existence-of-saturated} applied to $\SCmod{F_n}'(L_n)$ gives a Borel tiling $\QC \subseteq \^{\SCmod{F_n}'(L_n)}$ approximately saturated within $\SCmod{F_n}'(L_n)$ modulo $\rmod{F_n}$-deficient. Putting $\PC_n \defeq \set{\pi_{F_n}^{-1}(Q) : Q \in \QC}$ finishes the construction. 

By \cref{deficient=>no_prob_meas}, if an $\Emod[(E_G)]{F_n}$-invariant Borel $A_n \subseteq \Xmod{F_n}$ is $\rmod{F_n}$-deficient, then it is $\mumod{F_n}$-null, so $\pi_{F_n}^{-1}(A_n)$ is $\mu$-null. Thus, discarding countably-many $E_G$-invariant $\mu$-null sets from $X$, we have that for each $n \ge 0$, $\submod{(\PC_n)}{F_n} \subseteq \^{\SCmod{F_n}'(L_n)}$ is approximately saturated within $\SCmod{F_n}'(L_n)$. By \cref{closure_of_SC}, $\submod{(\PC_n)}{F_n} \subseteq \SCmod{F_n}$, so $\PC_n \subseteq \SC$.

\begin{claimlemma}\label{Dn_are_hyperfinitizing}
	For each $n \in \N$, $D_n \defeq \dom(\PC_n)$ is a hyperfinitizing vertex-cut for $G$.
\end{claimlemma}
\begin{pf}
Because $\submod{(\PC_n)}{F_n}$ is maximal within $\SCmod{F_n}'(L_n)$ (\cref{saturated_properties}\labelcref{saturated-are-max}) and $\SCmod{F_n}'(L_n)$ is finitely $(\Gmod{F},\rmod{F})$-conic (\cref{SCmod_is_finitely_conic}), \cref{maximal_in_conic_is_hyperfinitizing} implies that $\pi_F(D_n)$\allowbreak$= \dom(\submod{(\PC_n)}{F_n})$ is a hyperfinitizing vertex-cut for $\Gmod{F}$, i.e. $(\Gmod{F})_{- \pi_F(D_n)}$ is hyperfinite. But then $G_{-D_n}$ is hyperfinite because $\pi_F$ is a Borel reduction $E_G \to \Emod[(E_G)]{F}$ and hyperfinitness pulls back under Borel reductions \cite{JKL}*{1.3(ii)}.
\end{pf}

Put $D_\w \defeq \limsup_n D_n \defeq \set{x \in X : x \in D_n \text{ for infinitely many $n \in \N$}}$ and $F_\w \defeq \bigcup_n F_n$.

\begin{claimlemma}\label{Dw_invariant}
	$D_\w$ is $E_G$-invariant.
\end{claimlemma}
\begin{pf}
	Suppose not, so there are $G$-adjacent points $x \in D_\w$ and $y \in X \setminus D_\w$. Let $N_0 \in \N$ be large enough so that $y \notin D_n$ for all $n \ge N_0$; hence $[y]_{F_\w} = [y]_{F_n}$ for all $n \ge N_0$. Because $[x]_{F_\w}$ is $\coc$-infinite (by \cref{increasing_ratio=>infinite}), there is $N_1 \ge N_0$ such that $\coc([x]_{F_n}) > \coc([y]_{F_\w})$ for all $n \ge N_1$. Finally, take $n \ge N_1$ so that $[x]_{F_{n+1}} \in \PC_n$. We will contradict that $\QC \defeq \submod{(\PC_n)}{F_n}$ is approximately saturated within $\SCmod{F_n}'(L_n)$. 
	
	Putting $x' \defeq [x]_{F_n}$ and $y' \defeq [y]_{F_n} = [y]_{F_\w}$, so $\coc(x') > \coc(y')$, and working mod $F_n$, we assume without loss of generality that $F_n = \Id_X$ and drop $F_n$ from the notation. 
	
	Let $(\QC_m)$ be a saturating $\SC'(L_n)$-approximation for $\QC$, and put $Q \defeq [x']_{E(\QC)}$ and $Q_m \defeq [x']_{E(\QC_m)}$, so $Q = \bigincU_m Q_m$. Fix any $k \in \N$. Then $Q_k' \defeq Q_k \cup \set{y'}$ is $G$-connected and finite, and $\maxr Q_k \ge \coc(x') > \coc(y')$, so $\Maxr (Q_k') = \Maxr Q_k$. In particular, $\cocmax(Q_k') = \frac{\coc(Q_k) + \coc(y')}{\maxr Q_k} > \cocmax(Q_k) > \max \set{L_n, r(x_k)}$ for all $x_k \in \Maxr Q_k' = \Maxr Q_k$ because $Q_k \in \SC'(L_n)$. But then the main property of $r$ implies that $Q_k' \in \SC$, and hence $Q_k' \in \SC'(L_n)$, contradicting that $(\QC_m)$ is a saturating $\SC'(L_n)$-approximation.
\end{pf}

\begin{claimlemma}
$D_\infty$ is $\mu$-conull.
\end{claimlemma}
\begin{pf}
Let $X' \defeq X \setminus D_\infty$ and suppose that it is not null. Because $X'$ is $G$-invariant and $G$ is $\mu$-nowhere hyperfinite, $G \rest {X'}$ is not $\mu$-hyperfinite, so $p \defeq \hvp(G \rest {X'}) > 0$. By \cref{Dn_are_hyperfinitizing}, $D_n \cap X'$ is a hyperfinitizing cut for $G \rest {X'}$, so $\mu(D_n \cap X') \ge p$ for all $n \in \N$. But then $D_\infty \cap X' = \bigcap_n \bigcup_{m \ge n} (D_m \cap X')$ also has measure at least $p$ by the downward continuity of $\mu$, contradicting $D_\infty \cap X' = \0$.
\end{pf}

In particular, $\bigcup_{n \in \N} D_n$ is conull, so for a large enough $N \in \N$, $D \defeq \bigcup_{n \le N} D_n$ is $\mu$-co-$\e$. Then the set $\PC$ of all $F_N$-classes contained in $D$ is contained in $\bigcup_{n \le N} \PC_n \subseteq \SC$ and $\dom(\PC) = D$ is $\mu$-co-$\e$.
\end{proof}

\section{Visible asymptotic averages on a graph}\label{sec:asymptotic_averages}

If \cref{ptwise_ergodic_for_graphs} is indeed true, then for any $f \in L^\infty(X,\mu)$ and a.e. $x \in X$, there must be arbitrarily $\coc$-large finite $G$-connected sets containing $x$ over which the $\coc$-average of $f$ is arbitrarily close to $\E(f | \BC_{E_G})(x)$. Motivated by this, we look at the set of all reals in general that are achievable in this manner, thus defining a new invariant developed in this section.

Throughout this section, let $G$ be a locally countable (abstract) graph on a set $X$ and let $\coc : E_G \to \R^+$ a cocycle. We also let $f : X \to \R$ be a bounded function.

\subsection{For an abstract graph}

\cite{Miller-Ts:erg_hyp_dec}*{Definition 8.2} introduces the set $\Mean{w}{G} f(x)$ of asymptotic $w$-weighted means along $G$ in the $G$-connected component $[x]_{E_G}$, where $w : X \to \R^+$ is a weight-function. This set is independent of the representative $x$ of the $G$-connected component \cite{Miller-Ts:erg_hyp_dec}*{Proposition 8.3} and it is a closed interval when $w$ is a bounded function \cite{Miller-Ts:erg_hyp_dec}*{Proposition 8.5}. Here, we generalize this definition to arbitrary cocycles on $E_G$.

\begin{defn}\label{defn:cone-averages}
	For a $G$-connected set $C$ and $x \in C$, we call $r \in \R$ a \emph{$(G,\coc)$-asymptotic average} of $f$ at $x$ over $C$ if there are arbitrarily $\coc$-large finite $G$-connected sets $V \subseteq C$ containing $x$ with $\meanrf{V}$ arbitrarily close to $r$; more precisely, for every $\e > 0$ and $L > 0$, there is a finite $G$-connected set $V \subseteq C$ containing $x$ with $\coc_x(V) \ge L$ and $\meanrf{V} \approx_\e r$. We denote by $\MeanrfGrest{C}(x)$ the set of all such $r \in \R$.
\end{defn}

The compactness of the interval $[-\Linf{f},\Linf{f}]$ immediately implies:

\begin{obs}
    If $C$ is a $\coc$-infinite $G$-connected set then $\MeanrfGrest{C}(x) \ne \0$ for each $x \in C$.
\end{obs}

\begin{remark}
	For a $\viscocG$-downward cone $C = \dconecocG{x}$, we note that the restriction of $\coc$ to $C$ is a coboundary, being the differential of the weight-function $\coc_x : C \to \R^+$. Note that $\coc_x$ is bounded above by $1$. It is easy to see that the definition of a $(G, \coc)$-asymptotic average over $C$ as above coincides with that of an asymptotic $\coc_x$-weighted mean along $G \rest{C}$ as defined in \cite{Miller-Ts:erg_hyp_dec}*{Definition 8.2}. Thus, the following three lemmas are just restatements of \cite{Miller-Ts:erg_hyp_dec}*{8.4 and 8.5} in our terms.
\end{remark}

\begin{lemma}\label{invariance_of_cone-averages}
	The function $\MeanrfGrest{C}$ is constant for any $G$-connected set $C \subseteq X$, i.e.~$\MeanrfGrest{C}(x) = \MeanrfGrest{C}(y)$ for any $x,y \in C$.
\end{lemma}
\begin{proof}
	There is a $G$-path connecting $x$ and $y$, whose effect on the averages of $f$ over arbitrarily $\coc$-large sets is arbitrarily small.
\end{proof}

\begin{lemma}[Intermediate value property]\label{intermediate_value_property}
	Let $U,V \in \FinrG$ be such that $U \subseteq V$ and let
	\[
	\De \defeq 2 \Linf{f} \frac{\maxr (V \setminus U)}{\coc(U)}.
	\]
	For every real $r$ between $\meanrf{U}$ and $\meanrf{V}$, there is $W \in \FinrG$ with $U \subseteq W \subseteq V$ and $\meanrf{W} \approx_\De r$.
\end{lemma}
\begin{proof}
    Firstly note that there is $V' \in \FinrG$ with $U \subseteq V' \subseteq V$ such that $V' \setminus U$ is finite, yet $\coc(V') \ge (1-\e)\coc(V)$, where $\e \in (0,1)$ is small enough to guarantee $\meanrf{V'} \approx_\Delta \meanrf{V}$ by \cref{convexity_of_average}\labelcref{item:convexity_of_average:increment_bound}. Now, we can add the vertices of $V' \setminus U$ to $U$ one-by-one, obtaining a finite sequence of $G$-connected supersets of $U$ increasing up to $V'$. It remains to observe that adding one vertex can change the average at most by $\De$ again by \cref{convexity_of_average}\labelcref{item:convexity_of_average:increment_bound}.
\end{proof}

\begin{lemma}\label{averages_on_cones_closed_interval}
	For any $\viscocG$-downward cone $C$ and $x \in C$, $\MeanrfGrest{C}(x)$ is closed and convex; thus, it is a closed subinterval of $\big[-\Linf{f}, \Linf{f}\big]$.
\end{lemma}
\begin{proof}
    Suppose $\MeanrfGrest{C}(x) \ne \0$, which implies that $\coc(C) = \infty$. Then the closedness of $\MeanrfGrest{C}(x)$ follows from the asymptotic nature of the definition of $\MeanrfGrest{C}$. Convexity follows from \cref{intermediate_value_property}. Indeed, for any $r_1,r_2 \in \MeanrfGrest{C}$ and a real $r$ between $r_1,r_2$, there are arbitrarily $\coc$-large $G$-connected subsets $U \ni x$ of $C$ with $\meanrf{U}$ arbitrarily close to $r_1$. For each such $U$, there are arbitrarily $\coc$-large $G$-connected subsets $V \ni x$ of $C$ with $\meanrf{V}$ arbitrarily close to $r_2$ and $\frac{\coc(U)}{\coc(V)}$ arbitrarily small. Thus, we can make $\meanrf{U \cup V}$ arbitrarily close to $\meanrf{V}$, hence we may choose $V$ containing $U$ to begin with. We may also assume that $r$ is between $\meanrf{U}$ and $\meanrf{V}$ because otherwise, $U$ or $V$ would witness the arbitrary closeness of $r$ to $\MeanrfGrest{C}(x)$. Finally, we may apply \cref{intermediate_value_property} to $U, V$ and $r$ and obtain a set $W$ with $U \subseteq W \subseteq V$ with $\meanrf{W} \approx_\Delta r$, where
	\[
	\Delta \defeq 2 \Linf{f} \frac{\maxr (V \setminus U)}{\coc(U)}
	=
	2 \Linf{f} \frac{\maxr[\coc_y] (V \setminus U)}{\coc_y(U)}
	\le
	\frac{2 \Linf{f}}{\coc_y(U)}
	\]
	for some/any $y \in \Maxr C$.\footnote{\label{ftnt:why_visible_averages}Having a bound on $\maxr[\coc_y] (V \setminus U)$ independent of $U, V$ in the proof of \cref{averages_on_cones_closed_interval} is the reason why only visible asymptotic averages are included in $\MeanrGf$ (in \cref{defn:vis-asymp-averages}).} Thus, choosing $U$ arbitrarily $\coc$-large makes $\Delta$ arbitrarily small, hence $r$ arbitrarily close to $\MeanrfGrest{C}$.
\end{proof}

\begin{defn}\label{defn:vis-asymp-averages}
	We call $r \in \R$ a \emph{$(G,\coc)$-visible asymptotic average} of $f$ at $x \in X$ if it is a $(G,\coc)$-asymptotic average of $f$ over some $\viscocG$-downward cone $C \ni x$. We denote by $\MeanrGf(x)$ the set of all $(G,\coc)$-visible\cref{ftnt:why_visible_averages} asymptotic averages of $f$ at $x$.
\end{defn}

\begin{prop}[Invariance]\label{invariance_of_visible-averages}
	The map $x \mapsto \MeanrGf(x)$ is $E_G$-invariant.
\end{prop}
\begin{proof}
	This is immediate from the amalgamation property of downward cones (\cref{vis:properties}\labelcref{item:vis:amalgamation}) and \cref{invariance_of_cone-averages}.
\end{proof}

\begin{prop}[Convexity]\label{convexity_of_visible-averages}
	For any $x \in X$, $\MeanrGf(x)$ is an increasing union of sets of the form $\MeanrfGrest{C}(x)$, where $C$ is a downward cone. In particular, it is a convex subset of $\big[-\Linf{f}, \Linf{f}\big]$.
\end{prop}
\begin{proof}
	By definition, $\MeanrGf(x)$ is a union of sets of the form $\MeanrfGrest{C}(x)$. Because downward cones amalgamate (\cref{vis:properties}\labelcref{item:vis:amalgamation}), this union is directed, so it can be turned into an increasing union, using the countability of $[x]_G$. The convexity of $\MeanrGf(x)$ is then due to \cref{averages_on_cones_closed_interval} and the fact that $C_1 \subseteq C_2$ implies $\MeanrfGrest{C_1} \subseteq \MeanrfGrest{C_2}$.
\end{proof}

\subsection{For a measurable graph}

Equipping $X$ with a standard Borel structure, we now suppose further that $G, \coc$, and $f$ are Borel, and we let $\mu$ be a $\coc$-invariant Borel probability measure on $X$.

Encoding intervals in $\R$ as points in $\big(\set{0,1,2,3} \times \R\big)^2$, where $0,1,2,3$ encode whether each endpoint is open or closed, we equip the set $\IC$ of all intervals with a natural standard Borel structure.

\begin{prop}
	The map $x \mapsto \MeanrGf(x) : X \to \IC$ is Borel.
\end{prop}
\begin{proof}
	The fact that the image is in $\IC$ is by \cref{convexity_of_visible-averages}. The Borelness follows by the definition of visible asymptotic averages, using Luzin--Novikov uniformization, which turns scanning over each $E_G$-class into a natural number quantifier.
\end{proof}

We now obtain a tiling of a $\mu$-co-$\e$ part of the space with tiles over which the $\coc$-average of $f$ is almost in $\MeanrGf$. For $r \in \R$, $A \subseteq \R$, and $\e > 0$, we write $r \in_\e A$ to mean that $\dist(r,A) < \e$. Also, for any $U \in \FinrG$, we write $\MeanrGf(U)$ to mean $\MeanrGf(x)$ for some/any $x \in U$.

\begin{prop}\label{tiling_with_average-approximating_sets}
	Suppose that $G$ is $\mu$-nowhere hyperfinite. For every $\e > 0$, there is a Borel tiling $\PC \subseteq \FinrG$ with a $\mu$-co-$\e$ domain such that for each $P \in \PC$, $\meanrf{P} \in_\e \MeanrGf(P)$.
\end{prop}
\begin{proof}
	Let $\SC$ denote the collection of all $U \in \FinrG$ satisfying $\meanrf{U} \in_\e \MeanrGf(U)$. It remains to show that $\SC$ is $(G,\coc)$-conic because then \cref{tiling_with_large_visibility} gives the desired tiling. To this end, suppose towards a contradiction that some $U \in \FinrG$ admits arbitrarily $\coc$-large visible neighborhoods $V$ that are not in $\SC$, i.e. $\dist(\meanrf{V}, \MeanrGf(U)) \ge \e$. Being visible neighborhoods of $U$, these $V$ are all contained in a single $\viscocG$-downward cone $C \supseteq U$, e.g. $\dconecocG{x}$ for any $x \in \Maxr U$. The compactness of $[-\Linf{f}, \Linf{f}]$ extracts a sequence $(V_n)$ with all $V_n$ contained in $C$ such that $\coc(V_n) \to \infty$ and $\lim_n \meanrf{V_n}$ exists and is outside of $\MeanrGf(U)$, a contradiction.
\end{proof}

\begin{remark}
	For a bounded weight-function, instead of a cocycle, a stronger version of \cref{tiling_with_average-approximating_sets} holds: the requirement of $\mu$-nowhere hyperfiniteness is unnecessary and the tiling in the conclusion has a conull domain. This is proven in \cite{Miller-Ts:erg_hyp_dec}*{8.8} and it is due to the fact that in this case of an actual bounded weight-function instead of a cocycle, the analogue of \cref{tiling_with_large_visibility} is much easier to prove and has a stronger conclusion.
\end{remark}

To illustrate the power of \cref{tiling_with_average-approximating_sets}, we prove the statement that motivated the definition of $\MeanrGf$. However, we do not use it in our proofs below.

\begin{cor}
$\E(f | \BC_{E_G}) \in \^{\MeanrGf}$ a.e.
\end{cor}
\begin{proof}
By subtracting $\E(f | \BC_{E_G})$ from $f$, we may assume that $\E(f | \BC_{E_G}) \equiv 0$. For simplicity of the presentation, suppose that $G$ is ergodic (the proof is uniform through the ergodic decomposition of of $G$ with respect to $\mu$ \cite{Ditzen:thesis}*{Theorem 6 of Chapter 2}), so $\MeanrGf$ is constant a.e. Suppose towards a contradiction that $0$ is $\Delta > 0$ away from $\MeanrGf$. Because $\MeanrGf$ is an interval, it is contained either in $[\Delta, \infty)$ or $(-\infty,-\Delta]$, and we suppose without loss of generality that it is the former. Take $\e > 0$ small enough so that $(\Delta - \e)(1-\e) > \frac{2\Delta}{3}$ and so that $\int_Y |f| d\mu < \frac{\Delta}{3}$ for any $Y \subseteq X$ of measure $\le \e$. Let $\PC$ be a tiling given by \cref{tiling_with_average-approximating_sets} for this $\e$. Putting $F \defeq E(\PC)$ and $D \defeq \dom(\PC)$, the local-global bridge lemma gives a contradiction:
\begin{align*}
\int_X f d\mu 
&= 
\int_{X \setminus D} f d\mu + \int_D f d\mu
\\
\eqcomment{$\mu(X \setminus D) < \e$}
&\ge
-\frac{\Delta}{3} + \int_D f d\mu
\\
\eqcomment{\cref{local-global_bridge}}
&=
-\frac{\Delta}{3} + \int_D \meanrf{F} d\mu
\\
\eqcomment{$\meanrf{F} \in_\e [\Delta, \infty)$ and $D$ is $\mu$-co-$\e$}
&\ge
-\frac{\Delta}{3} + (\Delta - \e) (1-\e) > \frac{\Delta}{3} > 0.\qedhere
\end{align*}
\end{proof}

\section{Proof of Main Lemma \namedthmlabelref{core_lemma}}\label{sec:proof}

This section is entirely devoted to the proof of Main Lemma \namedthmlabelref{core_lemma}. We let $X,\mu,G,\coc,f$, and $\e$ be as in the statement of the theorem. By subtracting $\E(f | \BC_{E_G})$ from $f$, we assume without loss of generality that $\E(f | \BC_{E_G}) \equiv 0$, yet $\Lone{f} > 0$.

We need to find a $G$-connected finite Borel equivalence relation $F$ with $\meanrFf \approx_\e 0$ on a $\mu$-co-$\e$ set. The following allows us to replace finite with $\coc$-finite.

\begin{lemma}\label{coc-finite->finite}
For any $G$-connected $\coc$-finite Borel equivalence relation $F$ and any $\e' > 0$, there is a $G$-connected finite Borel subequivalence relation $F' \subseteq F$ such that $\Meanrf{F'} \approx_{\e'} \meanrFf$ on a $\mu$-co-$\e'$ set.
\end{lemma}
\begin{proof}
    Fix any $\delta > 0$ less than $\e'$ and $\frac{\e'}{2 \Linf{f}}$. Let $X'$ be the set of all $x \in X$ such that $\coc(B_x) < (1 - \delta) \coc([x]_F)$, where $B_x \defeq \set{y \in [x]_F : x <_\coc y}$. By \cref{Maxr_nonempty_finite}, $B_x$ is finite for each $x \in X$, so for each $F$-class $C$, the set $C' \defeq C \cap X'$ is finite; moreover, $\coc(C') \ge (1-\delta) \coc(C)$. However, these $C'$ may not be $G$-connected and we extend them to $G$-connected finite sets as follows. Recalling \cref{cocord_well-order}, we let $X''$ be the set of all $x \in X$ that belong to the $\cocless$-maximum $G$-connected finite subset of $[x]_F$ containing $X' \cap [x]_F$ (the relation $\cocless$ is defined right after \cref{membership_is_Borel}). Now for each $F$-class $C$, the set $C'' \defeq X'' \cap C$ is finite, $G$-connected, and $\coc(C'') \ge (1-\delta)\coc(C)$. Finally, note that both $X'$ and $X''$ are Borel by Luzin--Novikov uniformization, which turns scanning over each $F$-class a natural number quantifier.
    
    Now define the equivalence relation $F'$ by setting $F' \rest{X''} \defeq F \rest{X''}$ and $F' \rest{X \setminus X''} \defeq \Id_{X \setminus X''}$. For each $x \in X''$, $[x]_{F'} = [x]_F \cap X''$, so $\coc([x]_{F'}) \ge (1-\delta) \coc([x]_F)$, which implies $\meanr{[x]_{F'}} f \approx_{\e'} \meanr{[x]_F}$, by \cref{convexity_of_average}\labelcref{item:convexity_of_average:increment_bound}. Moreover, by \cref{local-global_bridge},
    \[
    \mu(X'') = \int_X \1_{X''} d\mu = \int_X \meanr{F} \1_{X''} d\mu \ge \int_X (1-\delta) d\mu = 1 - \delta > 1 - \e',
    \]
    so $F'$ is as desired.
\end{proof}

Thus, without loss of generality (replacing $\e$ with $\frac{\e}{2}$), it is enough to find a $G$-connected $\coc$-finite Borel equivalence relation $F$ with $\meanrFf \approx_\e 0$ on a $\mu$-co-$\e$ set. We call such equivalence relations \textit{good for Main Lemma}.

\begin{assumption}
To simplify notation in the proof, we assume that $G$ is ergodic. In other words, we present the proof restricting to one ergodic component of the ergodic decomposition of $G$ with respect to $\mu$ \cite{Ditzen:thesis}*{Theorem 6 of Chapter 2}. This makes the map $x \mapsto \MeanrGf(x)$ constant a.e.~because it is $E_G$-invariant; thus, we drop $x$ from the notation. However, the proof goes through without this assumption, but all the constants below that depend on $\MeanrGf$ (e.g. $\delta$ in \cref{hypothesis:not_one-sided}) would depend on $x \in X$ (although would still be an $E_G$-invariant functions).
\end{assumption}

For $\de \ge 0$, put $I^-_\delta \defeq (-\infty, \delta]$, $I_\delta \defeq (- \delta, \delta)$, and $I^+_\delta \defeq [\delta, \infty)$.

\subsection{Cutting one side of the set of asymptotic averages}

Note that taking a quotient of $G$ by a $G$-connected $\coc$-finite Borel equivalence relation $F$ results in a graph $\Gmod{F}$ with a more restricted set of visible asymptotic averages, i.e.
\[
\Meanrf{G : F} \defeq \Mean{\rmod{F}}{\Gmod{F}} (\meanrFf) \subseteq \MeanrGf,
\]
where the inclusion follows from \cref{convexity_of_average}\labelcref{item:convexity_of_average:disjU_of_averages}: average of averages is just the average.

Having the ability (\cref{tiling_with_average-approximating_sets}) to build tilings whose each tile $P$ has $\meanrf{P}$ arbitrarily close to $\MeanrGf$, it makes sense to try building a $G$-connected $\coc$-finite Borel equivalence relation $F$ such that quotienting out by it shrinks the set of visible asymptotic averages arbitrarily tightly around $0$. In fact, a weaker conclusion is enough due to the following.

\begin{lemma}[Cutting one side is enough]\label{cutting_one_side}
	For $\delta_\e \defeq \frac{\e^2}{2(\Linf{f} + 1)}$, any $\coc$-finite Borel equivalence relation $F$, and any a sign $\sigma \in \set{+,-}$, if $\meanrFf \in I^\sigma_{\delta_\e}$ with probability $\le \delta_\e$, then $\meanrFf \in I_\e$ with probability $\ge 1-\e$.
\end{lemma}
\begin{proof}
	Fixing an $F$ as in the hypothesis and taking $\sigma \defeq \mathord{-}$ (the proof is symmetric), we suppose that the set $X^-_{\delta_\e} \defeq \set{x \in X : \meanrFf(x) \in I^-_{\delta_\e}}$ has measure $\le {\delta_\e}$ and aim to show that $X_\e \defeq \set{x \in X : \meanrFf(x) \in I_\e}$ has measure $\ge 1 - \e$. It is enough to show that $X^+_\e \defeq \set{x \in X : \meanrFf(x) \in I^+_\e}$ has measure $\le \frac{\e}{2}$ because $X \subseteq X^-_{\delta_\e} \cup X_\e \cup X^+_\e$, so $\mu(X_\e) \ge \mu\big(X \setminus (X^-_{\delta_\e} \cup X^+_\e)\big) \ge 1 - {\delta_\e} - \frac{\e}{2} \ge 1 - \e$. To this end, using \cref{local-global_bridge}, we compute:
	\begin{align*}
		0 = \int_X f d\mu = \int_X \meanrFf d\mu
		&= 
		\int_{X^-_{\delta_\e}} \meanrFf d\mu + \int_{X^+_\e} \meanrFf d\mu + \int_{X \setminus (X^-_{\delta_\e} \cup X^+_\e)} \meanrFf d\mu
		\\
		&\ge
		-\Linf{f} \cdot \mu(X^-_{\delta_\e}) + \e \cdot \mu(X^+_\e) - {\delta_\e} \cdot \mu\big(X \setminus (X^-_{\delta_\e} \cup X^+_\e)\big)
		\\
		&\ge
		-\Linf{f} \cdot {\delta_\e} + \e \cdot \mu(X^+_\e) - \delta_\e
		\\
		&=
		\mu(X^+_\e) \cdot \e - \delta_\e \cdot (\Linf{f} + 1),
	\end{align*}
	so
	$
	\mu(X^+_\e) \le {\delta_\e} \cdot \frac{\Linf{f} + 1}{\e} = \frac{\e}{2}.
	$
\end{proof}

Thus, if for every $\delta > 0$, there were a $G$-connected $\coc$-finite Borel equivalence relation $F$ and a sign $\sigma \in \set{+,-}$ with $(\Meanrf{G : F}) \cap I_\de^\sigma = \0$, then \cref{tiling_with_average-approximating_sets} applied to $\Gmod{F}$, $\mumod{f}$, and $\meanrFf$ would yield a $G$-connected $\coc$-finite Borel equivalence relation $F' \supseteq F$ with $\meanrf{F'} \notin I_{2\de}^\sigma$ on a $\mu$-co-$\de$ set. Taking $\de$ sufficiently small, \cref{cutting_one_side} would imply that $\meanrf{F'} \in I_\e$ on a $\mu$-co-$\e$ set, so $F'$ is good for Main Lemma. Thus, without loss of generality, we assume the following.

\begin{hypothesis}\label{hypothesis:not_one-sided}
	There is a $\de > 0$ such that for any $G$-connected $\coc$-finite Borel equivalence relation $F$, $\Meanrf{G : F}$ intersects both $I^+_\delta$ and $I^-_\delta$.
\end{hypothesis}

We will exploit this hypothesis and the non-$\mu$-hyperfiniteness of $G$ via packed tilings in \cref{subsec:proof:packed_domains_large,subsec:proof:iterative_tiling}.

\subsection{Domains of packed tilings}\label{subsec:proof:packed_domains_large}

Let $F$ be a $G$-connected $\coc$-finite equivalence relation and let $\pi_F : X \to \Xmod{F}$ denote the quotient map.

For $\la > 0$, we call $U \in \FinrGmod{F}$ \emph{$\lambda$-central} (resp. \emph{$\la$-positive}, \emph{$\la$-negative}) if $\mean{\rmod{F}}{U} (\meanrFf) \in I_\la$ (resp., $I^+_\lambda$, $I^-_\lambda$). Note that $\mean{\rmod{F}}{U} (\meanrFf) = \meanr{\pi_F^{-1} U} f$, so the notions of $\lambda$-central/negative/positive coincide for $U$ and its pullback $\pi_F^{-1} (U)$. Also, $\cocmax(\pi_F^{-1}(U)) \ge \cocmaxF(U)$ by \cref{coc-quotient}.

For $\lambda, L > 0$, let $\SCmod{F}(\la,L)$ denote the collection of all $\lambda$-central $U \in \FinrGmod{F}$ with $\cocmaxF(U) > L$. We omit writing the subscript $/F$ if $F = \Id_X$.

\begin{lemma}[Finitizing visibility]\label{visibility-finitizing_domain}
	For any $\lambda < \delta$ as in \cref{hypothesis:not_one-sided}, $p \defeq \frac{\la}{\Linf{f}}$, and any tiling $\PC \subseteq \FinrGmod{F}$ that is $p$-packed within $\SCmod{F}(\la,L)$, $(\Gmod{F})_{-\dom(\PC)}$ has finite $\rmod{F}$-visibility.
\end{lemma}
\begin{proof}
    Taking the quotient by $F$ preserves all of our hypotheses; in particular, $\Gmod{F}$ is $\mumod{F}$-ergodic and \cref{hypothesis:not_one-sided} is satisfied by $\Gmod{F}$, so $\Meanrf{\Gmod{F} : E(\PC)}$ is constant and intersects both $I^+_\delta$ and $I^-_\delta$. Also, $\Linf{\meanrFf} \le \Linf{f}$, so assume without loss of generality that $F = \Id_X$.

	Towards the contrapositive, we suppose that $G' \defeq G_{-\dom(\PC)}$ does not have finite $\coc$-visibility and aim to show that $\PC$ is not $p$-packed within $\SC(\la,L)$. Let $x \in X \setminus \dom(\PC)$ be such that $C \defeq (\vis^\coc_{G'})^x$ is $\coc$-infinite. We use below that for any visible neighborhood $V \subseteq C$ of $x$, $\cocmax(V) = \coc_x(V)$.
	
	By \cref{packed_properties}\labelcref{packed_properties:maximal}, $\PC$ is maximal within $\SC(\lambda,L)$, so there is no $V \in \SC(\la,L)$ that is entirely contained in $C$. This and the convexity of $\MeanrfGrest{C}(x)$ (\cref{averages_on_cones_closed_interval}) imply that $x$ cannot have both $\la$-positive and $\la$-negative arbitrarily $\coc$-large visible neighborhoods $V \subseteq C$. Thus, all $\coc$-large enough visible neighborhoods of $x$ within $C$ must have the same $\la$-sign. For concreteness, suppose that they are all $\la$-positive.
		
	Because $\Meanrf{G : E(\PC)}$ intersects $I_\delta^-$ and $\lambda < \delta$, there is a $\la$-negative $E(\PC)$-invariant $U \in \FinrG$ containing $x$ such that $\Xmod[U]{E(\PC)}$ is finite, $\cocmax(U) \ge L$, and $2 \Linf{f} \frac{\coc(x)}{\coc(U)} < \la$. The latter condition ensures, by the intermediate value property (\cref{intermediate_value_property}), that there is $W \in \FinG[C]$ disjoint from $U$ but $G$-adjacent to $U$ (i.e. $\EdgsBtw{G}{W,U} \ne \0$) such that 
	\[
		0 \le \meanrf{U \cup W} < \la.
	\]
	Because $W \subseteq C$, $\maxr W \le \coc(x) \le \maxr U$, so $\cocmax(U \cup W) \ge \cocmax(U) \ge L$ and hence, $U \cup W \in \SC( \la, L)$. Lastly, $U \cup W$ contains only finitely-many tiles from $\PC$, so it is enough to prove that $U \cup W$ is a $p$-pack over $\PC$, i.e. $\coc(W) \ge p \cdot \coc(U)$. To this end, by \cref{convexity_of_average}\labelcref{item:convexity_of_average:disjU_of_averages},
	\begin{align*}
		\coc(W) \cdot \meanrf{W}
		&= 
		\big(\coc(U) + \coc(W)\big) \cdot \meanrf{U \cup W} - \coc(U) \cdot \meanrf{U}
		\\
		&\ge
		\big(\coc(U) + \coc(W)\big) \cdot 0 - \coc(U) \cdot (- \la) = \la \cdot \coc(U).
	\end{align*}
	
	\noindent In particular, $\meanrf{W} > 0$, so 
	$
	\coc(W)
	\ge 
	\frac{\la}{\meanrf{W}} \cdot \coc(U)
	\ge 
	\frac{\la}{\Linf{f}} \cdot \coc(U)
	\ge
	p \cdot \coc(U).
	$
\end{proof}

\begin{cor}\label{hyperfinitizing_domain}
For any $\lambda < \delta$ as in \cref{hypothesis:not_one-sided}, $p \defeq \frac{\la}{\Linf{f}}$, and any tiling $\PC \subseteq \FinrGmod{F}$ that is $p$-packed within $\SCmod{F}(\la,L)$, $\pi_F^{-1}(\dom(\PC))$ is a hyperfinitizing Borel vertex-cut for $G$; in particular, $\mu\big(\dom(\PC)\big) \ge \hvp(G)$.
\end{cor}
\begin{proof}
By \cref{visibility-finitizing_domain} and \cref{finite_visibility=>hyperfinite}, $D \defeq \dom(\PC)$ is a hyperfinitizing vertex-cut for $\Gmod{F}$. The hyperfiniteness of $(\Gmod{F})_{-D}$ implies that of $G_{-\pi_F^{-1}(D)}$ by \cite{JKL}*{1.3(ii)} because $\pi_F$ is a Borel reduction $E_G \to \Emod[(E_G)]{F}$.
\end{proof}

\subsection{Iterative tiling and its limsup}\label{subsec:proof:iterative_tiling}

To construct a $\coc$-finite equivalence relation $F$ that is good for Main Lemma, we first obtain a coherent sequence of approximately saturated tilings $(\PC_n)$ that contain $G$-connected sets of higher and higher \cocratio and become more and more central and more and more packed. We then show that putting together enough finite-many of these $\PC_n$ yields a desired $F$.

\begin{claimlemma}\label{finitely_based}
    For any $\lambda, L > 0$ and any $G$-connected $\coc$-finite Borel equivalence relation $F$, the collection $\SCmod{F}(\lambda, L)$ has finitely based quotients.
\end{claimlemma}
\begin{proof}
It is enough to show that $\SCmod{F}(\lambda, L)$ is finitely based because for any Borel tiling $\PC \subseteq \SCmod{F}(\la,L)$, the finite basedness of the quotient of $\SCmod{F}(\la,L)$ by $E(\PC)$ is implied by the finite basedness of $\SCmod{E(\PC)}(\la,L)$. Working mod $F$, we assume without loss of generality that $F = \Id_X$. It remains to show that $\SC(\lambda,L)$ is finitely based.

Let $A \in \SC(\lambda,L)$ and let $B \subseteq A$ be finite. Fixing $x \in \Maxr A$, we have that $\coc_x(A) = \cocmax(A) > L$, hence letting $A' \subseteq A$ be a finite $G$-connected subset containing $B \cup \set{x}$ with $\coc_x(A \setminus A')$ small enough, we have that $\cocmax(A') > L$. Moreover, because $\meanrf{A} \in I_\lambda$ and $I_\lambda$ is open, making $\coc_x(A \setminus A')$ possibly even smaller ensures, by \cref{convexity_of_average}\labelcref{item:convexity_of_average:increment_bound}, that $\meanrf{A'} \in I_\lambda$ hence $A' \in \SC(\lambda,L)$.
\end{proof}

Let $\de \in (0,1)$ be as in \cref{hypothesis:not_one-sided} and for each $n \ge 0$, put 
\begin{align*}
	\la_n &\defeq 3^{-n} \cdot \de \cdot \e < \e
	\\
	L_n &\defeq 4^n
	\\
	p_n &\defeq \frac{\la_{n+2}}{\Linf{f} + \la_{n+1}}.
\end{align*}

\begin{remark}
	All we need below is that $\lim_n \la_n L_n = \w$ (this warrants \labelcref{eq:L_n_is_8_and_average_change_bnd} below and hence, \cref{all_B_have_the_same_sign}), and that $(\la_n)$ decays exponentially to $0$ (used in \cref{average_at_the_border}). The choice of $(p_n)$ is made to yield a contradiction in the proof of \cref{packing_over}.
\end{remark}

We recursively define a coherent sequence $(\PC_n)_{n \ge 0}$ of Borel tilings contained in $\FinrG$; in particular, the equivalence relations $F_n \defeq \bigcup_{k<n} E(\PC_k)$ are increasing, where $F_0 \defeq \Id_X$. For $n \ge 0$, suppose that $F_n$ is defined. By \cref{finitely_based}, \cref{existence-of-packed-and-saturated} applies to $\SCmod{F_n}(\lambda_n, L_n)$ with $p \defeq p_n$ and yields a Borel tiling $\QC \subseteq \^{\SCmod{F_n}(\lambda_n, L_n)}$ that is $p_n$-packed within $\SCmod{F_n}(\lambda_n, L_n)$ and approximately saturated within $\SCmod{F_n}'(\lambda_n,L_n) \defeq \SCmod{F_n}(\lambda_n, L_n) \cap \FinGmod{F_n}$ modulo $\rmod{F_n}$-deficient. Putting $\PC_n \defeq \set{\pi_{F_n}^{-1}(Q) : Q \in \QC}$ finishes the construction. 

By \cref{deficient=>no_prob_meas}, if an $\Emod[(E_G)]{F_n}$-invariant Borel $A_n \subseteq \Xmod{F_n}$ is $\rmod{F_n}$-deficient, then it is $\mumod{F_n}$-null, so $\pi_{F_n}^{-1}(A_n)$ is $\mu$-null. Thus, discarding countably-many $E_G$-invariant $\mu$-null sets from $X$, we have that for each $n \ge 0$, $\submod{(\PC_n)}{F_n}$ is $p_n$-packed within $\SCmod{F_n}(\lambda_n, L_n)$ and approximately saturated within $\SCmod{F_n}'(\lambda_n,L_n)$. In particular, $\submod{(\PC_n)}{F_n} \subseteq \^{\SCmod{F_n}(\lambda_n, L_n)}$, which implies by \cref{incU_finite=>increasing_ratio} and \cref{convexity_of_average}\labelcref{item:convexity_of_average:increment_bound} that for each $n \in \N$ and $P \in \PC_n$, 
\begin{equation}\label{eq:PC_n-conditions}
\cocmax (P) > L_n \;\text{ and }\; \meanrf{P} \in \^{I_{\lambda_n}} \defeq [-\lambda_n, \lambda_n].
\end{equation}

Let $D_n \defeq \dom(\PC_n)$ and $D_\w \defeq \limsup_n D_n \defeq \set{x \in X : x \in D_n \text{ for infinitely many $n \in \N$}} = \bigcap_N \bigcup_{n \ge N} D_n$. By \cref{hyperfinitizing_domain}, $\mu(D_n) \ge \hvp(G)$, and $\hvp(G) > 0$ by our assumption that $G$ is not $\mu$-hyperfinite. Thus, $D_\w$ has positive measure by the downward continuity of $\mu$.

In the next subsection, we show that $D_\w$ is actually conull. Granted this, the proof of \cref{ptwise_ergodic_for_graphs} is completed as follows. Fix $n \in \N$ large enough so that $D \defeq \bigcup_{k \le n} D_k$ is $\mu$-co-$\e$. Then $D$ is $F_n$-invariant and for each $x \in D$, $[x]_{F_n} \in \PC_k$ for some $k \le n$, so $[x]_{F_n}$ is $\lambda_k$-central, and hence $\e$-central. Thus, $F_n$ is good for Main Lemma.

\subsection{The conullness of $D_\w$ amounts to a packing condition}

Suppose towards a contradiction that $X \setminus D_\w$ has positive measure, hence so does the inner-boundary $\bdrin{G}(D_\w)$ by the $E_G$-quasi-invariance of $\mu$.

Putting $F_\w \defeq \bigcup_n F_n$, we observe that $D_\w$ is $F_\w$-invariant and each $F_\w$-class $U$ outside of $D_\w$ is actually an $F_n$-class for large enough $n$; in particular, $U$ is $\coc$-finite. On the other hand, each $F_\w$-class within $D_\w$ is $\coc$-infinite by \cref{increasing_ratio=>infinite} because for each $n \in \N$ and $P \in \PC_n$, we have $\cocmax(P) > L_n$ by \labelcref{eq:PC_n-conditions} and $\lim_n L_n = \infty$. Hence, $\DirEdgsBtw{G}{D_\w, X \setminus D_\w} = \bigincU_{N = 1}^\w \vec{H}_N$, where 
\[
	\vec{H}_N \defeq \set{(x,y) \in \DirEdgsBtw{G}{D_\w, X \setminus D_\w} : \coc\big([x]_{F_N}\big) \ge \coc\big([y]_{F_\w}\big) \text{ and } \forall n \ge N \, [y]_n = [y]_{\infty}}.
\]
The $E_G$-quasi-invariance of $\mu$ implies that for all large enough $N \ge 1$, the inner boundary $\bdrin{\vec{H}_N}(D_\w)$ has positive measure and moreover, for all $n \ge N$,
\begin{equation}\label{eq:L_n_is_8_and_average_change_bnd}
L_n \ge 8 \text{ and } \frac{2 \cdot \Linf{f}}{L_n} < \frac{1}{3} \la_n.
\end{equation}
Furthermore, we can choose such an $N$ so that $D_N' \defeq \bdrin{\vec{H}_N} D_N \subseteq D_\infty$ has positive measure.

For each $n \ge N$, let $\PC_n' \defeq \set{P \in \PC_n : P \cap D_N' \ne \0}$. Note that $(\PC_n')_{n \ge N}$ does not stabilize anywhere on $[D_N']_{F_\infty}$ because $D_N' \subseteq D_\infty$. Thus, it remains to show that $(\PC_n')_{n \ge N}$ is a $p$-packing sequence for some $p > 0$ because then, \cref{packing_stabilizes} implies that $[D_N']_{F_\infty}$ is $\coc$-deficient, hence null (by \cref{deficient=>no_prob_meas}), contradicting $\mu(D_N') > 0$.

\subsection{Verifying the packing condition}

In this subsection we show that $(\PC_n')_{n \ge N}$ is $\frac{2}{13}$-packing, thus completing the proof of Main Lemma. Because each tile in $\PC' \defeq \bigvee_n \PC_n'$ is equal to an $F_\infty$-class that meets $D_N'$, it is enough to fix one such $F_\infty$-class $C$ and show that $(\PC_n' \rest{C})_{n \ge N}$ is $\frac{2}{13}$-packing. Fix such a $C$.

\begin{claimlemma}\label{all_B_have_the_same_sign}
    There is a sign $\sigma_C \in \set{+, -}$ such that $\meanrf{[y]_{F_\w}} \in \R^{\sigma_C}$ for every $y \in \bdrout{\vec{H}_N} (D_N' \cap C)$.
\end{claimlemma}
\begin{pf}
	Suppose towards a contradiction that there are (possibly equal) points $y_-,y_+ \in \bdrout{\vec{H}_N} (D_N' \cap C)$ such that $\meanrf{[y_-]_{F_\w}} \le 0 \le \meanrf{[y_+]_{F_\w}}$. Let $x_-,x_+ \in D_N' \cap C$ be (possibly equal, even when $y_- \ne y_+$) points such that $(x_-, y_-), (x_+, y_+) \in \vec{H}_N$. Fix $n > N$ large enough so that $P \defeq [x_-]_{F_{n+1}} = [x_+]_{F_{n+1}}$ and such that $x_+ \in D_n$, so $P \in \PC_n$. 
	
	Put $Q \defeq \Xmod[P]{F_n}$, and for each $\sigma \in \set{-, +}$, put $x_\sigma' \defeq [x_\sigma]_{F_n}$ and $y_\sigma' \defeq [y_\sigma]_{F_n}$. The definition of $\vec{H}_N$ implies that $y_\sigma' = [y_\sigma]_{F_\infty}$ and
	\begin{equation}\label{eq:x>y}
	\rmod{F_n}(x_-') \ge \rmod{F_n}(y_-') \;\text{and}\; \rmod{F_n}(x_+') \ge \rmod{F_n}(y_+').
	\end{equation}
    Recall that $\QC \defeq \submod{(\PC_n)}{F_n}$ admits a saturating $\SCmod{F_n}'(\lambda_n,L_n)$-approximation $(\QC_m)$. We now switch to working mod $F_n$, so we assume without loss of generality that $F_n = \Id_X$ and drop $F_n$ from the notation; in particular, we write $f$ instead of $\meanrf{F_n}$.
    
    We show that if $\meanrf{Q} \le 0$ then $y_+'$ must have been in $\dom(\QC)$, and if $\meanrf{Q} \ge 0$ then $y_-'$ must have been in $\dom(\QC)$. Without loss of generality, we suppose that $\meanrf{Q} \le 0$ (the argument is symmetric). For each $m \in \N$, putting $Q_m \defeq [x_+']_{E(\QC_m)}$, observe that $x_-' \in Q = \bigincU_m Q_m$ and that $\lim_m \meanrf{Q_m} = \meanrf{Q}$ by \cref{convexity_of_average}\labelcref{item:convexity_of_average:increment_bound} because $\lim_m \coc(Q \setminus Q_m) = 0$. Thus, for all large enough $k \in \N$ the following holds and we fix such a $k$:
    
	\begin{enumerate}[(i), widest=iii]
	    \item\label{all_B_have_the_same_sign:is_a_tile} $Q_k \in \bigcup_{m \in \N} \QC_m \subseteq \SC'(\lambda_n,L_n)$;
	
	    \item\label{all_B_have_the_same_sign:both_points_in_Qk} $x_-' \in Q_k$;
	    
	    \item\label{all_B_have_the_same_sign:averages_approx-equal} $\meanrf{Q_k} \approx_{\frac{\lambda_n}{3}} \meanrf{Q}$.
	\end{enumerate}
	
	Taking $Q_k' \defeq Q_k \cup \set{y_+'}$, \labelcref{all_B_have_the_same_sign:both_points_in_Qk} and \labelcref{eq:x>y} imply that $\maxr Q_k \ge \maxr \set{x_-', x_+'} \ge \maxr \set{y_-', y_+'}$, so $\maxr Q_k' = \maxr Q_k$ and hence $\cocmax (Q_k') = \frac{\coc(Q_k) + \coc(y_+')}{\maxr (Q_k)} > \cocmax (Q_k) > L_n$, where the last inequality is due to \labelcref{all_B_have_the_same_sign:is_a_tile}. Furthermore, because $f(y_+') \ge 0$ and $\meanrf{Q_k} \in I_{\lambda_n}$ (by \labelcref{all_B_have_the_same_sign:is_a_tile}), we have $\meanrf{Q_k'} \ge \meanrf{Q_k} > -\lambda_n$. On the other hand, by \cref{convexity_of_average}\labelcref{item:convexity_of_average:increment_bound}, the fact that $\cocmax(Q_k) > L_n$, and \labelcref{eq:L_n_is_8_and_average_change_bnd},
	\[
	\meanrf{Q_k'} - \meanrf{Q_k} \le \frac{2 \cdot \Linf{f} \cdot \coc(y_+')}{\coc(Q_k)} \le \frac{2 \cdot \Linf{f} \cdot \coc(y_+')}{L_n \cdot \maxr Q_k} \le \frac{2 \cdot \Linf{f}}{L_n} < \frac{\lambda_n}{3},
	\]
	so $\meanrf{Q_k'} < \meanrf{Q_k} + \frac{\lambda_n}{3} \le \frac{2\lambda_n}{3}$ by \labelcref{all_B_have_the_same_sign:averages_approx-equal}. Thus, $Q_k'$ is $\lambda_n$-central, and hence $Q_k' \in \SC'(\lambda_n,L_n)$, contradicting that $(\QC_m)$ is a saturating $\SC'(\lambda_n,L_n)$-approximation for $\QC$.
\end{pf}

Suppose, for the sake of concreteness, that $\sigma_C = -$ as the proof for $\sigma_C = +$ is analogous: instead of $\meanrf{P} \le - 2 \la_{n+1}$ below, we would have $\meanrf{P} \ge 2 \la_{n+1}$, and the proof of \cref{packing_over} would use the reverse inequalities and opposite signs.

\begin{claimlemma}\label{tiles_average_at_the_border}
	$\meanrf{P} \le - \frac{2}{3} \la_n = - 2 \la_{n+1}$ for each $n \ge N$ and each $P \in \PC_n' \rest{C}$.
\end{claimlemma}
\begin{pf}
    Suppose towards a contradiction that $\meanrf{P} > - \frac{2}{3} \la_n$, so $\mean{\rmod{F_n}}{Q} (\meanrf{F_n}) = \meanrf{P} > - \frac{2}{3} \la_n$, where $Q \defeq \Xmod[P]{F_n}$. Like in the proof of \cref{all_B_have_the_same_sign}, let $x'$ denote the $F_n$-class of some $x \in P \cap D_N'$ and let $y'$ denote the $F_n$-class of some $y \in X \setminus D_\infty$ such that $(x,y) \in \vec{H}_N$. By the definition of $ \vec{H}_N$, $y' = [y]_{F_\infty}$ and $\coc(x') \ge \coc(y')$. Lastly, our assumption of $\sigma_C = -$ implies that $\meanrf{y'} < 0$.
    
    Recall that $\QC \defeq \submod{(\PC_n)}{F_n}$ admits a saturating $\SCmod{F_n}'(\lambda_n,L_n)$-approximation $(\QC_m)$. We now switch to working mod $F_n$, so we assume without loss of generality that $F_n = \Id_X$ and drop $F_n$ from the notation.
    
    For each $m \in \N$, putting $Q_m \defeq [x']_{E(\QC_m)}$, observe that $\lim_m \meanrf{Q_m} = \meanrf{Q}$ by \cref{convexity_of_average}\labelcref{item:convexity_of_average:increment_bound} because $\lim_m \coc(Q \setminus Q_m) = 0$. Thus, for all large enough $k \in \N$ the following holds and we fix such a $k$:
    
	\begin{enumerate}[(i), widest=ii, itemsep=4pt]
	    \item\label{tiles_average_at_the_border:is_a_tile} $Q_k \in \bigcup_{m \in \N} \QC_m \subseteq \SC'(\lambda_n,L_n)$;
	    
	    \item\label{tiles_average_at_the_border:averages_approx-equal} $\meanrf{Q_k} > - \frac{2}{3} \la_n$.
	\end{enumerate}
	
    We take $Q_k' \defeq Q_k \cup \set{y'}$, so $\maxr Q_k \ge \coc(x') \ge \coc(y')$, and hence $\maxr Q_k' = \maxr Q_k$, thus $\cocmax (Q_k') > \cocmax (Q_k) > L_n$, where the last inequality is due to \labelcref{tiles_average_at_the_border:is_a_tile}. Because $\meanrf{y'} < 0$, we have $\meanrf{Q_k'} < \meanrf{Q_k} < \lambda_n$. On the other hand, by \cref{convexity_of_average}\labelcref{item:convexity_of_average:increment_bound}, the fact that $\cocmax(Q_k) > L_n$, and \labelcref{eq:L_n_is_8_and_average_change_bnd},
	\[
	\meanrf{Q_k} - \meanrf{Q_k'}
	\le
	2 \Linf{f} \frac{\coc(y')}{\coc(Q_k)} 
	\le 
	2 \Linf{f} \frac{\coc(y')}{L_n \cdot \maxr Q_k} 
	\le 
	\frac{2 \Linf{f}}{L_n} 
	< 
	\frac{1}{3} \la_n,
	\]
	so $\meanrf{Q_k'} > \meanrf{Q_k} - \frac{1}{3} \la_n > - \frac{2}{3} \la_n - \frac{1}{3} \la_n = -\lambda_n$ by \labelcref{tiles_average_at_the_border:averages_approx-equal}. Thus, $Q_k'$ is also $\lambda_n$-central, so $Q_k' \in \SC'(\lambda_n,L_n)$, contradicting that $(\QC_m)$ is saturating $\SC'(\lambda_n,L_n)$-approximation for $\QC$.
\end{pf}

\begin{claimlemma}\label{average_at_the_border}
	$\meanrf{[x]_{F_{n+1}}} \le - 2 \la_{n+1}$ for each $n \ge N$ and $x \in D_N' \cap C$.
\end{claimlemma}
\begin{pf}
    Because $x \in D_N' \cap C \subseteq D_N$, $P \defeq [x]_{F_{n+1}} \in \PC_k' \rest{C}$ for some $k \in \N$ such that $N \le k \le n$. By \cref{tiles_average_at_the_border}, $\meanrf{P} \le -2 \lambda_{k+1} \le - 2 \lambda_{n+1}$.
\end{pf}

\begin{claimlemma}\label{packing_over}
	For each $n \ge N$, each $P \in \PC_{n+1}' \rest{C}$ is a $\frac{2}{13}$-pack over $\bigvee_{m \le n} \PC_n'$.
\end{claimlemma}
\begin{pf}
    If $U \defeq \dom(\bigvee_{m \le n} \PC_n') \cap P = \0$, then $P$ is trivially a $\frac{2}{13}$-pack over $\bigvee_{m \le n} \PC_n'$, so suppose $U \ne \0$. Let $m \le n$ be the largest number with $\dom(\PC_m') \cap P \ne 0$. Thus, $U$ is $F_{m+1}$-invariant (hence, a disjoint union of $F_{m+1}$-classes), so \cref{average_at_the_border} implies
    \begin{equation}\label{U_is_way_below_0}
        \meanrf{U} \le -2 \lambda_{m+1}
    \end{equation}
    Next, note that $V \defeq (\bigcup_{M = m}^n D_M \cap P) \setminus U$ is a disjoint union of tiles from $\bigvee_{M=m}^n \PC_M$, so
    \begin{equation}\label{V_is_not_too_above_0}
        \meanrf{V} \le \lambda_m.
    \end{equation}
    Lastly, put $W \defeq P \setminus (U \cup V)$. Note that because $P \in \PC_{n+1}$, we have $\Xmod[P]{F_{n+1}} \in \SCmod{F_{n+1}}(\lambda_{n+1}, L_{n+1})$, and hence also $\Xmod[P]{F_m} \in \SCmod{F_m}(\lambda_m, L_m)$, so the $p_m$-packedness of $\submod{(\PC_m)}{F_m}$ within $\SCmod{F_m}(\lambda_m, L_m)$ implies that $P$ is not a $p_m$-pack over $\PC_m$. Thus, because $D_m \cap P \subseteq U \cup V$, we have
    \begin{equation}\label{eq:W_is_small}
    \coc(W) 
    <
    p_m \cdot (\coc(U) + \coc(V)) 
    = 
    \frac{\lambda_{m+2}}{\Linf{f} + \lambda_{m+1}} \cdot (\coc(U) + \coc(V)).
    \end{equation}
    Then
	\begin{align*}
	- \la_{n+1} \cdot \big(\coc(U) + \coc(V) + \coc(W)\big)
	&=
	- \la_{n+1} \cdot \coc(P)
	\\
	\eqcomment{by \labelcref{eq:PC_n-conditions}}
	&\le 
	\meanrf{P} \cdot \coc(P)
	\\
	&=
	\meanrf{U} \cdot \coc(U) + \meanrf{V} \cdot \coc(V) + \meanrf{W} \cdot \coc(W)
	\\
	\eqcomment{by \cref{U_is_way_below_0,V_is_not_too_above_0}}
	&\le 
	- 2 \la_{m+1} \cdot \coc(U) + \la_m \cdot \coc(V) + \Linf{f} \cdot \coc(W),
	\end{align*}
	so
	\begin{align*}
		(2\lambda_{m+1} - \la_{n+1}) \cdot \coc(U) 
		&\le 
		(\la_m + \la_{n+1}) \cdot \coc(V) + (\Linf{f} + \la_{n+1}) \cdot \coc(W)
		\\
		\eqcomment{by \labelcref{eq:W_is_small}}
		&\le 
		(\la_m + \la_{n+1}) \cdot \coc(V) + \frac{\Linf{f} + \la_{n+1}}{\Linf{f} + \la_{m+1}} \cdot \la_{m+2} \cdot (\coc(U) + \coc(V))
		\\
		&\le
		(\la_m + \la_{n+1}) \cdot \coc(V) + \la_{m+2} \cdot (\coc(U) + \coc(V))
	\end{align*}
	Thus, $(2\lambda_{m+1} - \la_{n+1} - \lambda_{m+2}) \cdot \coc(U) \le (\la_m + \la_{n+1} + \la_{m+2}) \cdot \coc(V)$. Because $\lambda_{n+1} \le \lambda_{m+1}$, $\la_m = 9 \la_{m+2}$, and $\la_{m+1} = 3 \la_{m+2}$, we finally get:
	\[
	\frac{\coc(V)}{\coc(U)} \ge \frac{2\lambda_{m+1} - \la_{n+1} - \lambda_{m+2}}{\la_m + \la_{n+1} + \la_{m+2}} \ge \frac{\lambda_{m+1} - \lambda_{m+2}}{\la_m + \la_{m+1} + \la_{m+2}} = \frac{(3 - 1)}{(9 + 3 + 1)} = \frac{2}{13}.\qedhere
	\]
\end{pf}

\noindent \cref{packing_over} implies that $(\PC_n' \rest{C})_{n \ge N}$ is a $\frac{2}{13}$-packing sequence (\cref{defn:packed}), thus completing the proof of Main Lemma \namedthmlabelref{core_lemma}.\qed(\namedthmlabelref{core_lemma})

\end{document}